\documentclass[11pt,a4]{article} 
\usepackage{amsmath, amssymb,amsfonts}  
\usepackage{amsthm}
\usepackage{color,fancybox,ascmac}
\usepackage{mathtools}
\usepackage{mathrsfs}

\usepackage{charter}

\numberwithin{equation}{subsection}
\theoremstyle{plain} %
\newtheorem{theorem}{\indent\sc\bf Theorem}[subsection]
\newtheorem{lemma}[theorem]{\indent\sc\bf Lemma}
\newtheorem{corollary}[theorem]{\indent\sc\bf Corollary}
\newtheorem{proposition}[theorem]{\indent\sc\bf Proposition}

\theoremstyle{definition} %

\newcommand{\bA}{\mathbb{A}}
\newcommand{\bC}{\mathbb{C}}
\newcommand{\bF}{\mathbb{F}}

\newcommand{\bH}{\mathbb{H}}

\newcommand{\bQ}{\mathbb{Q}}
\newcommand{\bR}{\mathbb{R}}

\newcommand{\bZ}{\mathbb{Z}}

\newcommand{\Q}{\mathbb{Q}}
\newcommand{\Z}{\mathbb{Z}}

\newcommand{\SL}{\mathrm{SL}}
\newcommand{\GL}{\mathrm{GL}}
\newcommand{\Sp}{\mathrm{Sp}}

\newcommand{\grH}{\mathfrak{H}}
\newcommand{\grp}{\mathfrak{p}}

\newcommand{\e}{\mathbf{e}}

\newcommand{\tp}{{}^t\!}
\newcommand{\tr}{\mathrm{tr}}
\newcommand{\Tr}{\mathrm{Tr}}

\renewcommand{\Im}{{\rm Im}}
\newcommand{\cA}{\mathcal{A}}

\newcommand{\cH}{\mathcal{H}}

\newcommand{\cJ}{\mathcal{J}}

\newcommand{\cL}{\mathcal{L}}

\newcommand{\cO}{\mathcal{O}}

\newcommand{\grP}{\mathfrak{P}}

\newcommand{\cS}{\mathcal{S}}
\newcommand{\cT}{\mathcal{T}}

\newcommand{\cX}{\mathcal{X}}

\newcommand{\sS}{\mathscr{S}}

\newcommand{\inv}{^{-1}}
\newcommand{\cross}{^{\times}}
\newcommand{\ord}{\mathrm{ord}}
\newcommand{\MM}{\mathrm{M}}
\newcommand{\dd}{\mathbf{d}}

\newcommand{\vecm}[1]{\left( \begin{array}{c} #1 \end{array} \right)}
\newcommand{\Nmat}[2]{\left( \begin{array}{cc} #1 \\ #2 \end{array} \right)}

\newcommand{\ol}[1]{\overline{#1}}

\newcommand{\suma}[1]{\sum_{#1} \,}
\newcommand{\sumd}[1]{\sum_{#1} {}^{'}\,}
\newcommand{\sumt}[1]{\sum_{#1} {}^2\,}

\begin{document}

\begin{center}
{\huge On the Arakawa lifting\\[0.5cm]
Part I: Eichler commutation relations}\\[1cm]
{\Large By Atsushi Murase and Hiro-aki Narita}
\end{center}

\bigskip

%
%
%
%


\abstract{
We investigate the theta correspondence of cusp forms 
for the dual pair $(O^*(4),\Sp(1,1))$
originally introduced by Tsuneo Arakawa in the non-adelic setting. 
We call this Arakawa lifting. 
In this paper, reformulating the theta correspondence 
in the adelic setting, 
we provide commutation relations of Hecke operators 
satisfied by Arakawa lifting at all non-Archimedean places, 
which is referred to as Eichler commutation relations 
for classical modular forms. 
Their Archimedean analogue is also given in terms of reproducing kernels.
}


 \section{Introduction}

This is the first part in the series of the papers, whose aim is
to investigate the theta correspondence of cusp forms for
the dual  pair $(O^*(4),\Sp(1,1))$, which we call Arakawa lifting.

The study of this paper has an origin in Eichler (\cite{Ei1}) 
on the action of Hecke operators on theta series. 
Examining the action of Hecke operators is a fundamental tool 
to study arithmetic properties of automorphic forms. 
On the other hand, the notion of theta series fits into 
the theta correspondence or 
Howe correspondence of automorphic representations of 
reductive groups in the representation theoretic context. 
In the spirit of the Langlands philosophy, 
the Hecke actions via the theta correspondence are 
preferred to understand in terms of L-morphisms 
between relevant L-groups, 
which are nothing but the notion of the Langlands functoriality. 
This point of view is taken up  by Rallis (\cite{Ra}) for example.
Our aim is to know the Hecke action on Arakawa lifting explicitly. 
We can refer to this as ``Eichler's commutation relations'' 
in the sense of the classical modular forms 
and also to ``functoriality'', following Rallis (\cite{Ra}). 
This lift is understood in terms of an inner form version 
of the endoscopic lift for the symplectic group of degree two 
(cf. \cite{Ro}). 
Different from the case of similitude groups 
 $GO^*(4)$ and $\mathrm{GSp}(1,1)$  in our previous paper \cite{MN1}, 
our commutation relations in this paper deal 
with local Hecke algebras for isometry groups  $O^*(4)$ and $\Sp(1,1)$ 
 at all non-Archimedean places including ramified ones. 
As an additional novelty of our results 
we should remark that their Archimedean analogue is given 
in terms of the reproducing kernels of holomorphic discrete series of 
$O^*(4)(\bR)$
 and quaternoinic discrete series of $\Sp(1, 1)(\bR)$.

Let $\Sp(1,q)$ be the quaternion unitary group of a hermitian form of signature
$(1,q)$.
Tsuneo Arakawa, in his unpublished notes, constructed a ``theta lifting'' 
from the space of holomorphic cusp forms on $\SL_2$ to that of
cusp forms on $\Sp(1,q)$.
Note that $(\SL_2,\Sp(1,q))$ does not form a dual pair when $q>1$ 
and Arakawa obtained the lifting by restricting the theta lifting
for the dual  pair $(\SL_2,SO(4,4q))$ regarding $\Sp(1,q)$
as a subgroup of $SO(4,4q)$.
When $q=1$, Arakawa showed that the image of the lifting is in the space
of cusp forms on $\Sp(1,1)$ generating the quaternionic discrete series
representation at the Archimedean place.

The second named author, based on Arakawa's notes, reformulated 
Arakawa's definition of automorphic forms on $Sp(1,q)$ 
and showed that 
the image of the lifting is in the space
of cusp forms on $\Sp(1,q)$ generating the quaternionic discrete series
representation at the Archimedean place when $q\ge1$
(\cite{Na1},\cite{Na3}). 
The first and  second named authors studied
the arithmetic properties of the theta lifting
for  similitude groups $(GO^*(4),\mathrm{GSp}(1,1))$
(\cite{MN1},\cite{MN2},\cite{MN3}).
The second named author investigated 
a relation between the theta lifting
for $(GO^*(4),\mathrm{GSp}(1,1))$ (or $(GO^*(4),\mathrm{GSp}^*(2))$) 
and that for $(GO(2,2),\mathrm{GSp}(2))$
in terms of the Jacquet-Langlands-Shimizu correspondence (\cite{Na4},\cite{Na5}).
We here note that $\mathrm{GSp}(1,1)$ and $\mathrm{GSp}^*(2)$ are
 inner forms of 
$\mathrm{GSp}(2)$.

In this paper, we construct a theta lifting $\cL$
for the dual  pair
$(H,G)$, where $H=O^*(4)$ (respectively $G=\Sp(1,1)$) is the 
the quaternion unitary  group
of  an anti-hermitian (respectively hermitian) form of degree $2$,
and study its basic properties.
In particular, we show the functoriality of $\cL$ at any 
non-Archimedean places including
ramified ones by showing Eichler commutation relations.
We also give an Archimedian analogue of 
commutation relations,  by which we show that
$\cL$ sends holomorphic cuspidal automorphic forms on $H$ to
cuspidal automorphic forms on $G$ generating the quaternionic
discrete series representation at the Archimedean place.
In the forthcoming Part II and III
 of the series of the papers, 
we will investigate 
the image of $\cL$ and an inner product formula for $\cL$ based on
the results of this paper.

To be more precise, 
let $B$ be a definite quaternion algebra
over $\bQ$ with the main involution $x\mapsto \ol{x}$.
We denote by $B^{(m,n)}$ the set of $m\times n$ matrices 
whose components are in $B$.
We fix an identification  $B\otimes_{\bQ}\bC=\mathrm{M}_2(\bC)$.
Let 
\begin{align*}
 H_{\bQ}&=\left\{h\in\mathrm{M}_2(B)\mid \tp\ol{h}\Nmat{0&1}{-1&0}h
=\Nmat{0&1}{-1&0}
\right\},\\
 G_{\bQ}&=\left\{g\in\mathrm{M}_2(B)\mid \tp\ol{g}\Nmat{0&1}{1&0}g
=\Nmat{0&1}{1&0}
\right\}.
\end{align*}
Then $(H,G)$ forms a dual  pair.
We fix an integer $\kappa >8$ and let $\sigma_{\kappa}$ be
the symmetric tensor representation of 
$\mathrm{M}_2(\bC)$ on $V_{\kappa}$,
the space of homogeneous polynomials in $X$ and $Y$ of total degree $\kappa$.
Let $r$ be the Weil representation of $H_{\bA}\times G_{\bA}$ on 
$W=\sS(B_{\bA}^{(2,1)})\otimes_{\bC}V_{\kappa}$.
Here $\sS(B_{\bA}^{(2,1)})$ is the space of Schwartz-Bruhat functions
on $B_{\bA}^{(2,1)}$.
 With  suitable
test functions $\varphi_0,\varphi_0^*\in W$, we construct  theta kernels
\[
 \theta(h,g)=\sum_{X\in B^{(2,1)}}\,r(h,g)\varphi_0(X) 
\]
and
\[
 \theta^*(h,g)=\sum_{X\in B^{(2,1)}}\,r(h,g)\varphi_0^*(X). 
\]
Let $\cS_{\kappa}^H$ be the space of $V_{\kappa}$-valued holomorphic cusp forms
on $H_{\bQ}\backslash H_{\bA}$ (see Section \ref{sec:af-H}). We also let
 $\cS_{\kappa}^G$ be the space of $V_{\kappa}$-valued cusp forms
on $G_{\bQ}\backslash G_{\bA}$ generating the quaternionic discrete
series representation at the Archimedean place 
(see Section \ref{sec:af-G}).

We show that the theta lift 
\[
 \cL(f)(g):=\int_{H_{\bQ}\backslash H_{\bA}}\,\theta(h,g)f(h)dh
\qquad (g\in G_{\bA})
\]
of $f\in \cS_{\kappa}^H$
is in $\cS_{\kappa}^G$.
We also show that
the theta lift 
\[
 \cL^*(F)(h):=\int_{G_{\bQ}\backslash G_{\bA}}\,\theta^*(h,g)F(g)dg
\qquad (h\in H_{\bA})
\]
of $F\in \cS_{\kappa}^G$
is in $\cS_{\kappa}^H$
and that $\cL^*$ is the adjoint of $\cL$ with respect to the 
Petersson inner products on $\cS_{\kappa}^H$ and $\cS_{\kappa}^G$
respectively.
Note that the proof of these facts for Arakawa's lifting in \cite{Na2}
are based on the study of the lifts of
Poincar\'e series  on $\SL(2)$.
On the other hand,
 the method of the proof in this paper uses
an Archimedian analogue of
Eichler commutation relations:
\begin{equation}
 \label{eq:ecr-intro}
 c_{\kappa}^H\,\int_{H_{\infty}}\,\theta(hx,g)\omega_{\kappa}(x)dx
=  c_{\kappa}^G\,\int_{G_{\infty}}\,\Omega_{\kappa}(y)\theta(h,gy\inv)dy,
\end{equation}
where
\begin{align*}
 \omega_{\kappa}(h)&
=\sigma_{\kappa}\left(2\inv(-\sqrt{-1},1)h\vecm{\sqrt{-1}\\1}\right)\inv
\qquad(h\in H_{\infty}),\\
 \Omega_{\kappa}(g)& =(\Delta_g\ol{\Delta_g})\inv\,\sigma_{\kappa}(\Delta_g)\inv
\qquad(g\in G_{\infty},\Delta_g=2\inv (1,1)g\vecm{1\\1}).
\end{align*}
are spherical functions on $H_{\infty}=H(\bR)$ and $G_{\infty}=G(\bR)$ respectively,
 and $c_{\kappa}^H$ and $c_{\kappa}^G$ are constants defined in 
Sections \ref{subsec:af-H} and \ref{subsec:af-G} respectively.

We also show the correspondence of the local factors of the standard
$L$-functions
at all  non-Archimedean places for $\cL$.
Namely we prove that, if $f\in \cS_{\kappa}^H$ is a Hecke eigenform,
so is $\cL(f)$ and the standard $L$-function $L(s,\cL(f),\mathrm{std})$
of $\cL(f)$
coincides with the product of the Riemann zeta function and 
the standard $L$-function $L(s,f,\mathrm{std})$ of $f$.
Similar facts for $\cL^*$ are also proved.
The proof is based on the following 
Eichler commutation relations.
Let $p$ be a finite place of $\bQ$. Denote by $\cH(H_p,U_p)$ and 
$\cH(G_p,K_p)$ the Hecke algebras, where $U_p$ and $K_p$ are
maximal open compact subgroups of $H_p$ and $G_p$ defined in Section \ref{sec:groups} respectively.
Then
there exists an algebra homomorphism $\upsilon\colon \cH(G_p,K_p)\,
\to\,\cH(H_p,U_p)$ such that
\[
\int_{H_p}\,\theta(hx,g)\,\upsilon(\alpha)(x)dx=
 \int_{G_p}\,\theta(h,gy\inv)\,\alpha(y)dy
\]
holds for $\alpha\in\cH(G_p,K_p)$.

The paper is organized as follows.
In Section \ref{sec:notation}, we prepare the notation and
several basic facts on the  quaternion algebra used in the later sections.
In Section \ref{sec:groups}, we define two algebraic groups $H$ and $G$.
In Sections \ref{sec:af-H} and \ref{sec:af-G}, we introduce the space
of automorphic forms on $H$ and $G$ respectively, and define the
standard $L$-functions attached to Hecke eigenforms.
In Section \ref{sec:Arakawa-lifting}, after recalling the definition
of the Weil representation of $H_{\bA}\times G_{\bA}$,
we define the theta kernel $\theta\colon H_{\bA}\times G_{\bA}\to\mathrm{End}(V_{\kappa})$ with a suitable test function and the theta lifting $\cL$.
In Section \ref{sec:adjoint-AL},  we define the adjoint $\cL^*$ of 
$\cL$.
The heart of the paper is  the following two sections.
In Section \ref{sec:commutation-relations}, we show that $\cL$
defines a linear mapping from the space of $V_{\kappa}$-valued
holomorphic cusp forms on $H$ to the space of $V_{\kappa}$-valued
cusp forms on $G$ generating the quaternionic discrete series representation
at the Archimedean place by an analogue (\ref{eq:ecr-intro}) of Eichler 
commutation relations (Theorem \ref{th:commutation-relations}). The proof of (\ref{eq:ecr-intro}) is quite technical,
but we give full details of the proof since the method seems to be a  new
one.
In Section \ref{sec:functoriality}, we state the functoriality properties
of $\cL$ and $\cL^*$ as a correspondence of the standard $L$-functions
 (Theorem \ref{th:functoriality}).
Note that, in our previous work \cite{MN1}, the functoriality is
given as a correspondence of the spinor $L$-functions.
The proof of the functoriality is based on 
the Eichler commutation relations (Theorem \ref{th:Eichler-commutation-relations}).
Section \ref{sec:ecr-unramified} 
(respectively Section \ref{sec:ecr-ramified}) 
is devoted to a sketch of the proof
of the commutation relations
in  the unramified (respectively ramified) case.
Detailed accounts of Sections \ref{sec:ecr-unramified} and
\ref{sec:ecr-ramified} are given in Sections \ref{sec:ecr-unramified-details} and \ref{sec:ecr-ramified-details} respectively. 

 \section{Notation and Preliminaries}
\label{sec:notation}

 \subsection{}
\label{subsec:notation-1}

We denote by $\bZ$, $\bQ$, $\bR$ and $\bC$ the ring of rational integers and the fields of rational numbers, real numbers and complex numbers,
respectively. For a place $v$ of $\bQ$, $\bQ_v$ denotes the $v$-adic completion of $\bQ$. For a finite place $p$ of $\bQ$, let $|x|_p$ be the 
$p$-adic absolute value of $x\in\bQ_p\cross$.
We denote by $\bQ_{\bA}$ the adele ring of $\bQ$.

For a ring $R$, denote by $R^{(m,n)}$ the set of $m\times n$ matrices
whose components are in $R$.
We also write $\mathrm{M}_m(R)$ for $R^{(m,m)}$.
Let $X$ be an algebraic group $X$ over $\bQ$.
We denote by $X_v$  the group of $\bQ_v$-rational
points of $X$ for a place $v$ of $\bQ$, 
and by $X_{\bA}$  the adelization of $X$.

For $x\in\bC$, we put $\e(x)=\exp(2\pi\sqrt{-1}x)$.
Let $\psi$ be the additive character of $\bQ_{\bA}/\bQ$ such that
$\psi(x_{\infty})=\e(x_{\infty})$ for $x_{\infty}\in\bR$.
For a place $v$ of $\bQ$, 
the restriction of $\psi$ to $\bQ_v$ is denoted by $\psi_v$.

Denote by $X_{\mathrm{unr}}(\bQ_p\cross)$ the set of quasi-characters
of $\bQ_p\cross$ trivial on $\bZ_p\cross$.
For a prime number $p$, we write $\mathbb{F}_p$ for the finite field
$\bZ/p\bZ$.

We put
\[
 \delta(P)=
\begin{cases}
 1&\text{if $P$ holds},\\
 1&\text{if $P$ does not hold}
\end{cases}
\]
for a condition $P$ and $\delta_{i,j}=\delta(i=j)$
for $i,j\in\bZ$.

 \subsection{}
\label{subsec:B}

Throughout the paper, we fix a definite quaternion algebra $B$ 
over $\bQ$ with discriminant $d_B$.
Let $x\mapsto \ol{x}$ be the main involution of $B$.
For a subset $X$ of $B$, put $X^-=\{\alpha\in X\mid \alpha+\ol{\alpha}=0\}$.

We write $N_r$ and $\tr_r$ for the reduced norm and the reduced trace
of $\mathrm{M}_r(B)$ to $\bQ$, respectively.
We often suppress the subscript $r$ if there is no fear of confusion.

We fix an isomorphism from $B\otimes_{\bQ}\bR$ to the 
Hamilton quaternion algebra
\[
 \bH=\bR\cdot 1+\bR\cdot i+\bR\cdot j+\bR\cdot k\qquad(i^2=j^2=k^2=-1,\,ij=-ji=k)
\]
and identify $B\otimes_{\bQ}\bR$ with $\bH$.
Throughout the paper  we always distinguish the element $i$ of $\bH$ from
the imaginary unit $\sqrt{-1}\in\bC$.
We define an embedding $A\colon \bH\to\mathrm{M}_2(\bC)$ by
\begin{equation}
 \label{eq:def-A}
A(a_0+a_1i+a_2j+a_3k)
=
\Nmat{a_0+a_1\sqrt{-1}&a_2+a_3\sqrt{-1}}{-a_2+a_3\sqrt{-1} & a_0-
a_1\sqrt{-1}}
\end{equation}
for $a_0,a_1,a_2,a_3\in\bR$.
The map $A$ extends to an isomorphism of $\bH\otimes_{\bR}\bC$
to $\mathrm{M}_2(\bC)$.

 \subsection{}

The measures on $\bR$ and $\bC$ are taken to be the usual Lebesgue measures.

We define the measures $d_L\alpha$ on $\bH$  and $d_L\beta$
on $\bH^-$ by $d_L\alpha=da_0da_1da_2da_3\;\;
(\alpha=a_0+a_1i+a_2j+a_3k, a_i\in\bR)$ and
$d_L\beta=db_1db_2db_3\;\;(\beta=b_1i+b_2j+b_3k, b_i\in\bR)$,
respectively.
The invariant measure $d\cross t$ on $\bH^1=\{t\in \bH\mid N(t)=1\}$ is normalized by
$\mathrm{vol}(\bH^1)=1$.

 \subsection{}
\label{subsec:n-p-measure}

For a place $v$ of $\bQ$, we define the measures 
$d\alpha_v$ on $B_v$  and $d\beta_v$
on $B_v^-$ so that they are self-dual with respect to the pairings
$(\alpha_v,\alpha_v')\mapsto \psi_v(\tr(\ol{\alpha_v}\alpha_v'))$
and 
$(\beta_v,\beta_v')\mapsto \psi_v(\tr(\ol{\beta_v}\beta_v'))$,
respectively.
Note that, for $v=\infty$, $d\alpha_{\infty}=4d_L\alpha_{\infty}$.
We define the measure on $B_v^{(m,n)}$ by
$dX=\prod_{1\le i\le m,\,1\le j\le n}dx_{ij}\;\;(X=(x_{ij}))$,
and the measures $d\alpha$ on $B_{\bA}$
and $d\beta$ on $B^-_{\bA}$ by the product measures
$\prod_{v\le\infty}d\alpha_v$ and $\prod_{v\le\infty}d\beta_v$,
respectively.

 \subsection{}
\label{subsec:Schwartz-Bruhat}

For a place $v$ of $\bQ$, we denote by 
$\mathscr{S}(B_v^{(m,n)})$ the space of Schwartz-Bruhat functions on $B_v^{(m,n)}$.

 \subsection{}
\label{subsec:?}

For a positive integer $\kappa$,  let $V_{\kappa}$ be the space
of homogeneous polynomials in the 
two variables $X$ and $Y$ of total degree $\kappa$.
Then $\{P_r\}_{r=0}^{\kappa}$ forms a basis of $V_{\kappa}$
and $\dim V_{\kappa}=\kappa+1$, where
\begin{equation}
 \label{eq:Pr}
P_r(X,Y)=X^r\,Y^{\kappa-r}\in V_{\kappa}\qquad(0\le r\le \kappa).
\end{equation}
We define $\sigma_{\kappa}\colon \bH\to \mathrm{End}(V_{\kappa})$ by
\begin{equation}
\label{eq:sigma}
 (\sigma_{\kappa}
(\alpha)P)(X,Y)=P((X,Y)A(\alpha))
\qquad(\alpha\in\bH,P\in V_{\kappa}),
\end{equation}
where $A$ is defined by (\ref{eq:def-A}).
Then $\sigma_{\kappa}$ defines
 an irreducible representation of $\bH$ on $V_{\kappa}$.
The inner product of $V_{\kappa}$ defined by
\[
 ( P,P')_{\kappa}=\int_{\bC^2}P(x,y)\,\ol{P'(x,y)}e^{-2\pi
(|x|^2+|y|^2)}dxdy
\qquad(P,P'\in V_{\kappa})
\]
is $A(\bH^1)$-invariant.
Put $\|P\|_{\kappa}=\sqrt{(P,P)_{\kappa}}$ for $P\in V_{\kappa}$.
The following formula is often used in this paper.

\begin{lemma}\label{lem:Ei}
 For $\alpha,\beta\in\bH$, we have
\[
 \int_{\bH}\e(\tr(\alpha y)+\sqrt{-1}\ol{y}y)\,\sigma_{\kappa}(y+\beta)dy
=\e(\sqrt{-1}\ol{\alpha}\alpha)\,\sigma_{\kappa}(\sqrt{-1}\ol{\alpha}+\beta).
\]
\end{lemma}

\begin{proof}
 Note that the matrix coefficients of $\sigma_{\kappa}(y)\;(y\in \bH)$
are spherical functions on $\bH\simeq \bR^4$. Then the lemma follows from
an argument in Chapter 1, Section 2 in \cite{Ei2}.
\end{proof}

 \section{Groups}
\label{sec:groups}

 \subsection{}
\label{subsec:Groups-lattice}

We henceforth fix a maximal order $\cO$ of $B$.
For a finite place $p$ of $\bQ$, let $\cO_p=\cO\otimes_{\Z}\bZ_p$.
For $p|d_B$, denote by $\grP_p$ the maximal ideal of $\cO_p$.
We define a lattice $L_p$ (respectively $L_p'$) of $B_p^{(1,2)}$
(respectively $B_p^{(2,1)}$) by
\begin{align*}
 L_p &=
\begin{cases}
 (\cO_p,\cO_p)& (p\not | d_B)\\
 (\cO_p,\grP_p\inv)&(p|d_B)
\end{cases},\\
L_p'&=\vecm{\cO_p\\\cO_p}.
\end{align*}
Then $L_p$ and $L'_p$ are maximal $\cO_p$- lattices of
$B_p^{(1,2)}$ and $B_p^{(2,1)}$ respectively
(see \cite{Shi}, Section  3.2).

 \subsection{}
\label{subsec:groups-H}

Let $H$ be an algebraic group  over $\bQ$ defined by
\begin{align*}
 H_{\bQ} &=\left\{h\in \MM_2(B) \mid \tp\ol{h}Jh=J\right\}, J=\Nmat{0&1}{-1&0}.
\end{align*}
Note that a homomorphism 
\[
 H'\ni \left(\Nmat{a&b}{c&d},\alpha\right)\mapsto \alpha\Nmat{a&b}{c&d}\in H
\]
defines an isomorphism $H'/Z'\simeq H$,
where $\bQ$-algebraic groups $H'$ and $Z'$ are defined by
\begin{align*}
 H'_{\bQ}&=\{(h,\alpha)\in \GL_2(\bQ)\times B\cross\mid \det h\cdot N(\alpha)=1\},\\
Z'_{\bQ}&=\{(a1_2,a\inv)\mid a\in\bQ\cross\}.
\end{align*}

For a finite place $p$ of $\bQ$, we put $U_p=\{u\in H_p\mid u L'_p=L'_p\}$.
Then $U_p$ is a maximal open
compact subgroup of $H_p$. We put $U_f=\prod_{p<\infty}U_p$.

We define an action of $H_{\infty}$ on the upper half plane $\grH=\{\tau\in\bC\mid \Im(\tau)>0\}$ by
\[
 h\cdot \tau=(a\tau+b)(c\tau+d)\inv\qquad\left(h=\alpha\Nmat{a&b}{c&d}\in H_{\infty},
\tau\in\grH\right).\
\]
Then
\begin{align*}
 U_{\infty}&=\{u\in H_{\infty}\mid u\cdot\sqrt{-1}=\sqrt{-1}\}\\
&=\left\{t\,\Nmat{\cos\theta&\sin\theta}{-\sin\theta&\cos\theta}\mid
t\in\bH^1,\,0\le \theta<2\pi\right\}
\end{align*}
is a maximal compact subgroup of $H_{\infty}$ and
the representation $\rho_{\kappa}$ of $U_{\infty}$ on $V_{\kappa}$ defined by
\[
 \rho_{\kappa}\left(t\,\Nmat{\cos\theta&\sin\theta}{-\sin\theta&\cos\theta}\right)=
e^{-\sqrt{-1}\kappa\theta}\sigma_{\kappa}(t)
\qquad(t\in\bH^1,\,0\le \theta<2\pi)
\]
 is irreducible.

We normalize the Haar measure $dh$  on $H_{\bA}$ by
$\mathrm{vol}(H_{\bQ}\backslash H_{\bA})=1$.

 \subsection{}
\label{subsec:groups-G}

Let $G=\mathrm{Sp}(1,1)$ be the quaternion unitary group of signature
$(1,1)$ defined over $\bQ$ with
\[
 G_{\bQ}=\left\{g\in \MM_2(B)\mid \tp \ol{g}Qg=Q\right\},\quad 
Q=\Nmat{0&1}{1&0}.
\]
This is also known as the indefinite symplectic group of signature
$(1,1)$.

For a finite place $p$ of $\bQ$, we put
\[
 K_p=\left\{k\in G_p\mid L_pk=L_p \right\}.
\]
Then $K_p$
 is a maximal open compact subgroup of $G_p$.
Note that, for $p|d_B$, a maximal open compact subgroup of $G_p$ is
conjugate in $G_p$ to either $K_p$ or $G_p\cap \GL_2(\cO_p)$.
Put $K_f=\prod_{p<\infty}K_p$.

The real points $G_{\infty}$ of $G$ acts on $\cX=\{z\in\bH\mid \tr(z)>0\}$
by
\[
 g\cdot z=(az+b)(cz+d)\inv\qquad
\left(g=\Nmat{a&b}{c&d}\in G_{\infty},\,z\in\cX\right).
\]
Then
\[
 K_{\infty}=\{k\in G_{\infty}\mid k\cdot 1=1\}
=\left\{\Nmat{a&b}{b&a}\mid a,b\in\bH,\,N(a)+N(b)=1,\,\tr(\ol{a}b)=0\right\}
\]
is a maximal compact subgroup of $G_{\infty}$ isomorphic to
$\bH^1\times \bH^1$ via $\Nmat{a&b}{b&a}\mapsto (a+b,a-b)$.
For a positive integer $\kappa$, let $\tau_{\kappa}$ and $\tau_{\kappa}^-$
 be  the representations of $K_{\infty}$ on $V_{\kappa}$ defined by
\begin{align*}
 \tau_{\kappa}(k)&=\sigma_{\kappa}(a+b)\\
 \tau_{\kappa}^-(k)&=\sigma_{\kappa}(a-b)
\qquad
\left(k=\Nmat{a&b}{b&a}\in K_{\infty}\right).
\end{align*}
Note that $\tau_{\kappa}$ and $\tau_{\kappa}^-$ are irreducible and that
$\tau_{\kappa}^-$ is not equivalent to $\tau_{\kappa}$.

We normalize the Haar measure  on $G_p$ by
$\mathrm{vol}(K_p)=1$.
We normalize the Haar measure on $G_{\infty}$  by
\[
 dg_{\infty}=a^{-6}d_L\beta d\cross a\,dk,
\]
where 
\[
 g_{\infty}=\Nmat{1&\beta}{0&1}\Nmat{a&0}{0&a\inv}k
\qquad(\beta\in\bH^-,a\in\bR_{>0},k\in K_{\infty})
\]
and
$dk$ is normalized by $\mathrm{vol}(K_{\infty})=1$.
Define a Haar measure $dg$ on $G_{\bA}$  by
$dg=\prod_{v\le\infty}dg_v$.

 \section{Automorphic forms on $H$}  \label{sec:af-H}

 \subsection{} \label{subsec:af-H}

We normalize the Haar measure $dx_{\infty}$ on $H_{\infty}$
by
\begin{equation}
 \label{eq:measure-H}
\int_{H_{\infty}}\,\varphi(x_{\infty})d_{\infty}=
\int_{\bR}db\,\int_0^{\infty}a^{-2}d\cross a\,\int_{U_{\infty}}du_{\infty}
\,\varphi\left(\Nmat{1&b}{0&1}\Nmat{a&0}{0&a\inv}u_{\infty}\right)
\end{equation}
for $\varphi\in L^1(H_{\infty})$,
where $du_{\infty}$ is normalized by $\mathrm{vol}(U_{\infty})=1$.

For a positive integer $\kappa$, 
we define an $\mathrm{End}(V_{\kappa})$-valued spherical function
$\omega_{\kappa}$ on $H_{\infty}$ by
\begin{align*}
 \omega_{\kappa}(h) 
&=\sigma_{\kappa}\left(\dfrac{1}{2}\left(-\sqrt{-1},1\right)
h\vecm{\sqrt{-1}\\1}\right)\inv
\qquad(h\in H_{\infty}).
\end{align*}
It is easily verified that
\begin{equation}
 \label{eq:U-equivariance-omega}
 \omega_{\kappa}(uhu')=\rho_{\kappa}(u')\inv 
\omega_{\kappa}(h)\rho_{\kappa}(u)\inv
\qquad(u,u'\in U_{\infty},h\in H_{\infty})
\end{equation}
and that $\omega_{\kappa}$ is integrable on $H_{\infty}$ if $\kappa>2$.

Suppose that $\kappa>2$.
Let $\cS_{\kappa}^H$ be the space of smooth functions 
$f\colon H_{\bQ}\backslash H_{\bA} \to V_{\kappa}$ satisfying
the following three conditions:
\begin{itemize}
 \item[(1)] \quad$f(hu_fu_{\infty})=\rho_{\kappa}(u_{\infty})\inv f(h)\quad(h\in H_{\bA},u_f\in U_f,u_{\infty}\in U_{\infty}$).

 \item[(2)] \quad $\|f(h)\|_{\kappa}$ is bounded on $H_{\bA}$.
\item[(3)] \quad For $h\in H_{\bA}$, we have
\begin{equation}
 \label{eq:reproducing-H}
 f(h)=c_{\kappa}^H\,\int_{H_{\infty}}\,
\omega_{\kappa}(x_{\infty})f(hx_{\infty}\inv)
dx_{\infty},
\end{equation}
where 
\begin{equation}
 \label{eq:c_kappa^H}
c_{\kappa}^H=\dfrac{\kappa-1}{2\pi}.
\end{equation}
\end{itemize}
Note that
\begin{equation}
 \label{eq:cuspidality-f}
\int_{\bQ\backslash\bQ_{\bA}}\,f\left(\Nmat{1&b}{0&1}h\right)db=0
\qquad(h\in H_{\bA}),
\end{equation}
which is deduced from the boundedness and the Fourier expansion
of $f$.

We define an inner product $\langle,\rangle_H$ on $\cS_{\kappa}^H$
by%
\[
 \langle f,f'\rangle_H=
\int_{H_{\bQ}\backslash H_{\bA}}( f(h),f'(h))_{\kappa}dh
\qquad(f,f'\in\cS_{\kappa}^H),
\]
where $dh$ is the right-invariant measure 
on $H_{\bQ}\backslash H_{\bA}$ induced by the Haar measure on $H_{\bA}$
given in Section \ref{subsec:groups-H}.

 \subsection{}
\label{subsec:Omega-H}

For a finite place $p$ of $\bQ$, let
$\cH(H_p,U_p)$ be the Hecke algebra: By definition, $\cH(H_p,U_p)$
is the space of compactly supported bi-$U_p$-invariant functions on $H_p$
and the product is defined by
\[
 \left(\varphi_p*\varphi'_p\right)
(h_p)=\int_{H_p}\varphi_p(h_px_p\inv)\varphi'_p(x_p)dx_p
\qquad(\varphi_p,\varphi'_p\in \cH(H_p,U_p),h_p\in H_p).
\]
Here the Haar measure $dx_p$ is normalized by $\mathrm{vol}(U_p)=1$.
The Hecke algebra $\cH(H_p,U_p)$ acts on $\cS_{\kappa}^H$ by
\[
 \left(f*\varphi_p\right)(h)=\int_{H_p}f(hx_p\inv)\varphi_p(x_p)dx_p
\qquad(f\in\cS_{\kappa}^H,\varphi_p\in\cH(H_p,U_p),h\in H_{\bA}).
\]

We say that 
$f\in\cS_{\kappa}^H$ is a \emph{Hecke eigenform}
if
\begin{equation}
 \label{eq:Hecke-eigenvalues-H}
f*\varphi_p=\lambda_p^f(\varphi_p)\,f
\end{equation}
holds for any finite place $p$ of $\bQ$ and any $\varphi_p\in \cH(H_p,U_p)
$ with   a $\bC$-algebra homomorphism $\lambda_p^f$ from
$\cH(H_p,U_p)$ to $\bC$.
The following result is a straightforward consequence of 
\cite{Sa}, Section 5.2.

\begin{lemma}
\label{lem:Hecke-eigen-H}
Suppose that $f\in\cS_{\kappa}^H$ is a Hecke eigenform and
 let $p$ be a finite place of $\bQ$. 
\begin{itemize}
 \item[(1)]\ There exists uniquely a bi-$U_p$-invariant
function $\omega_p^f$ on $G_p$ satisfying $\omega_p^f(1)=1$ and
\[
 \int_{H_p}\,\omega_p^f(hx\inv)\varphi_p(x)dx
=\lambda_p^f(\varphi_p)\omega_p^f(h)\qquad(h\in H_p)
\]
for any $\varphi_p\in\cH(H_p,U_p)$.
 
 \item[(2)] \ We have
\[
 \int_{U_p}\,f(hu_px)du_p=\omega_p^f(x_p)f(hx_p\inv x)
\]
for $h, x\in H_{\bA}$, where $x_p$ is the $p$-component of $x$.
\end{itemize}
\end{lemma}

 \subsection{}
\label{subsec:Satake-H-unramified} 

In this and the next subsections, we assume that
$f\in\cS_{\kappa}^H$ is a Hecke eigenform and 
 define the local factor $L_p(s,f,\mathrm{std})$ of the standard $L$-function 
attached to $f$ at a finite place $p$ of $\bQ$.

In this subsection, we suppose that $B$ is unramified at $p$
(namely $p \nmid d_B$).
We identify
$B_p$ with $\MM_2(\bQ_p)$ 
by an isomorphism
$B_p\simeq \MM_2(\bQ_p)$ such that $\cO_{p}$ maps to $\MM_2(\bZ_p)$ and 
such that
the main involution of $B_p$ corresponds to 
\[
 X=\Nmat{a&b}{c&d}\mapsto J\inv \tp XJ=\Nmat{d&-b}{-c&a},\qquad J=\Nmat{0&1}{-1&0}.
\]
Then $\tr(X)$ and $N(X)$ are the trace and the determinant of $X\in \mathrm{M}_2(\bQ_p)$, respectively.
We identify $H_p$ and $U_p$ with 
\[
 \left\{h\in \GL_4(\bQ_p)\mid \tp h\Nmat{0&J}{-J&0}h
=\Nmat{0&J}{-J&0}\right\},
\]
and 
\[
 \left\{u\in \GL_4(\bZ_p)\mid \tp u\Nmat{0&J}{-J&0}u
=\Nmat{0&J}{-J&0}\right\},
\]
respectively.

For $\chi=(\chi_1,\chi_2)\in X_{\mathrm{unr}}(\bQ_p\cross)^2$, 
we define the spherical function $\omega_{\chi}$ on $H_p$ by
\begin{align*}
 \omega_{\chi}(h)&=\int_{U_p}\eta^H_{\chi}(uh)du,\\
 \eta^H_{\chi}(h)&=
\chi_1(a_1)\chi_2(a_2)|a_1|_p\\
&\qquad
\qquad\left(h=
\begin{pmatrix}
 a_1 & * & * & *\\
0 & a_2 & * & *\\
0 & 0 & a_2\inv & *\\
0 & 0 & 0 & a_1\inv
\end{pmatrix}\,u\in H_p,\ a_1,a_2\in \bQ_p\cross, u\in U_p
\right)
\end{align*}
(for the definition of $X_{\mathrm{unr}}(\bQ_p\cross)$, see Section
\ref{subsec:notation-1}).

\begin{lemma}[\cite{Sa}, Section 5.4, Theorem 2]
\label{lem:SatakeParameter-H-unramified}
 There exists a $\chi=(\chi_1,\chi_2)\in X_{\mathrm{unr}}(\bQ_p\cross)^2$
such that $\omega_p^f=\omega_{\chi}$.
\end{lemma}

We define the local factor  of the standard $L$-function 
 attached to $f$ at $p$ by
\[
 L_p(s, f,\mathrm{std})=\prod_{i=1}^2\,L_p(s,\chi_i)L_p(s,\chi_i\inv),
\]
where $L_p(s,\chi_i^{\pm})=(1-\chi_i^{\pm}(p)p^{-s})\inv$.
Note that the definition of $L_p(s,f,\mathrm{std})$
does not depend on the choice of $\chi$.

 \subsection{}
\label{subsec:Satake-H-ramified} %

In this subsection we suppose that $B$ is ramified at $p$
(namely $p | d_B$).

Let $X_{\mathrm{unr}}(B_p\cross)$ be the set of homomorphisms
from $B_p\cross$ to $\bC\cross$ trivial on $\cO_p\cross$.
For  $\chi\in X_{\mathrm{unr}}(B_p\cross)$, we define the spherical function $\omega_{\chi}$ on $H_p$ by
\begin{align*}
 \omega_{\chi}(h)&=\int_{U_p}\,\eta_{\chi}(uh)du,\\
\eta_{\chi}\left(\Nmat{1&b}{0&1}\Nmat{\alpha&0}{0&\ol{\alpha}\inv}u\right)
&=\chi(\alpha)|N(\alpha)|_p^{1/2}\qquad (b\in \bQ_p^-,\alpha\in B_p\cross,
u\in U_p).
\end{align*}

\begin{lemma}[\cite{Sa}, Section 5.4, Theorem 2]
\label{lem:SatakeParameter-H-ramified}
\quad
 There exists a $\chi\in X_{\mathrm{unr}}(B_p\cross)$ such that
$\omega_p^f=\omega_{\chi}$.
\end{lemma}
Define
\[
 L_p(s, f,\mathrm{std})=
    \left(1-\chi(\Pi_p)p^{-s-1/2}\right)\inv 
\left(1-\chi\inv(\Pi_p)p^{-s-1/2}\right)\inv.
\]

 \subsection{} \label{sec:global-L-H}

We define the standard $L$-function attached to a Hecke
eigenform $f\in\cS_{\kappa}^H$ by
\begin{equation}
 \label{eq:global-L-H}
L(s,f,\mathrm{std})
=\prod_{p<\infty}\,L_p(s,f,\mathrm{std}).
\end{equation}

 \section{Automorphic forms on $G$} \label{sec:af-G}

 \subsection{}\label{subsec:af-G}

Recall that the Haar measure $dx_{\infty}$ on $G_{\infty}$
is normalized by
\begin{equation}
 \label{eq:measure-G}
\int_{G_{\infty}}\,\varphi(x_{\infty})dx_{\infty}=
\int_{\bH^-}d_L\beta\,\int_0^{\infty}a^{-6}d\cross a\,\int_{K_{\infty}}dk_{\infty}
\,\varphi\left(\Nmat{1&\beta}{0&1}\Nmat{a&0}{0&a\inv}k_{\infty}\right)
\end{equation}
\begin{equation*}
 =\int_{K_{\infty}}dk_{\infty}\,\int_0^{\infty}a^{6}d\cross a\,\int_{\bH^-}d_L\beta\,
\,\varphi\left(k_{\infty}\Nmat{a&0}{0&a\inv}\Nmat{1&\beta}{0&1}\right)
\end{equation*}
for $\varphi\in L^1(G_{\infty})$,
where $dk_{\infty}$ is normalized by $\mathrm{vol}(K_{\infty})=1$
(see Section \ref{subsec:groups-G}).

For $g=\Nmat{a&b}{c&d}\in G_{\infty}$, put
\[
 \Delta_g=\dfrac{1}{2}(1,1)g\vecm{1\\1}=\dfrac{1}{2}(a+b+c+d).
\]
Then $\Delta_g\in\bH\cross$.
For a positive integer $\kappa$,
we define an $\mathrm{End}(V_{\kappa})$-valued spherical function
$\Omega_{\kappa}$ on $G_{\infty}$ by
\begin{equation}
\label{eq:def-Omega}
 \Omega_{\kappa}(g) =N(\Delta_g)\inv\,\sigma_{\kappa}(\Delta_g)\inv
\qquad(g\in G_{\infty}),\\
\end{equation}
We see that
\begin{equation}
 \label{eq:K-equivariance-Omega}
\Omega_{\kappa}(kgk')=\tau_{\kappa}(k')\inv 
\Omega_{\kappa}(g)\tau_{\kappa}(k)\inv
\qquad(k,k'\in K_{\infty},g\in G_{\infty}).
\end{equation}
It is known  that $\Omega_{\kappa}$ is integrable on $G_{\infty}$ 
if $\kappa>4$ (see \cite{Ar1}, Lemma 1.1).
We henceforth fix an integer $\kappa>4$.

 Let $\cS_{\kappa}^G$ be the space of smooth functions 
$F\colon G_{\bQ}\backslash G_{\bA}\to V_{\kappa}$ satisfying
the following three conditions:
\begin{itemize}
 \item[(1)] \quad For $g\in G_{\bA},k_f\in K_f,k_{\infty}\in K_{\infty}$,
we have
\begin{equation}
 \label{eq:def-af-G-1}
F(gk_fk_{\infty})=\tau_{\kappa}(k_{\infty})\inv F(g).
\end{equation}
 \item[(2)] \quad $\|F(g)\|_{\kappa}$ is bounded on $G_{\bA}$.
 \item[(3)] \quad For $g\in G_{\bA}$, we have
\begin{equation}
 \label{eq:def-af-G-2}
 F(g)=c_{\kappa}^G\,\int_{G_{\infty}}\,
\Omega_{\kappa}(x_{\infty})F(gx_{\infty}\inv)
dx_{\infty},
\end{equation}
where 
\begin{equation}
 \label{eq:c_kappa^G}
c_{\kappa}^G=\dfrac{\kappa(\kappa-1)}{8\pi^2}.
\end{equation}
\end{itemize}
Note that 
\begin{equation}
 \label{eq:cuspidal}
\int_{B^-\backslash B^-_{\bA}}\,F\left(\Nmat{1&\beta}{0&1}g\right)d\beta
=0\qquad(g\in G_{\bA})
\end{equation}
for $F\in\cS_{\kappa}^G$
(see \cite{Ar2}, Proposition 3.1).
We define an inner product $\langle,\rangle_G$ on $\cS_{\kappa}^G$
by
\[
 \langle F,F'\rangle_G=
\int_{G_{\bQ}\backslash G_{\bA}}\,(F(g),F'(g))_{\kappa}dg
\qquad(F,F'\in\cS_{\kappa}^G),
\]
where  $dg$ is the right-invariant measure 
on $G_{\bQ}\backslash G_{\bA}$ induced by the Haar measure on $G_{\bA}$
given in Section \ref{subsec:groups-G}.

For a general theory of automorphic forms on $G$
generating the quaternionic discrete series
representation at the Archimedean place, we refer to
 \cite{Ar1}, \cite{Ar2},
\cite{Na1} and \cite{Na2}.

 \subsection{}
\label{subsec:Omega-G}

For a finite place $p$ of $\bQ$, let
$\cH(G_p,K_p)$ be the Hecke algebra: By definition, $\cH(G_p,K_p)$
is the space of compactly supported bi-$K_p$-invariant functions on $G_p$
and the product is defined by
\[
 \left(\varphi_p*\varphi'_p\right)
(g_p)=\int_{G_p}\varphi_p(g_px_p\inv)\varphi'_p(x_p)dx_p
\qquad(\varphi_p,\varphi'_p\in \cH(G_p,K_p),g_p\in G_p).
\]
Here we recall that 
the Haar measure $dx_p$ on $G_p$ is normalized by $\mathrm{vol}(K_p)=1$.
The Hecke algebra $\cH(G_p,K_p)$ acts on $\cS_{\kappa}^G$ by
\[
 \left(F*\varphi_p\right)(g)=\int_{G_p}F(gx_p\inv)\varphi_p(x_p)dx_p
\qquad(F\in\cS_{\kappa}^G,\varphi_p\in\cH(G_p,K_p),g\in G_{\bA}).
\]

We say that 
$F\in\cS_{\kappa}^G$ is a \emph{Hecke eigenform}
if
\begin{equation}
 \label{eq:Hecke-eigenvalues-G}
 F*\varphi_p=\Lambda_p^F(\varphi_p)\,F
\end{equation}
holds for any finite place $p$  of $\bQ$  and any $\varphi_p\in \cH(G_p,K_p)
$ with   a $\bC$-algebra homomorphism $\Lambda_p^F$ from
$\cH(G_p,K_p)$ to $\bC$.
The following result is a straightforward consequence of 
\cite{Sa}, Section 5.2.

\begin{lemma}
\label{lem:Hecke-eigen}
Suppose that $F\in\cS_{\kappa}^G$ is a Hecke eigenform and
 let $p$ be a finite place of $\bQ$. 
\begin{itemize}
 \item[(1)]\ There exists uniquely a bi-$K_p$-invariant
function $\Omega_p^F$ on $G_p$ satisfying $\Omega_p^F(1)=1$ and
\[
 \int_{G_p}\,\Omega_p^F(gx\inv)\varphi_p(x)dx=\Lambda_p^F(\varphi_p)\Omega_p^F(g)\qquad(g\in G_p)
\]
for any $\varphi_p\in\cH(G_p,K_p)$.
 
 \item[(2)] \ We have
\[
 \int_{K_p}\,F(gk_px)dk_p=\Omega_p^F(x_p)F(gx_p\inv x)
\]
for $g, x\in G_{\bA}$, where $x_p$ is the $p$-component of $x$.
\end{itemize}
\end{lemma}

 \subsection{}
\label{subsec:Satake-G-unramified} 

In this and the next subsections, we assume that
$F\in\cS_{\kappa}^G$ is a Hecke eigenform and 
 define the local factor $L_p(s,F,\mathrm{std})$ of the standard $L$-function 
attached to $F$ at a finite place $p$.

In this subsection, we suppose that $B$ is unramified at $p$
(namely $p \nmid d_B$).
We use the same notation as in Section \ref{subsec:Satake-H-unramified}.
We identify $G_p$ and $K_p$ with 
\[
 \Sp_2(\bQ_p)=\left\{g\in \GL_4(\bQ_p)\mid \tp g\Nmat{0&1_2}{-1_2&0}g
=\Nmat{0&1_2}{-1_2&0}\right\},
\]
and $\Sp_2(\bZ_p)=G_p\cap \GL_4(\bZ_p)$, respectively.

For $\Xi=(\Xi_1,\Xi_2)\in X_{\mathrm{unr}}(\bQ_p\cross)^2$, 
we define the spherical function $\Omega_{\Xi}$ on $G_p$ by
\begin{align*}
 \Omega_{\Xi}(g)&=\int_{K_p}\eta^G_{\Xi}(kg)dk,\\
 \eta^G_{\Xi}(g)&=
\Xi_1(a_1)\Xi_2(a_2)|a_1|^2_p|a_2|_p\\
&\qquad
\qquad\left(g=
\begin{pmatrix}
 a_1 & * & * & *\\
0 & a_2 & * & *\\
0 & 0 & a_1\inv & 0\\
0 & 0 & * & a_2\inv
\end{pmatrix}\,k\in G_p,\ a_1,a_2\in \bQ_p\cross, k\in K_p
\right).
\end{align*}

\begin{lemma}[\cite{Sa}, Section 5.4, Theorem 2]
\label{lem:SatakeParameter-G-unramified}
 There exists a $\Xi=(\Xi_1,\Xi_2)\in X_{\mathrm{unr}}(\bQ_p\cross)^2$
such that $\Omega_p^F=\Omega_{\Xi}$.
\end{lemma}

We define the local factor  of the standard $L$-function 
 attached to $F$ at $p$ by
\[
 L_p(s, F,\mathrm{std})=\zeta_p(s)\,\prod_{i=1}^2\,L_p(s,\Xi_i)L_p(s,\Xi_i\inv),
\]
where
\begin{align*}
 \zeta_p(s)&=(1-p^{-s})\inv,\\
 L_p(s,\Xi_i^{\pm 1})&=\left(1-\Xi_i^{\pm 1}(p)p^{-s}\right)\inv.
\end{align*}
Note that the definition of $L_p(s,F,\mathrm{std})$
does not depend on the choice of $\Xi$.

 \subsection{}
\label{subsec:Satake-G-ramified} %

In this subsection we suppose that $B$ is ramified at $p$
(namely $p|d_B$).
We use the same notation as in Section \ref{subsec:Satake-H-ramified}.

For  $\Xi\in X_{\mathrm{unr}}(B_p\cross)$, we define the spherical function $\Omega_{\Xi}$ on $G_p$ by
\begin{align*}
 \Omega_{\Xi}(g)&=\int_{K_p}\,\eta^G_{\Xi}(kg)dk,\\
\eta^G_{\Xi}\left(\Nmat{1&\beta}{0&1}\Nmat{\alpha&0}{0&\ol{\alpha}\inv}k\right)
&=\Xi(\alpha)|N(\alpha)|_p^{3/2}\qquad (\beta\in B_p^-,\alpha\in B_p\cross,
k\in K_p).
\end{align*}

\begin{lemma}[\cite{Sa}, Section 5.4, Theorem 2]
\label{lem:SatakeParameter-G-ramified}
\quad
 There exists a $\Xi\in X_{\mathrm{unr}}(B_p\cross)$ such that
$\Omega_p^F=\Omega_{\Xi}$.
\end{lemma}
Define
\[
 L_p(s, F,\mathrm{std})=\left(1-p^{-s}\right)\inv
    \left(1-\Xi(\Pi_p)p^{-s-1/2}\right)\inv 
\left(1-\Xi\inv(\Pi_p)p^{-s-1/2}\right)\inv,
\]
where $\Pi_p$ is a prime element of $B_p$.

 \subsection{} \label{sec:global-L-G}

We define the standard $L$-function attached to a Hecke
eigenform $F\in\cS_{\kappa}^G$ by
\begin{equation}
 \label{eq:global-L-G}
L(s,F,\mathrm{std})
=\prod_{p<\infty}\,L_p(s,F,\mathrm{std}).
\end{equation}

 \section{Arakawa lifting}
\label{sec:Arakawa-lifting}

 \subsection{}

Let $v$ be a place of $\bQ$.
We define the Weil representation $r_v$ of $H_v\times G_v$ on the space
\[
 W_v=
\begin{cases}
 \mathscr{S}(B_v^{(2,1)}) 
     & (v<\infty),\\
 \mathscr{S}(B_v^{(2,1)})\otimes \mathrm{End}(V_{\kappa}) & (v=\infty)
\end{cases}
\]
as follows: For $\varphi\in W_v$ and $X\in B_v^{(2,1)}$,
\begin{align*}
 r_v(h,1_2)\varphi(X)&=\varphi(\tp\ol{h}X)\qquad (h\in H_v)\\
 r_v\left(1_2,\Nmat{\alpha&0}{0&\ol{\alpha}\inv}\right)\varphi(X) & = 
       |N(\alpha)|_v^2\,\varphi(X\alpha)\qquad(\alpha\in B_v\cross)\\
r_v\left(1_2,\Nmat{1&\beta}{0&1}\right)\varphi(X) & = 
   \psi_v\left(-\dfrac{1}{2}\tr(\beta\tp\ol{X}JX)\right)\varphi(X)\qquad(\beta\in B_v^-)\\
r_v\left(1_2,\Nmat{0&1}{1&0}\right)\varphi(X) & = 
\int_{B_v^{(2,1)}}\psi_v\left(-\tr(\tp\ol{Y}JX)\right)\varphi(Y)dY.
\end{align*}

To give another realization of the Weil representation, we define an intertwining
operator $I_v$ from $W_v$ to 
\[
 W_v'=
\begin{cases}
 \mathscr{S}(B_v^{(1,2)}) 
     & (v<\infty),\\
 \mathscr{S}(B_v^{(1,2)})\otimes \mathrm{End}(V_{\kappa}) & (v=\infty)
\end{cases}
\]
by a partial Fourier transform
\[
 I_v\varphi(x_1,x_2)=\int_{B_v}\psi_v\left(\tr(\ol{y}x_2)\right)
   \varphi\vecm{x_1\\y}dy
\qquad
(\varphi\in W_v, x_1,x_2\in B_v).
\]
The inverse of $I_v$ is given by
\[
  I_v\inv\phi\vecm{x_1\\x_2}=\int_{B_v}\psi_v\left(-\tr(\ol{y}x_2)\right)
   \phi(x_1,y)dy
\qquad
(\phi\in W_v', x_1,x_2\in B_v).
\]
We set
\[
 r_v'(h,g)=I_v\circ r_v(h,g)\circ I_v\inv\qquad(h\in H_v,\,g\in G_v),
\]
which defines a smooth representation of $H_v\times G_v$ on
$W_v'$.
A direct calculation shows the following:

\begin{lemma}
\label{lem:r'}
 For $\phi\in W_v'$ and $X'\in B_v^{(1,2)}$,
\begin{align*}
 r_v'\left(\Nmat{\alpha&0}{0&\ol{\alpha}\inv},1_2\right)\phi(X') & = 
       |N(\alpha)|_v^2\,\phi(\ol{\alpha}X')\qquad(\alpha\in B_v\cross)\\
r_v'\left(\Nmat{1&b}{0&1},1_2\right)\phi(X') & = 
   \psi_v\left(-b\,\ol{X'}Q\tp X'\right)\phi(X')\qquad(b\in \bQ_v)\\
r_v'\left(\Nmat{0&1}{-1&0},1_2\right)\phi(X') & = 
\int_{B_v^{(1,2)}}\psi_v\left(-\tr(\ol{X'}Q\tp Y)\right)\phi(Y)dY\\
r_v'(1_2,g)\phi(X')&=\phi(X'g)\qquad (g\in G_v).
\end{align*}
\end{lemma}

 \subsection{}
Let $\mathscr{S}(B_{\bA}^{(2,1)})$ and $\mathscr{S}(B_{\bA}^{(1,2)})$ 
be the restricted
tensor products of $\mathscr{S}(B_{v}^{(2,1)})$ and
$\mathscr{S}(B_{v}^{(1,2)})$ 
over the places $v$
of $\bQ$ with respect to $\{\varphi_{0,p}\}_{p<\infty}$
and $\{\phi_{0,p}\}_{p<\infty}$, respectively.
Here $\varphi_{0,p}$ and $\phi_{0,p}$
 denotes the characteristic functions of $L_p'$ and $L_p$, respectively
(for the definitions of $L_p'$ and $L_p$, see Section \ref{subsec:Groups-lattice}).
We define smooth representations $r$ and $r'$ of $H_{\bA}\times G_{\bA}$
on 
$W=\mathscr{S}(B_{\bA}^{(2,1)})\otimes\mathrm{End}(V_{\kappa})$
and 
$W'=\mathscr{S}(B_{\bA}^{(1,2)})\otimes\mathrm{End}(V_{\kappa})$
as  restricted tensor products of $r_v$ and $r_v'$ respectively:
\[
 r=\otimes_v\,r_v,\, r'=\otimes_v\,r'_v.
\]
Then $r'(g)=I\circ r(g)\circ I\inv$, where 
\[
 I\varphi(x_1,x_2)=\int_{B_{\bA}}\psi\left(\tr(\ol{y}x_2)\right)
   \varphi\vecm{x_1\\y}dy
\qquad
(\varphi\in W, x_1,x_2\in B_{\bA}).
\]

We define  test functions  $\varphi_0=\otimes_v\varphi_{0,v}\in W$
and  $\varphi_0=\otimes_v\varphi_{0,v}\in W$,
where 
\begin{align*}
 \varphi_{0,\infty}(X)&=\e(\sqrt{-1}\,\tp\ol{X}X)\,\sigma_{\kappa}((1,-\sqrt{-1})\ol{X})\qquad (X\in\bH^{(2,1)}),\\
 \phi_{0,\infty}(X')&=\e(\sqrt{-1}\,\ol{X'}\,\,\tp\! X')\,\sigma_{\kappa}((1,1)\tp\ol{X'})\qquad (X\in\bH^{(1,2)}).
\end{align*}

The following results are proved in a straightforward manner.

\begin{lemma}
\label{lem:phi0}

For a place $v$ of $\bQ$, we have 
\[
 \phi_{0,v}=I_v\varphi_{0,v}.
\]
 \end{lemma}

\begin{lemma}
\label{lem:phi0-UK}
 \begin{itemize}
  \item[(1)] \ For a finite place $p$, we have
\begin{align*}
 r_{p}(u_{p},k_{p})\varphi_{0,p}&=\varphi_{0,p},\\
r'_{p}(u_{p},k_{p})\phi_{0,p}&=\phi_{0,p}
\end{align*}
for $u_{p}\in U_{p}, k_{p}\in K_{p}$.
  \item[(2)] \ We have
\begin{align*}
 \left(r_{\infty}(u_{\infty},k_{\infty})\varphi_{0,\infty}\right)(X)
&=\tau_{\kappa}(k_{\infty})\inv\,\circ\,\varphi_{0,\infty}(X)\,\circ\,
\rho_{\kappa}(u_{\infty}),\\
\left(r'_{\infty}(u_{\infty},k_{\infty})\phi_{0,\infty}\right)(X')
&=\tau_{\kappa}(k_{\infty})\inv\,\circ\,\phi_{0,\infty}(X')\,\circ\,
\rho_{\kappa}(u_{\infty})
\end{align*}
for $u_{\infty}\in U_{\infty},
k_{\infty}\in K_{\infty},X\in\bH^{(2,1)},X'\in\bH^{(1,2)}$.
 \end{itemize}
\end{lemma}

 \subsection{}

We define the theta kernel $\theta\colon H_{\bA}\times G_{\bA}\to
\mathrm{End}(V_{\kappa})$ by
\[
 \theta(h,g)=\sum_{X\in B^{(2,1)}}\,r(h,g)\varphi_0(X)
\qquad(h\in H_{\bA}, g\in G_{\bA}).
\]
By Poisson summation formula, we have
\[
 \theta(h,g)=\sum_{X'\in B^{(1,2)}}\,r'(h,g)\phi_0(X').
\]

\begin{lemma}
 We have
\[
 \theta\left(\gamma hu_fu_{\infty},\gamma'gk_fk_{\infty}\right)
=\tau_{\kappa}(k_{\infty})\inv\circ \theta(h,g)\circ \rho_{\kappa}(u_{\infty})
\]
for $\gamma\in H_{\bQ},h\in H_{\bA}, u_f\in U_f, u_{\infty}\in U_{\infty},
\gamma'\in G_{\bQ},g\in G_{\bA}, k_f\in K_f, k_{\infty}\in K_{\infty}
$.
\end{lemma}

\begin{proof}
 This follows from  Lemma \ref{lem:phi0-UK}.
\end{proof}

 \subsection{}

We define 
the theta lift $\cL(f)$ of $f\in \cS_{\kappa}^H$ by
\begin{equation}\label{eq:theta-lift}
 \cL(f)(g)=\int_{H_{\bQ}\backslash H_{\bA}}\theta(h,g)f(h)dh.
\end{equation}
Here we recall that the Haar measure $dh$ on $H_{\bA}$ is normalized by
$\mathrm{vol}(H_{\bQ}\backslash H_{\bA})=1$.
The integral (\ref{eq:theta-lift}) is absolutely convergent.

\begin{theorem}[\cite{Na3}]\label{th:holomorphy-L}
\quad
We have
 $\cL(f)\in\cS_{\kappa}^G$.
\end{theorem}

The theta lifting $\cL\colon \cS_{\kappa}^H\,\to\,\cS_{\kappa}^G$
was first introduce essentially by Arakawa in his unpublished notes.
The second named author of this paper, based on Arakawa's notes,
 developed Arakawa's work to 
the the lifting from elliptic modular forms to automorphic
forms on $\mathrm{Sp}(1,q)$ generating the quaternionic discrete series representations at the Archimedean place (\cite{Na3}).

Specifying some arithmetic group defined by a global maximal lattice, Arakawa and the second named author proved Theorem \ref{th:holomorphy-L}
 by showing that the lifts of
Poincar\'e series on $H$ satisfy (\ref{eq:def-af-G-2})
under the assumption $\kappa> 6$.
In Section \ref{sec:commutation-relations},
 we give another proof of Theorem \ref{th:holomorphy-L} by using
an Archimedean 
analogue of Eichler commutation relations
under the assumption $\kappa>8$.

We call $\cL(f)$ the \emph{Arakawa lift} of $f$.
The arithmetic properties of the Arakawa lifting $\cL$
(in particular the Fourier expansion of $\cL(f)$) have been 
studied in \cite{MN1}, \cite{MN2} and \cite{MN3}.
In \cite{Na4} and \cite{Na5}, the second named author 
treated the case where the involved representation 
at the Archimedean place belongs to a more general discrete 
series.
He also studied a relation between 
the theta lifting for $(GO^*(4),\mathrm{GSp}(1,1))$
(or $(GO^*(4),\mathrm{GSp}^*(2))$)
and  that for $(GO(2,2),\mathrm{GSp}(2))$
in terms of the Jacquet-Langlands-Shimizu correspondence.

 \section{The adjoint of the Arakawa lifting}
\label{sec:adjoint-AL}  

 \subsection{}

Define $\varphi^*_0\in W$ and $\phi^*_{0}\in W'$ by
\begin{align*}
 \varphi^*_0 &=\varphi^*_{0,\infty}\otimes\bigotimes_{p<\infty}\varphi_{0,p}\\
 \phi^*_0 &=\phi^*_{0,\infty}\otimes\bigotimes_{p<\infty}\phi_{0,p},
\end{align*}
where
\begin{align*}
 \varphi^*_{0,\infty}(X) &=\e\left(\sqrt{-1}\tp\ol{X}X\right)
\sigma_{\kappa}((1,\sqrt{-1})X)\qquad(X\in\bH^{(2,1)}),\\
 \phi^*_{0,\infty}(X') &=\e\left(\sqrt{-1}X'\,\tp\ol{X'}\right)
\sigma_{\kappa}((1,-1)\tp X')\qquad(X'\in\bH^{(1,2)}).
\end{align*}

The proofs of the following results are straightforward and omitted.

\begin{lemma}
\label{lem:*0}
 \begin{itemize}
  \item[(1)]\ We have $\phi^*_0=I\varphi^*_0$.
  \item[(2)] \ For $v,v'\in V_{\kappa}$, we have
\begin{align*}
 ( \varphi_0(X)v,v')_{\kappa}&=( v,\varphi^*_0(X)v')_{\kappa}\qquad(X\in B_{\bA}^{(2,1)}),\\
 ( \phi_0(x_1,x_2)v,v')_{\kappa}&=( v,\phi^*_0(x_1, -x_2)v')_{\kappa}\qquad(x_1,x_2\in B_{\bA}).
\end{align*}
  \item[(3)] \ For $u_{\infty}\in U_{\infty}$ and $k_{\infty}\in K_{\infty}$,
we have
\begin{align*}
 \left(r(u_{\infty},k_{\infty})\varphi^*_0\right)(X)&=
\rho_{\kappa}(u_{\infty})\inv\circ \varphi^*_0(X)\circ 
\tau_{\kappa}^-(k_{\infty})\qquad(X\in B_{\bA}^{(2,1)}),\\
\left(r'(u_{\infty},k_{\infty})\phi^*_0\right)(X')&=
\rho_{\kappa}(u_{\infty})\inv\circ \phi^*_0(X')\circ 
\tau_{\kappa}^-(k_{\infty})\qquad(X'\in B_{\bA}^{(1,2)}).
\end{align*}
 \end{itemize}
\end{lemma}

\begin{lemma}
\label{lem:adjoint-formula}
 We have
\[
 (r(h,g)\varphi_0(X)v,v')_{\kappa}
=(v,r(h,g^-)\varphi^*_0(X)v')_{\kappa}
\qquad(h\in H_{\bA},g\in G_{\bA},X\in B_{\bA}^{(2,1)},v,v'\in V_{\kappa}),
\]
where
\[
 g^-=\Nmat{a&-b}{-c&d}\qquad \text{ for }\ g=\Nmat{a&b}{c&d}.
\]
\end{lemma}

 \subsection{}

Set
\[
\theta^*(h,g) =\sum_{X\in B^{(2,1)}}\,\left(r(h,g^-)\varphi^*_0\right)(X)\in
\mathrm{End}(V_{\kappa})\qquad (h\in H_{\bA},g\in G_{\bA}).
\]
By Poisson summation formula, we have
\[
 \theta^*(h,g) =\sum_{X'\in B^{(1,2)}}\,\left(r'(h,g^-)\phi^*_0\right)(X').
\]
The following facts are direct consequences of Lemma \ref{lem:*0} (3)
and Lemma \ref{lem:adjoint-formula}.

\begin{lemma}
\label{lem:adjoint-theta}
\begin{enumerate}
 \item  We have
\[
 \theta^*\left(\gamma hu_fu_{\infty},\gamma'gk_fk_{\infty}\right)
=\rho_{\kappa}(u_{\infty})\inv\circ \theta^*(h,g)\circ \tau_{\kappa}(k_{\infty})
\]
for $\gamma\in H_{\bQ},h\in H_{\bA}, u_f\in U_f, u_{\infty}\in U_{\infty},
\gamma'\in G_{\bQ},g\in G_{\bA}, k_f\in K_f, k_{\infty}\in K_{\infty}
$.
 \item We have
 \[
 (\theta(h,g)v,v')_{\kappa}
=(v,\theta^*(h,g)v')_{\kappa}
\qquad(h\in H_{\bA},g\in G_{\bA},v,v'\in V_{\kappa}).
\]
\end{enumerate}

\end{lemma}

\bigskip

For $F\in \cS_{\kappa}^G$, we define
\begin{equation}
 \label{eq:def-L^*}
 \cL^*(F)(h)=\int_{G_{\bQ}\backslash G_{\bA}}
\theta^*(h,g)\,F(g)dg
\qquad(h\in H_{\bA}).
\end{equation}

\begin{theorem}
\label{th:hol-L*}
Suppose that $\kappa>8$.
 \begin{itemize}
  \item[(1)] \ For $F\in \cS_{\kappa}^G$,
we have $\cL^*(F)\in \cS_{\kappa}^H$.
  \item[(2)] \ We have
\[
 \langle \cL(f), F\rangle_{G}=
\langle f,\cL^*(F)\rangle_{H}
\]
for $f\in \cS_{\kappa}^H$ and  $F\in \cS_{\kappa}^G$.
  \item[(3)] \ For $F\in \cS_{\kappa}^G$,
\begin{equation}
 \label{eq:reproduce-f}
 \cL(\cL^*(F))(g)=
\int_{H_{\bQ}\backslash H_{\bA}}dh\,
\int_{G_{\bQ}\backslash G_{\bA}}dy\,
\theta(h,g)\circ \theta^*(h,y)\,F(y).
\end{equation}
 \end{itemize}
\end{theorem}

The proof of (1) is given in the next section.
The second assertion follows from Lemma \ref{lem:adjoint-theta}.
The third assertion is easily verified.

In a forthcoming paper, we will investigate $\cL(\cL^*(F))$
by using (\ref{eq:reproduce-f}).

 \section{An Archimedean analogue of 
Eichler commutation relations}
\label{sec:commutation-relations}

 \subsection{}

The object of this section is to prove Theorem \ref{th:holomorphy-L}
and Theorem \ref{th:hol-L*} (1) by using the following 
Archimedean analogue of Eichler commutation relations:

\begin{theorem}
 \label{th:commutation-relations}
Suppose that $\kappa>8$.
For $(h,g)\in H_{\bA}\times G_{\bA}$, we have
\begin{equation}
\label{eq:ecr}
 c_{\kappa}^H\,\int_{H_{\infty}}\,\theta(hx,g)\omega_{\kappa}(x)dx
=  c_{\kappa}^G\,\int_{G_{\infty}}\,\Omega_{\kappa}(y)\theta(h,gy\inv)dy
\end{equation}
and
\begin{equation}
\label{eq:ecr*}
 c_{\kappa}^H\,\int_{H_{\infty}}\,\omega_{\kappa}(x)\theta^*(hx\inv,g)dx
=  c_{\kappa}^G\,\int_{G_{\infty}}\,\theta^*(h,gy)\Omega_{\kappa}(y)dy,
\end{equation}
where the Haar measures $dx$ and $dy$ on $H_{\infty}$ and
$G_{\infty}$ are given by (\ref{eq:measure-H}) and (\ref{eq:measure-G}),
respectively.

\end{theorem}

 \subsection{}

We first prove Theorem \ref{th:holomorphy-L} assuming (\ref{eq:ecr}).
We can show Theorem \ref{th:hol-L*} (1) assuming (\ref{eq:ecr*}) 
in a similar manner and omit its proof.
Throughout this section, we fix an integer $\kappa>4$.
To simplify the notation, we write $\omega$ and $\Omega$
for $\omega_{\kappa}$ and $\Omega_{\kappa}$, respectively.

Let $f\in S_{\kappa}^H$ and put $F=\cL(f)$. The condition (\ref{eq:def-af-G-1}) for $F=\cL(f)$ is
easily verified.
We have
\begin{align*}
\cL(f)(g)
&= \int_{H_{\bQ}\backslash H_{\bA}}\,\theta(h,g)f(h)dh\\
 &=\int_{H_{\bQ}\backslash H_{\bA}}\,\theta(h,g)
\left(c_{\kappa}^H\,\int_{H_{\infty}}\omega(x)f(hx\inv)dx\right)
dh\\
 &=\int_{H_{\bQ}\backslash H_{\bA}}
\left(c_{\kappa}^H\,\int_{H_{\infty}}\,\theta(hx,g)\omega(x)dx\right)
f(h)dh\\
&=\int_{H_{\bQ}\backslash H_{\bA}}
\left(c_{\kappa}^G\,\int_{G_{\infty}}\,\Omega(y)\theta(h,gy\inv)dy\right)
f(h)dh\\
&=c_{\kappa}^G\,\int_{G_{\infty}}\Omega(y)
  \left(\int_{H_{\bQ}\backslash H_{\bA}}\,\theta(h,gy\inv)f(h)dh\right)dy\\
 &=c_{\kappa} ^G\,\int_{G_{\infty}}\,\Omega(y)\cL(f)(gy\inv)dy,
\end{align*}
proving (\ref{eq:def-af-G-2}) for $F=\cL(f)$.

To prove the boundedness of $F$, it suffices to show that
\[
 F_0(g):=\int_{B^-\backslash B^-_{\bA}}\,
F\left(\Nmat{1&\beta}{0&1}g\right) d\beta
\]
vanishes, namely $F$ is cuspidal.
We have
\begin{align*}
F_0(g)&=\int_{B^-\backslash B^-_{\bA}}\,\left(
\int_{H_{\bQ}\backslash H_{\bA}}\,
\sum_{X\in B^{(2,1)}}\,
r\left(h,\Nmat{1&\beta}{0&1}g\right)\varphi_0(X)f(h)dh
\right) d\beta\\
 & = \int_{H_{\bQ}\backslash H_{\bA}}\,
      \sum_{X\in B^{(2,1)}}\,
      \left(\int_{B^-\backslash B^-_{\bA}}\,
      \psi\left(-\dfrac{1}{2}\tr(\beta \tp\ol{X}JX)\right)d\beta\right)
  r(h,g)\varphi_0(X)f(h)dh\\
 &=\int_{H_{\bQ}\backslash H_{\bA}}\,
     \sum_{X\in B^{(2,1)},\,\tp\ol{X}JX=0}\,
    r(1,g)\varphi_0(\tp \ol{h}X)f(h)dh.
\end{align*}
Observe that
\[
 \left\{X\in B^{(2,1)}\mid \tp\ol{X}JX=0\right\}=
 \left\{\vecm{0\\0}\right\}\cup
  \left\{\tp\ol{\gamma}\vecm{0\\1}\mid \gamma\in N^H_{\bQ}
\backslash H_{\bQ}\right\},
\]
where 
\[
 N^H_{\bQ}=\left\{\Nmat{1&b}{0&1}\mid b\in\bQ\right\}.
\]
It follows that
\begin{align*}
 & F_0(g)\\
&=(r(1,g)\varphi_0)(0)\int_{H_{\bQ}\backslash H_{\bA}}\,f(h)dh
+\int_{N^H_{\bA}\backslash H_{\bA}}\,
\left(r(1,g)\varphi_0\right)\left(\tp\ol{h}\vecm{0\\1}\right)
\left(\int_{\bQ\backslash \bQ_{\bA}}\,f\left(\Nmat{1&b}{0&1}h\right)db
\right)dh\\
 & =0,
\end{align*}
since $f$ is cuspidal.

 \subsection{}

To prove Theorem \ref{th:commutation-relations}, 
we need some preparations.
Define $\Phi_{\kappa}\colon \bH\cross\,\to\,\mathrm{End}(V_{\kappa})$
by
\[
 \Phi_{\kappa}(x)=N(x)\inv \sigma_{\kappa}(x)\inv\qquad(x\in\bH\cross).
\]
Note that $\Omega(g)=\Phi_{\kappa}\left(2\inv(1,1)g\vecm{1\\1}\right)$ 
for $g\in G_{\infty}$.
The following is easily verified.

\begin{lemma}
 \label{lem:Phi-kappa}
\begin{enumerate}
 \item For $a\in\bR\cross$ and $x\in\bH\cross$, we have
$\Phi_{\kappa}(ax)=a^{-\kappa-2}\Phi_{\kappa}(x)$.
 \item \ For $x,y\in\bH\cross$, we have
$\Phi_{\kappa}(xy)=\Phi_{\kappa}(y)\Phi_{\kappa}(x)$.
 \item \ For $x\in \bH\cross$, we have $\sigma_{\kappa}(\ol{x})
=N(x)^{\kappa+1}\Phi_{\kappa}(x)$.
\end{enumerate}
\end{lemma}

We define a map $M_{\kappa}\colon \bH\to \mathrm{End}(V_{\kappa})$ by
\begin{equation}
 \label{eq:M-kappa}
 M_{\kappa}(\xi)=\int_{\bH^-}\,\e(-\tr(\xi\beta))\Phi_{\kappa}(1+\beta)
d_L\beta\qquad(\xi\in\bH).
\end{equation}
It is easily verified that
$M_{\kappa}(\xi)=M_{\kappa}(x)$ for $\xi=t+x\;\;(t\in\bR,x\in\bH^-)$
and that
\[
 M_{\kappa}(\alpha\xi\alpha\inv)=\sigma_{\kappa}(\alpha) M_{\kappa}(\xi)
\sigma_{\kappa}(\alpha)\inv
\qquad (\alpha\in \bH\cross,\xi\in\bH).
\]
For $\xi\in\bH$, there exists a $u\in\bH^1$ such that
$\xi=u(s+it)u\inv\;\;(s,t\in\bR,t\ge 0)$. Then $M_{\kappa}(\xi)=
\sigma_{\kappa}(u) M_{\kappa}(it)\sigma_{\kappa}(u)\inv
$.
The following fact is proved in \cite{Ar1}, Lemma 1.2.
{Recall that we have introduced $P_r(X,Y)=X^rY^{\kappa-r}\in V_{\kappa}$ at (\ref{eq:Pr}).

\begin{proposition}
 \label{prop:M-kappa}
For $t\ge 0$, we have
\[
  M_{\kappa}(it)P_r=\delta_{r\kappa}\,\lambda(\kappa)\,t^{\kappa-1}
e^{-4\pi t}\,P_{\kappa}\qquad(0\le r\le \kappa),
\]
where
\[
 \lambda(\kappa)=\dfrac{2^{2\kappa-1}\pi^{\kappa+1}}{\kappa!}.
\]
In particular, $M_{\kappa}(0)=0$.
\end{proposition}

\begin{lemma}
 \label{lem:diagonal}
For $a,b\in\bC$, we have
\begin{equation}
 \label{eq:diagonal}
\sigma_{\kappa}(a+ib)\,P_r=
(a+b\sqrt{-1})^{r}(a-b\sqrt{-1})^{\kappa-r}\,P_r
\qquad(0\le r\le \kappa).
\end{equation}
\end{lemma}

\begin{proof}
 This follows from the definition (\ref{eq:sigma}) and
\[ 
 A(a+bi)=\Nmat{a+b\sqrt{-1}&0}{0&a-b\sqrt{-1}}
\]
(for the definition of $A$, see (\ref{eq:def-A})).
\end{proof}

The following formula is well-known.
\begin{lemma}
 \label{lem:integral-residue}
For $s,t\in\bR$ with $t>0$, we have
\[
 \int_{\bR}(t+\sqrt{-1}x)^{-\kappa}\,e(sx)dx=
\begin{cases}
 \dfrac{(2\pi)^{\kappa}}{(\kappa-1)!}\,s^{\kappa-1}\e(\sqrt{-1}st)
& (s>0),\\
 0&  (s\le 0).
\end{cases}
\]
\end{lemma}

 \subsection{}

Define 
$\phi_{\omega,\infty},\phi_{\omega,\infty}^*\colon 
\bH^{(1,2)}\to \mathrm{End}(V_{\kappa})$
and
$\varphi_{\Omega,\infty},\varphi_{\Omega,\infty}^*\colon 
\bH^{(2,1)}\to \mathrm{End}(V_{\kappa})$
by
\begin{align*}
 \phi_{\omega,\infty}(X')&
  =\int_{H_{\infty}}\,r'(h,1)\phi_{0,\infty}(X')\,\omega(h)dh
\\
\phi^*_{\omega,\infty}(X')&
=\int_{H_{\infty}}\,\omega(h)\,r'(h\inv,1)\phi^*_{0,\infty}(X')dh,
\\
 \varphi_{\Omega,\infty}(X)
&=\int_{G_{\infty}}\,\Omega(g)\,r(1,g\inv)\varphi_{0,\infty}(X)dg
\\
 \varphi^*_{\Omega,\infty}(X)&
=\int_{G_{\infty}}\,r(1,g^-)\varphi^*_{0,\infty}(X)\,\Omega(g)dg
\end{align*}
for $X'\in \bH^{(1,2)}$ and $X\in \bH^{(2,1)}$.
The following facts are easily verified.

\begin{lemma}
 \label{lem:action-of-H^1}
\begin{enumerate}
 \item We have
\[
\phi_{\omega,\infty}(0,0)=\phi_{\omega,\infty}^*(0,0)=0,\
 \varphi_{\Omega,\infty}\vecm{0\\0}= \varphi_{\Omega,\infty}^*\vecm{0\\0}=0.
\]
\item
For $\alpha\in\bH^1, X\in\bH^{(2,1)}$ and $X'\in\bH^{(1,2)}$, we have
\begin{align*}
\phi_{\omega,\infty}(X'\alpha)&
 =\sigma_{\kappa}(\alpha)\inv\,\phi_{\omega,\infty}(X'),\\
 I_{\infty}\inv\phi_{\omega,\infty}(X\alpha)&
 =\sigma_{\kappa}(\alpha)\inv\,I_{\infty}\inv\phi_{\omega,\infty}(X),\\
 \varphi_{\Omega,\infty}(X\alpha)&
   =\sigma_{\kappa}(\alpha)\inv\,\varphi_{\Omega,\infty}(X),\\
\phi_{\omega,\infty}^*(X'\alpha)&
 =\phi_{\omega,\infty}^*(X')\,\sigma_{\kappa}(\alpha),\\
 I_{\infty}\inv\phi_{\omega,\infty}^*(X\alpha)&
 =I_{\infty}\inv\phi_{\omega,\infty}^*(X)
\sigma_{\kappa}(\alpha),\\
 \varphi_{\Omega,\infty}^*(X\alpha)&
   =\varphi_{\Omega,\infty}(X)
\sigma_{\kappa}(\alpha).
\end{align*}
\end{enumerate}
\end{lemma}

One of the keys to the proof 
of Theorem \ref{th:commutation-relations} 
is the following results.

\begin{proposition}
 \label{prop:ecr-testf}
Let $x_1,x_2\in\bH$.
\begin{enumerate}
 \item We have
\begin{align*}
 \phi_{\omega,\infty}(x_1,x_2)
=2^{\kappa-2}\pi\inv \,\kappa\cdot\delta(\tr(\ol{x_1}x_2)>0)\,
(\tr(\ol{x_1}x_2))^{\kappa-1}\,
\e(\sqrt{-1}\tr(\ol{x_1}x_2))
\Phi_{\kappa}(x_1+x_2).
\end{align*}
 \item Suppose that $\ol{x_1}x_2=s-it\;\;(s,t\in\bR,t\ge 0)$.
Then we have
\[
 \varphi_{\Omega,\infty}\vecm{x_1\\x_2}\,
\sigma_{\kappa}(x_1+\sqrt{-1}x_2) P_r
=
\delta_{r\kappa}\cdot 2^{2\kappa}\pi\,\kappa\inv\cdot
t^{\kappa-1}e^{-4\pi t}
\cdot P_{\kappa}\qquad(0\le r\le \kappa).
\]
 \item We have
\[
 I_{\infty}\inv \phi_{\omega,\infty}\vecm{x_1\\x_2}=
\dfrac{c_{\kappa}^G}{c_{\kappa}^H}\,\varphi_{\Omega,\infty}\vecm{x_1\\x_2}.
\]
\end{enumerate}
\end{proposition}

\begin{proposition}
 \label{prop:ecr-testf-*}
Let $x_1,x_2\in\bH$.
\begin{enumerate}
 \item We have
\begin{align*}
 \phi_{\omega,\infty}^*(x_1,x_2)
=2^{\kappa-2}\pi\inv \,\kappa\cdot\delta(\tr(\ol{x_1}x_2)<0)
\left(-\tr(\ol{x_1}x_2)\right)^{\kappa-1} \e(-\sqrt{-1}\tr(\ol{x_1}x_2))
\Phi_{\kappa}(\ol{x_1-x_2}).
\end{align*}
 \item Suppose that $\ol{x_1}x_2=s-it\;\;(s,t\in\bR,t\ge 0)$.
If $t=0$, we have $\varphi_{\Omega}^*\vecm{x_1\\x_2}=0$.
If $t>0$, we have
\begin{align*}
& \varphi_{\Omega,\infty}^*\vecm{x_1\\x_2}\,P_r\\
&=
\delta_{r\kappa}\cdot 2^{2\kappa}\pi\,\kappa\inv\cdot
t^{\kappa-1}e^{-4\pi t}\left(\ol{x_1}x_1+\ol{x_2}x_2+2t\right)^{-\kappa}
\cdot \sigma_{\kappa}(x_1+\sqrt{-1}x_2)\,P_{\kappa}\qquad(0\le r\le \kappa).
\end{align*}
 \item We have
\[
 I_{\infty}\inv \phi_{\omega,\infty}^*\vecm{x_1\\x_2}=
\dfrac{c_{\kappa}^G}{c_{\kappa}^H}\,
\varphi_{\Omega,\infty}^*\vecm{x_1\\x_2}.
\]
\end{enumerate}
\end{proposition}

The proof of Proposition \ref{prop:ecr-testf} is given
 in 
Sections \ref{subsec:proof-erc-testf-1}--\ref{subsec:proof-erc-testf-3}.
We omit the proof of Proposition \ref{prop:ecr-testf-*},
since it is quite similar to that of Proposition \ref{prop:ecr-testf}
 and slightly easier.

 \subsection{}
\label{subsec:proof-erc-testf-1}

In this subsection, we prove
Proposition \ref{prop:ecr-testf} (1).
Let $x_1,x_2\in\bH$. 
By (\ref{eq:U-equivariance-omega}) and Lemma \ref{lem:phi0-UK},
we have
\begin{align*}
 & \phi_{\omega,\infty}(x_1,x_2)\\
 &= \int_{H_{\infty}}\,r'(h,1)\phi_{0,\infty}(x_1,x_2)\omega(h)dh\\
 & =\int_{\bR}db\,\int_0^{\infty}a^{-2}d\cross a\,
\int_{U_{\infty}}du
\,r'\left(\Nmat{1&b}{0&1}\Nmat{a&0}{0&a\inv}u,1\right)\phi_{0,\infty}
(x_1,x_2)\,\omega\left(\Nmat{1&b}{0&1}\Nmat{a&0}{0&a\inv}u\right)\\
 &= \int_{\bR}db\,\int_0^{\infty}a^{-2}d\cross a\,
  \e(-b\, \tr(\ol{x_1}x_2))\,a^4\,\phi_{0,\infty}(ax_1,ax_2)\,
\left(\dfrac{a+a\inv-\sqrt{-1}a\inv b}{2}\right)^{-\kappa}\\
 &=2^{\kappa} \sigma_{\kappa}(\ol{x_1+x_2})\,
\int_0^{\infty}\,a^{2\kappa+2}\e(\sqrt{-1}a^2(\ol{x_1}x_1+\ol{x_2}x_2))
   d\cross a
  \int_{\bR}\,\left(a^2+1-\sqrt{-1}b\right)^{-\kappa}
\e(-b\,\tr(\ol{x_1}x_2))db.
\end{align*}
By Lemma \ref{lem:integral-residue}, we see that the integral over $\bR$
is equal to
\[
 \dfrac{(2\pi)^{\kappa}}{(\kappa-1)!}\,\delta(\tr(\ol{x_1}x_2)>0)
\,(\tr(\ol{x_1}x_2))^{\kappa-1}\,\e(\sqrt{-1}(a^2+1)\tr(\ol{x_1}x_2)).
\]
Thus we obtain
\begin{align*}
  \phi_{\omega,\infty}(x_1,x_2)
 & =\dfrac{(4\pi)^{\kappa}}{(\kappa-1)!}\,\delta(\tr(\ol{x_1}x_2)>0)
\,\sigma_{\kappa}(\ol{x_1+x_2})\,(\tr(\ol{x_1}x_2))^{\kappa-1}\,
\e(\sqrt{-1}\,\tr(\ol{x_1}x_2))\\
&\quad\,\int_0^{\infty}\,a^{2\kappa+2}
\,\e(\sqrt{-1}a^2(\ol{x_1}x_1+\ol{x_2}x_2+\tr(\ol{x_1}x_2)))d\cross a.
\end{align*}
Since the last integral is equal to
\begin{align*}
 \dfrac{1}{2}\,\int_0^{\infty}\,
a^{\kappa+1}\e(\sqrt{-1}N(x_1+x_2)a)d\cross a
=\dfrac{\kappa!}{2}(2\pi N(x_1+x_2))^{-\kappa-1},
\end{align*}
we have
\begin{align*}
 \phi_{\omega,\infty}(x_1,x_2)
=2^{\kappa-2}\pi\inv \,\kappa\cdot\delta(\tr(\ol{x_1}x_2)>0)
\,(\tr(\ol{x_1}x_2))^{\kappa-1}\, \e(\sqrt{-1}\tr(\ol{x_1}x_2))
\Phi_{\kappa}(x_1+x_2),
\end{align*}
which completes the proof of \ref{prop:ecr-testf} (1).
Here we used Lemma \ref{lem:Phi-kappa} (3).
\subsection{}
\label{subsec:proof-erc-testf-2}

In this subsection, we prove
Proposition \ref{prop:ecr-testf} (2).

\begin{lemma}
 \label{lem:8-6-1}
Let $x_1,x_2\in\bH$. We have
\begin{align*}
 \varphi_{\Omega,\infty}\vecm{x_1\\x_2}&
=2^{\kappa+1}\int_0^{\infty}\,a^{\kappa}(a+1)^{1-\kappa}
  \e(\sqrt{-1}(N(x_1)+N(x_2))a)\,M_{\kappa}(-(a+1)\ol{x_1}x_2)
d\cross a\\
&\qquad \cdot\sigma_{\kappa}(\ol{x_1}-\sqrt{-1}\ol{x_2}).
\end{align*}
\end{lemma}

\begin{proof}
 By (\ref{eq:K-equivariance-Omega}) and
 Lemma \ref{lem:phi0-UK}, we have
\begin{align*}
 \varphi_{\Omega,\infty}\vecm{x_1\\x_2}
&=\int_{K_{\infty}}dk\,\int_0^{\infty}a^6d\cross a\,\int_{\bH^-}d_L\beta\,
   \Omega\left(k\Nmat{a&0}{0&a\inv}\Nmat{1&\beta}{0&1}\right)\\
 &\qquad r\left(1,\Nmat{1&-\beta}{0&1}\Nmat{a\inv&0}{0&a}k\inv\right)
    \,\varphi_{0,\infty}\vecm{x_1\\x_2}\\
 &=\int_0^{\infty}a^6d\cross a\,\int_{\bH^-}d_L\beta\,
\Phi_{\kappa}\left(\dfrac{a}{2}(1+a^{-2}+\beta)\right)\,
 \e(\tr(\beta\ol{x_1}x_2))a^{-4}\,\varphi_{0,\infty}\vecm{a\inv x_1\\a\inv x_2}\\
 &=\int_0^{\infty}a^{-2}d\cross a\,\int_{\bH^-}d_L\beta\,
\left(\dfrac{a\inv}{2}\right)^{-\kappa-2}
\Phi_{\kappa}\left(1+a^2+\beta\right)\,
 \e(\tr(\beta\ol{x_1}x_2))\,\varphi_{0,\infty}\vecm{a x_1\\a x_2}\\
&=2^{\kappa+2}\int_0^{\infty}a^{2\kappa}d\cross a\,\int_{\bH^-}d_L\beta\,
    \Phi_{\kappa}\left(1+a^2+\beta\right)\,
  \e\left(\tr(\beta\ol{x_1}x_2)+\sqrt{-1}a^2(N(x_1)+N(x_2))\right)
\\
&\qquad \cdot\sigma_{\kappa}(\ol{x_1}-\sqrt{-1}\ol{x_2}).
\end{align*}
Changing the variable $\beta$ into $(1+a^2)\beta$, we obtain
\begin{align*}
\varphi_{\Omega,\infty}\vecm{x_1\\x_2}
&=2^{\kappa+2} \int_0^{\infty}
  a^{2\kappa}(a^2+1)^{1-\kappa}
  \e\left(\sqrt{-1}a^2(N(x_1)+N(x_2)\right)
d\cross a\,\\
&\qquad \int_{\bH^-}\,\e(\tr((a^2+1)\beta\ol{x_1}x_2)))
\Phi_{\kappa}(1+\beta)d_L\beta\,
 \cdot\sigma_{\kappa}(\ol{x_1}-\sqrt{-1}\ol{x_2})\\
 &=2^{\kappa+1}
\int_0^{\infty}
  a^{\kappa}(a+1)^{1-\kappa}
  \e\left(\sqrt{-1}(N(x_1)+N(x_2))a\right)
M_{\kappa}(-(a+1) \ol{x_1}x_2)
d\cross a\\
&\qquad  \cdot\sigma_{\kappa}(\ol{x_1}-\sqrt{-1}\ol{x_2})
\end{align*}
and the lemma has been proved.
\end{proof}

We now suppose that $\ol{x_1}x_2=s-it\;\;(s,t\in\bR,t\ge 0)$.
Since
\begin{align*}
 (\ol{x_1}-\sqrt{-1}\ol{x_2})(x_1+\sqrt{-1}x_2)&
=N(x_1)+N(x_2)+\sqrt{-1}(\ol{x_1}x_2-\ol{x_2}x_1)\\
&=N(x_1)+N(x_2)-2\sqrt{-1}it,
\end{align*}
we have
\[
 \sigma_{\kappa}(\ol{x_1}-\sqrt{-1}\ol{x_2})
\,\sigma_{\kappa}(x_1+\sqrt{-1}x_2)\,P_r
=(N(x_1)+N(x_2)+2t)^r\,(N(x_1)+N(x_2)-2t)^{\kappa-r}\,P_r
\]
by Lemma \ref{lem:diagonal}.
We also have
\[
 M_{\kappa}(-(a+1)\ol{x_1}x_2)P_r=
M_{\kappa}(i(a+1)t)P_r=\delta_{r\kappa}\,\lambda(\kappa)(a+1)^{\kappa-1}
t^{\kappa-1}e^{-4\pi(a+1)t}P_{\kappa}
\qquad(0\le r\le \kappa)
\]
by Proposition \ref{prop:M-kappa}.
We thus have
\begin{align*}
& \varphi_{\Omega,\infty}\vecm{x_1\\x_2}\,
\sigma_{\kappa}(x_1+\sqrt{-1}x_2)\,P_r\\
&=
\delta_{r\kappa}\,
2^{\kappa+1}\lambda(\kappa)\,
(N(x_1)+N(x_2)+2t)^{\kappa}
\, t^{\kappa-1}e^{-4\pi t}\,
\int_0^{\infty}\,a^{\kappa}e^{-2\pi(N(x_1)+N(x_2)+2t)a}d\cross a
\,P_{\kappa}\\
 &=\delta_{r\kappa}\, 2^{2\kappa}\pi\kappa\inv\, t^{\kappa-1}e^{-4\pi t}\, P_{\kappa},
\end{align*}
which completes the proof of Proposition \ref{prop:ecr-testf} (2).

\subsection{}
\label{subsec:proof-erc-testf-3}

In this subsection, we prove
Proposition \ref{prop:ecr-testf} (3).
Observe that $\varphi_{\Omega,\infty}\vecm{0\\x_2}=0$ and 
$I_{\infty}\inv\phi_{\omega,\infty}\vecm{0\\x_2}=0$
in view of Proposition \ref{prop:ecr-testf} (1), (2)
and Lemma \ref{lem:action-of-H^1} (1).
We hereafter assume that $x_1\ne 0$.

\begin{lemma}
\label{lem:8-7-1}
We have
\begin{align*}
& I_{\infty}\inv\phi_{\omega,\infty}\vecm{x_1\\x_2}\\
&=2^{2\kappa-1}\pi\inv\kappa\,
\int_0^{\infty}\,a^{\kappa}(a+1)^{1-\kappa}
\,\e(\tr(-\ol{x_1}x_2+\sqrt{-1}N(x_1))a)\,
M_{\kappa}((a+1)\ol{x_2}x_1)\,d\cross a\,\cdot\, 
\sigma_{\kappa}(\ol{x_1}).
\end{align*}
 
\end{lemma}

\begin{proof}
 We put $c=2^{\kappa}\pi\inv \kappa$.
By Proposition \ref{prop:ecr-testf} (1), we have
\begin{align*}
 & I_{\infty}\inv\phi_{\omega,\infty}\vecm{x_1\\x_2}\\
 &=4\int_{\bH}\,\e(-\tr(\ol{y}x_2)) \phi_{\omega}(x_1,y)d_Ly\\
 & =c\,\int_{\bH}\,\delta(\tr(\ol{x_1}y)>0)\,(\tr(\ol{x_1}y))^{\kappa-1}
  \,\e(-\tr(\ol{x_2}y)+\sqrt{-1}\,\tr(\ol{x_1}y))\Phi_{\kappa}(x_1+y)d_Ly. 
\end{align*}
Changing the variable $y$ into $x_1y$ 
and using Lemma \ref{lem:Phi-kappa} (2), 
we obtain
\begin{align*}
 & I_{\infty}\inv\phi_{\omega,\infty}\vecm{x_1\\x_2}\\
 & =c\,\int_{\bH}\,\delta(\tr(y)>0)\,
    (\tr(y))^{\kappa-1}\,\e(\tr(-\ol{x_2}x_1y+\sqrt{-1}N(x_1)y))
    \Phi_{\kappa}(1+y)d_Ly\cdot N(x_1)^{\kappa+1}\Phi_{\kappa}(x_1)\\
 &=c\, 
  \int_{\bR}da\,\int_{\bH^-}d_L\beta\,
   \delta(\tr(a+\beta)>0)\,
    (\tr(a+\beta))^{\kappa-1}\,
\e(\tr((-\ol{x_2}x_1+\sqrt{-1}N(x_1))(a+\beta)))
\Phi_{\kappa}(1+a+\beta)\\
 &\qquad 
  \cdot  \sigma_{\kappa}(\ol{x_1})\\
 &=c\, 
  \int_{0}^{\infty}
       (2a)^{\kappa-1}\,\e(\tr((-\ol{x_2}x_1+\sqrt{-1}N(x_1)))a)da\,
\int_{\bH^-}\,\e(-\tr(\ol{x_2}x_1\beta))\,\Phi_{\kappa}(1+a+\beta)d_L\beta
\,\cdot\, \sigma_{\kappa}(\ol{x_1}).
\end{align*}
Changing the variable $\beta$ into $(1+a)\beta$, we obtain
\begin{align*}
&  I_{\infty}\inv\phi_{\omega,\infty}\vecm{x_1\\x_2}\\
&=c\cdot 2^{\kappa-1}\int_0^{\infty}\,a^{\kappa-1}(a+1)^{1-\kappa}
  \e(\tr(-\ol{x_2}x_1+\sqrt{-1}N(x_1))a)da\,\\
 & \qquad \int_{\bH^-}\,\e(-\tr((a+1)\ol{x_2}x_1\beta))\,\Phi_{\kappa}(1+\beta)d_L\beta\,
 \cdot\,\sigma_{\kappa}(\ol{x_1})\\
 &=2^{2\kappa-1}\pi\inv\kappa
\int_0^{\infty}\,a^{\kappa}(a+1)^{1-\kappa}
\,\e(\tr(-\ol{x_1}x_2+\sqrt{-1}N(x_1))a)\,
M_{\kappa}((a+1)\ol{x_2}x_1)\,d\cross a\,\cdot\,\sigma_{\kappa}(\ol{x_1})
\end{align*}
and we are done.
\end{proof}

In view of Lemma \ref{lem:action-of-H^1},
it  suffices to show the following lemma to complete the proof of
Proposition \ref{prop:ecr-testf} (3).

\begin{lemma}
 Suppose that
$\ol{x_1}x_2=s-it\;\;(s,t\in\bR,\,t\ge 0)$.
Then
\begin{align*}
 I_{\infty}\inv \phi_{\omega,\infty}\vecm{x_1\\x_2}\,
\sigma_{\kappa}(x_1+\sqrt{-1}x_2)\, P_r
=
\delta_{r\kappa}\cdot 2^{2\kappa-2}t^{\kappa-1}e^{-4\pi t}
\, P_{\kappa}
\quad(r=0,\ldots,\kappa).
\end{align*}
\end{lemma}

\begin{proof}
 Put $c'=2^{2\kappa-1}\pi\inv\kappa$. 
By Lemma \ref{lem:8-7-1}, Lemma \ref{lem:diagonal} and 
Proposition \ref{prop:M-kappa},
we have
\begin{align*}
 & I_{\infty}\inv \phi_{\omega,\infty}\vecm{x_1\\x_2}\,
\sigma_{\kappa}(x_1+\sqrt{-1}x_2)\, P_r\\
&=c'
\int_0^{\infty}\,a^{\kappa}(a+1)^{1-\kappa}
\,\e(\tr(-\ol{x_1}x_2+\sqrt{-1}N(x_1)a)\,
M_{\kappa}((a+1)\ol{x_2}x_1)\,d\cross a\\
 &\qquad\cdot \,\sigma_{\kappa}(N(x_1)+\sqrt{-1}\ol{x_1}x_2)\,P_r\\
 &=c'
\int_0^{\infty}\,a^{\kappa}(a+1)^{1-\kappa}
\,\e(2(-s+\sqrt{-1}N(x_1))a)\,
M_{\kappa}(i(a+1)t)\,d\cross a\\
 &\qquad\cdot \,(N(x_1)+t+\sqrt{-1}s)^{r}(N(x_1)-t+\sqrt{-1}s)^{\kappa-r}
 P_r\\
&=\delta_{r\kappa}\cdot c'\,\lambda(\kappa)\cdot 
  (N(x_1)+t+\sqrt{-1}s)^{\kappa}\\
&\qquad    \int_{0}^{\infty}\,a^{\kappa}(a+1)^{1-\kappa}
  ((a+1)t)^{\kappa-1}\,e^{4\pi\sqrt{-1}(-sa+\sqrt{-1}N(x_1)a)-4\pi (a+1)t}
   d\cross a\cdot P_{\kappa}\\
&=\delta_{r\kappa}\cdot\dfrac{2^{4\kappa-2}\pi^{\kappa}}{(\kappa-1)!}
   t^{\kappa-1}e^{-4\pi t}
(N(x_1)+t+\sqrt{-1}s)^{\kappa}
   \int_0^{\infty}\,a^{\kappa}e^{-4\pi(N(x_1)+t+\sqrt{-1}s)a}d\cross a\cdot P_{\kappa}\\
&=\delta_{r\kappa}\cdot 2^{2\kappa-2}t^{\kappa-1}e^{-4\pi t}\cdot P_{\kappa}.
\end{align*}

\end{proof}

 \subsection{}\label{subsec:estimate}

We define the norm of $X=(x_{ij})\in\bH^{(m,n)}$  by
\[
 \|X\|=\sqrt{\sum_{i=1}^m\sum_{j=1}^n\,N(x_{ij})}.
\]

The aim of this subsection is to prove the following estimates, which
are needed in the proof of Theorem \ref{th:commutation-relations}.

\begin{proposition}
 \label{prop:estimate}
Suppose that $\kappa>4$ and let $D$ be a compact subset of $H_{\infty}\times G_{\infty}$.
\begin{enumerate}
 \item 
There exists a positive constant $C_1$ such that
\begin{equation}
 \label{eq:estimate-phi-omega-general}
\left|\langle r'(h,g)\phi_{\omega,\infty}(X')P,P'\rangle_{\kappa}\right|
\le C_1 \|P\|_{\kappa}\|P'\|_{\kappa}\cdot (1+\|X'\|)^{-\kappa-2}
\end{equation}
holds for any $X'\in \bH^{(1,2)}$, $(h,g)\in D$ and $P,P'\in V_{\kappa}$.
 \item 
There exists a positive constant $C_2$ such that
\begin{equation}
 \label{eq:estimate-varphi-Omega-general}
\left|\langle r(h,g)\varphi_{\Omega,\infty}(X)P,P'\rangle_{\kappa}\right|
\le C_2 \|P\|_{\kappa}\|P'\|_{\kappa} \cdot (1+\|X\|)^{-\kappa}
\end{equation}
holds for any $X\in \bH^{(2,1)}$, $(h,g)\in D$ and $P,P'\in V_{\kappa}$.
\end{enumerate}

\end{proposition}

\bigskip

To prove the proposition, we need some preparations.

\begin{lemma}
 \label{lem:inequality-tr}
Let $x_1,x_2\in\bH$.
\begin{enumerate}
 \item We have $|\tr(\ol{x_1}x_2)|\le N(x_1)+N(x_2)$, and
the equality holds only when $x_2=\pm x_1$.
 \item Suppose that $\ol{x_1}x_2=s-it\;\;(s,t\in\bR,\,t\ge 0)$.
Then we have $t\le\dfrac{1}{2}(N(x_1)+N(x_2))$, 
and the equality holds only when $x_2=- x_1i$.
 \item Under the same assumption as in (2), $x_1+\sqrt{-1}x_2$
is invertible in $\bH\otimes_{\bR}\bC$ if and only if $x_2\not =- x_1i$.
In this case, $\ol{x_1}-\sqrt{-1}\ol{x_2}$ is invertible and we have
\begin{align*}
 (x_1+\sqrt{-1}x_2)\inv&=
(N(x_1)+N(x_2)-2\sqrt{-1}it)\inv 
(\ol{x_1}-\sqrt{-1}\ol{x_2}),\\
(\ol{x_1}-\sqrt{-1}\ol{x_2})\inv&=(x_1+\sqrt{-1}x_2)
(N(x_1)+N(x_2)-2\sqrt{-1}it)\inv. 
\end{align*}

\end{enumerate}
\end{lemma}

\begin{proof}
 The first assertion is an immediate consequence of the Cauchy-Schwarz
inequality and the inequality of arithmetic and geometric means. 
The second one follows from the first one and the fact that
$t=\dfrac{1}{2}\tr(\ol{x_1}x_2i)$.
The third one follows from
\[
 (\ol{x_1}-\sqrt{-1}\ol{x_2})(x_1+\sqrt{-1}x_2)=N(x_1)+N(x_2)-2\sqrt{-1}it
\]
and 
\[
 A(N(x_1)+N(x_2)-2\sqrt{-1}it)=\Nmat{N(x_1)+N(x_2)+2t&0}{0&N(x_1)+N(x_2)-2t},
\]
which is invertible in $\bH\otimes_{\bR}\bC=\mathrm{M}_2(\bC)$ if and only if
$x_2\not =- x_1i$ in view of the second assertion
(for the definition of $A$, see (\ref{eq:def-A})).
\end{proof}

\begin{lemma}
 \label{lem:varphi-Omega-anotherform}
Let $X=\tp(x_1,x_2)\in\bH^{(2,1)}$ with $\ol{x_1}x_2=s-it\;\;(s,t\in\bR,t\ge 0)$.
\begin{enumerate}
 \item If $x_2\ne -x_1i$, we have
\begin{align}
\label{eq:varphi-Omega-anotherform-nonsingular}
  \varphi_{\Omega,\infty}(X)
    \sigma_{\kappa}(\ol{x_1}-\sqrt{-1}\ol{x_2})\inv P_r
&=\delta_{r\kappa}\cdot 2^{2\kappa}\pi\kappa\inv\cdot t^{\kappa-1}
   e^{-4\pi t}(\|X\|^2+2t)^{-\kappa}\,P_{\kappa}.
\end{align}
 \item If $x_2 =-x_1i\ne 0$, we have
\begin{equation}
 \label{eq:varphi-Omega-anotherform-singular}
\varphi_{\Omega,\infty}(X)\sigma_{\kappa}(x_1)P_r
=\delta_{r\kappa}\cdot 2\pi\kappa\inv\cdot \|X\|^{2\kappa-2}
e^{-2\pi \|X\|^2}\, P_{\kappa}.
\end{equation}
\end{enumerate}
\end{lemma}

\begin{proof}
 First suppose that $x_2\ne -x_1i$.
By Lemma \ref{lem:inequality-tr},
 $\ol{x_1}-\sqrt{-1}\ol{x_2}$ 
is invertible in $\bH\otimes_{\bR}\bC$.
By Lemma \ref{lem:inequality-tr} (3),
Lemma \ref{lem:diagonal} and
Proposition \ref{prop:ecr-testf} (2), we have
\begin{align*}
 &\varphi_{\Omega,\infty}(X)
    \sigma_{\kappa}(\ol{x_1}-\sqrt{-1}\ol{x_2})\inv P_r\\
&=\varphi_{\Omega,\infty}(X)
  \sigma_{\kappa}(x_1+\sqrt{-1}x_2)
   \sigma_{\kappa}(\|X\|^2-2\sqrt{-1}it)\inv P_r\\
 &=(\|X\|^2+2t)^{-r}(\|X\|^2-2t)^{-\kappa+r} 
\varphi_{\Omega,\infty}(X)
  \sigma_{\kappa}(x_1+\sqrt{-1}x_2)\,P_r\\
 &= \delta_{r\kappa}\cdot 2^{2\kappa}\pi\kappa\inv\cdot t^{\kappa-1}
   e^{-4\pi t}(\|X\|^2+2t)^{-\kappa}\,P_{\kappa},
\end{align*}
which proves (\ref{eq:varphi-Omega-anotherform-nonsingular}).

Next suppose that $x_2=-x_1i\ne 0$. 
Then $\ol{x_1}x_2=-N(x_1)i=-\dfrac{1}{2}\|X\|^2\,i$.
Note that, if $r\ne \kappa$,
we have
\[
 \sigma_{\kappa}(\ol{x_1}-\sqrt{-1}\ol{x_2})\sigma_{\kappa}(x_1)\,P_r
=N(x_1)^{\kappa}\,\sigma_{\kappa}(1-\sqrt{-1}\,i)\,P_r=0
\]
by Lemma \ref{lem:diagonal}.
It follows from Lemma \ref{lem:8-6-1} that
\[
 \varphi_{\Omega,\infty}(X)\sigma_{\kappa}(x_1)\,P_r
=0
\qquad(0\le r\le \kappa-1).
\]
Since
$\sigma_{\kappa}(x_1+\sqrt{-1}x_2)\,P_{\kappa}
=  \sigma_{\kappa}(x_1)\,\sigma_{\kappa}(1-\sqrt{-1}\,i)\,P_{\kappa}
=2^{\kappa}\,  \sigma_{\kappa}(x_1)\,P_{\kappa}$
by Lemma \ref{lem:diagonal},
we have
\begin{align*}
  \varphi_{\Omega,\infty}(X)
  \sigma_{\kappa}(x_1)\,P_{\kappa}
=2\pi\kappa\inv\cdot \|X\|^{2\kappa-2}
e^{-2\pi \|X\|^2}\, P_{\kappa}
\end{align*}
by Proposition \ref{prop:ecr-testf}, which proves 
(\ref{eq:varphi-Omega-anotherform-singular}).
\end{proof}

For $z\in\bH\otimes_{\bR}\bC$,  
define $(s_{a,b}(z))_{a,b=0}^{\kappa}\in \bC^{(\kappa+1,\kappa+1)}$ by
\[
 \sigma_{\kappa}(z)P_b=\sum_{a=0}^{\kappa}\,P_a\,s_{a,b}(z)
\qquad(0\le b\le\kappa).
\]
Then
\begin{equation}
 \label{eq:s_ab}
\sum_{j=0}^{\kappa}\,s_{a,j}(z)s_{j,b}(w)=s_{a,b}(zw) \qquad
(z,w\in \bH\otimes_{\bR}\bC,\,0\le a,b\le \kappa).
\end{equation}
We see that $x\mapsto N(x)^{-\kappa/2}\left|s_{a,b}(x)\right|$ 
is bounded on $\bH\cross$, from which
the following two lemmas are deduced.

\begin{lemma}
 \label{lem:estimate-Phi_kappa}
There exists a positive constant $c$ such that
\begin{equation}
 \label{eq:estimate-Phi_kappa}
|(\Phi_{\kappa}(x)P_a,P_b)_{\kappa}|\le c\,N(x)^{-\kappa/2-1}
\end{equation}
for $x\in\bH\cross$ and $0\le a,b\le\kappa$.
\end{lemma}

\begin{lemma}
\label{lem:estimate-matrix-coefficients} 
There exists a positive constant $c'$ such that
\begin{equation}
 \label{eq:estimate-matrix-coefficients} 
|s_{a,b}(x+\sqrt{-1}y)|\le c'\,(N(x)+N(y))^{\kappa/2}
\end{equation}
holds for any $x,y\in\bH$ and $0\le a,b\le\kappa$.
\end{lemma}

We now prove Proposition \ref{prop:estimate}. It suffices to show the following
estimates: 
There exist positive constants $C_1,C_2$ such that the  inequalities
\begin{align}
 |(\phi_{\omega,\infty}(X')P_a,P_b)_{\kappa}|&\le\,C_1\,(1+\|X\|)^{-\kappa-2},
\label{eq:estimate-phi-omega}\\
 |(\varphi_{\Omega,\infty}(X)P_a,P_b)_{\kappa}|&\le\,C_2\,(1+\|X\|)^{-\kappa}
\label{eq:estimate-varphi-Omega}
\end{align}
hold for any $X'\in\bH^{(1,2)}$ and $X\in\bH^{(1,2)}$.

We first show (\ref{eq:estimate-phi-omega}).
Let $X'=(x_1,x_2)\in \bH^{(1,2)}$ and $0\le a,b\le\kappa$.
By Proposition \ref{prop:ecr-testf} and Lemma \ref{lem:estimate-Phi_kappa},
we have
\begin{equation*}
 \label{eq:estimate-1}
\left|( \phi_{\omega,\infty}(X')P_a,P_b)_{\kappa}\right|
\le c_1\cdot \delta(\tr(\ol{x_1}x_2)>0)\,(\tr(\ol{x_1}x_2))^{\kappa-1}
\,e^{-2\pi \tr(\ol{x_1}x_2)}\,N(x_1+x_2)^{-\kappa/2-1}
\end{equation*}
with a positive constant $c_1$.
Observe that, if $\tr(\ol{x_1}x_2)>0$, we have 
$N(x_1+x_2)>\|X'\|^2$ and
$\tr(\ol{x_1}x_2)\le \|X'\|^2$ by Lemma \ref{lem:inequality-tr}.
It follows that
\[
 \left|( \phi_{\omega,\infty}(X')P_a,P_b)_{\kappa}\right|
\le c_1\cdot\,\delta(\tr(\ol{x_1}x_2)>0)\|X'\|^{\kappa-4}
\]
and hence 
\begin{equation}
 \label{eq:estimates-1}
\mathrm{Sup}_{\|X'\|\le 1}\,
\left|( \phi_{\omega,\infty}(X')P_a,P_b)_{\kappa}\right|
<\infty
\end{equation}
under the assumption $\kappa>4$.
On the other hand, since
$t\mapsto t^{\kappa-1}e^{-2\pi t}$ is bounded on $\bR_{>0}$,
we have
\begin{equation}
 \label{eq:estimates-2}
 \left|( \phi_{\omega,\infty}(X')P_a,P_b)_{\kappa}\right|
\le c_2\cdot 
\delta(\tr(\ol{x_1}x_2)>0)\|X'\|^{-\kappa-2}
\end{equation}
with a positive constant $c_2$.
The estimate (\ref{eq:estimate-phi-omega}) follows from 
(\ref{eq:estimates-1}) and (\ref{eq:estimates-2}).

We next show (\ref{eq:estimate-varphi-Omega}).
Let $X=\tp(x_1,x_2)\in\bH^{(2,1)}$ and $0\le a,b\le\kappa$.
In view of Lemma \ref{lem:action-of-H^1} (2), we may (and do) 
suppose that $\ol{x_1}x_2=s-it\;\;(s,t\in\bR,t\ge0)$. 

We first consider the case $x_2\ne -x_1i$.
Then $\xi=\ol{x_1}-\sqrt{-1}\ol{x_2}$ is 
invertible in $\bH\otimes_{\bR}\bC$ by Lemma \ref{lem:inequality-tr}.
By (\ref{eq:s_ab}) and Lemma \ref{lem:varphi-Omega-anotherform} (1), we have
\begin{align*}
\varphi_{\Omega,\infty}(X)P_a&
=\sum_{j=0}^{\kappa}\,s_{j,a}(1)\varphi_{\Omega,\infty}(X)P_j\\
&
=\sum_{j=0}^{\kappa}
\left(\sum_{r=0}^{\kappa}s_{j,r}(\xi\inv)s_{r,a}(\xi)
\right)\,\varphi_{\Omega,\infty}(X)P_j\\
&
 =\sum_{r=0}^{\kappa}\,s_{r,a}(\xi)\,\varphi_{\Omega,\infty}(X)
\sigma_{\kappa}(\xi)\inv P_r\\
&=2^{2\kappa}\pi\kappa\inv\cdot s_{\kappa,a}(\xi)\,t^{\kappa-1}
e^{-4\pi t}\,(\|X\|^2+2t)^{-\kappa}\,P_{\kappa}.
\end{align*}
By Lemma \ref{lem:estimate-matrix-coefficients}, 
there exists a positive constant $c_3$  such that
\begin{equation}
 \label{eq:matrix-coef-varphi-Omega}
| (\varphi_{\Omega,\infty}(X)P_a,P_b)_{\kappa}|
\le
c_3\,t^{\kappa-1}e^{-4\pi t}\,\|X\|^{\kappa}\,(\|X\|^2+2t)^{-\kappa}
\end{equation}
holds for any $X\in\bH^{(2,1)}$.
Since
\[
  t^{\kappa-1}e^{-4\pi t}\,\|X\|^{\kappa}\,(\|X\|^2+2t)^{-\kappa}
 \le \left(\dfrac{t}{\|X\|^2+2t}\right)^{\kappa-1}
\dfrac{\|X\|^{\kappa}}{\|X\|^2+2t}
\le 2^{-\kappa+1}\|X\|^{\kappa-2},
\]
we have
\begin{equation}
 \label{eq:estimates-3}
\mathrm{Sup}_{\|X\|\le 1}|(\varphi_{\Omega}(X)P_a,P_b)_{\kappa}|<\infty.
\end{equation}
We also have 
\begin{equation}
 \label{eq:estimates-4}
| (\varphi_{\Omega,\infty}(X)P_a,P_b)_{\kappa}|
\le c_3c_4\,\|X\|^{-\kappa},
\end{equation}
where $c_4=\mathrm{Sup}_{t>0}\,t^{\kappa-1}e^{-4\pi t}$.
The estimate (\ref{eq:estimate-varphi-Omega})
in the case $x_2\ne -x_1i$ follows from 
(\ref{eq:estimates-3}) and (\ref{eq:estimates-4}).

We next consider the remaining case where $x_2=-x_1i$ and $x_1\ne 0$.
A similar argument as above  and Lemma \ref{lem:varphi-Omega-anotherform}
(2) show that
\begin{align*}
  (\varphi_{\Omega,\infty}(X)P_a,P_b)_{\kappa}
&=2\pi\kappa\inv\cdot s_{\kappa,a}(x_1\inv)
\,\|X\|^{2\kappa-2}\,e^{-2\pi\|X\|^2} (P_{\kappa},P_b)_{\kappa}.
\end{align*}
Since $|s_{\kappa,a}(x_1\inv)|\,\|X\|^{\kappa}$
is bounded,
we have
\[
 |(\varphi_{\Omega,\infty}(X)P_a,P_b)_{\kappa}|\le
c_5\cdot \|X\|^{\kappa-2}e^{-2\pi\|X\|^2}
\]
with a positive constant $c_5$,
which implies (\ref{eq:estimate-varphi-Omega}) in this case.

 \subsection{}

In this subsection, we complete the proof of (\ref{eq:ecr}). 

Let
\begin{align*}
 \phi_{\omega}(X')
    &=\phi_{\omega,\infty}(X'_{\infty})\prod_{p<\infty}\phi_{0,p}(X'_p)
    \qquad(X'=(X'_v)_{v\le\infty}\in B_{\bA}^{(1,2)}),\\
\varphi_{\Omega}(X)
    &=\varphi_{\Omega,\infty}(X_{\infty})\prod_{p<\infty}\varphi_{0,p}(X_p)
    \qquad(X=(X_v)_{v\le\infty}\in B_{\bA}^{(2,1)}).
\end{align*}

Proposition \ref{prop:estimate} 
immediately implies the following:

\begin{proposition}
 \label{prop:convergence}

Suppose that $\kappa>4$ and let $(h,g)\in H_{\bA}\times G_{\bA}$.

\begin{enumerate}
 \item The sum
\begin{equation}
 \label{eq:convergence-1}
  \sum_{X'\in B^{(1,2)}}\,r'(h,g)\phi_{\omega}(X')
\end{equation}
is absolutely convergent and coincides with the integral
\[
 \int_{H_{\infty}}\,\theta(hx,g)\omega(x)dx.
\]
 \item The sum
\begin{equation}
 \label{eq:convergence-2}
\sum_{X\in B^{(2,1)}}\,r(h,g)\varphi_{\Omega}(X)
\end{equation}
is absolutely convergent and coincides with the integral
\[
 \int_{G_{\infty}}\,\Omega(y)\theta(h,gy\inv)dy.
\]

\end{enumerate}
\end{proposition}

\bigskip

We now complete the proof of (\ref{eq:ecr}).
In view of Corollary 2.6 in Chapter VII 
\cite{SW} and Proposition \ref{prop:estimate}, 
under the assumption that $\kappa>\dim_{\Q}B^{(1,2)}=8$,
we can apply the Poisson summation formula to the sum (\ref{eq:convergence-1})
and obtain
\begin{align*}
  \int_{H_{\infty}}\,\theta(hx,g)\omega(x)dx
&=\sum_{X'\in B^{(1,2)}}\,r'(h,g)\phi_{\omega}(X')\\
 &=\sum_{X\in B^{(2,1)}}\,I\inv (r'(h,g)\phi_{\omega})(X).
\end{align*}
By  Proposition \ref{prop:ecr-testf}, we have
$I\inv (r'(h,g)\phi_{\omega})=
\dfrac{c_{\kappa}^G}{c_{\kappa}^H}\cdot
r(h,g)\varphi_{\Omega}$.
We thus have
\begin{align*}
 c_{\kappa}^H\,\int_{H_{\infty}}\,\theta(hx,g)\omega(x)dx
 &=c_{\kappa} ^G\,\sum_{X\in B^{(2,1)}}\,r(h,g)\varphi_{\Omega}(X)\\
 &=c_{\kappa}^G\, \int_{G_{\infty}}\,\Omega(y)\theta(h,gy\inv)dy,
\end{align*}
proving (\ref{eq:ecr}).

 \section{The correspondence of automorphic $L$-functions
 for Arakawa lifting}
\label{sec:functoriality}

 \subsection{}

In the remaining part of the paper, we will show the following
correspondence of automorphic $L$-functions for Arakawa lifting 
and its adjoint 
by using the Eichler commutation relations 
given in Section \ref{subsec:construction-upsilon-unramified}
for the unramified case
and in Section \ref{subsec:construction-upsilon-ramified}
for the ramified case.

\begin{theorem}
\label{th:functoriality}
 \begin{enumerate}
  \item If $f\in\cS_{\kappa}^H$ is a Hecke eigenform,
then so is $\cL(f)$ and 
\begin{equation}
 \label{eq:L-functoriality}
L(s,\cL(f),\mathrm{std})=\zeta(s)L(s,f,\mathrm{std}).
\end{equation}
Here $\zeta(s)$ stands for the Riemann zeta function.
For the definitions of the standard $L$-functions, see
(\ref{eq:global-L-H}) and (\ref{eq:global-L-G}).

\item If $F\in\cS_{\kappa}^G$ is a Hecke eigenform,
then so is $\cL^*(F)$ and 
\begin{equation}
 \label{eq:L-functoriality*}
L(s,F,\mathrm{std})=\zeta(s)L(s,\cL^*(F),\mathrm{std}).
\end{equation}
 \end{enumerate}
\end{theorem}

We only prove the first assertion of Theorem \ref{th:functoriality}
since the second assertion is similarly proved.
The key of the proof of Theorem \ref{th:functoriality} is the following
Eichler commutation relations:

\begin{theorem}
 \label{th:Eichler-commutation-relations}
  For a finite place $p$ of $\bQ$, there exists a $\bC$-
algebra homomorphism
\[
 \upsilon\colon \cH(G_p,K_p)\,\to\,\cH(H_p,U_p)
\]
such that the equality
\begin{equation}
 \label{eq:Eichler-commutation-relations}
\int_{G_p}\,\theta(h,gy\inv)\alpha(y)dy=\int_{H_p}\,\theta(hx,g)\upsilon(\alpha)(x)dx
\end{equation}
holds for $\alpha\in\cH(G_p,K_p)$.
\end{theorem}

\begin{corollary}
 \label{cor:functoriality}
If $f\in\cS_{\kappa}^H$ is a Hecke eigenform, so is $\cL(f)$ and
its Hecke eigenvalues are given by
\begin{equation}
 \label{eq:Hecke-eigenvalues}
\Lambda_p^{\cL(f)}(\alpha)=\lambda_p^f(\upsilon(\alpha))
\end{equation}
for any $\alpha\in\cH(G_p,K_p)$.
(For the definitions of $\Lambda_p^{\cL(f)}$ and $\lambda_p^f$,
see (\ref{eq:Hecke-eigenvalues-G}) and (\ref{eq:Hecke-eigenvalues-H})).
\end{corollary}

We will give an explicit construction of $\upsilon$
and the proof of Theorem \ref{th:functoriality} (1)
for the unramified (respectively ramified) case in
Section \ref{sec:ecr-unramified} (respectively in
Section \ref{sec:ecr-ramified}).

 \section{Eichler commutation relations at unramified places}
\label{sec:ecr-unramified}

 \subsection{}
\label{subsec:ecr-unramified} 

In this section, we consider the unramified case; namely we assume that
 $p$ does not divide the discriminant
$d_B$ of $B$. Throughout this section, we often drop the subscript 
$p$.
As in Section \ref{subsec:Satake-G-unramified}, we let
\begin{align*}
H & =\left\{h\in \GL_4(\bQ_p)\mid \tp h\Nmat{0&J}{-J&0}h=
\Nmat{0&J}{-J&0}\right\},\\
U&=H\cap \GL_4(\bZ_p),\\
 G& =\left\{g\in \GL_4(\bQ_p)\mid \tp g\Nmat{0&1_2}{-1_2&0}g=
\Nmat{0&1_2}{-1_2&0}\right\},\\
 K&=G\cap \GL_4(\bZ_p),
\end{align*}
where
\[
 J=\Nmat{0&1}{-1&0}\in\GL_2(\bQ_p).
\]
Define  subgroups $N_H$ and $M_H$ of $H$ by
\begin{align*}
N_H&=
\left\{
n_H(b):=\Nmat{1_2&b1_2}{0&1_2}
\mid b\in\bQ_p
\right\},\\
M_H&=
\left\{
d_H(A):=\Nmat{A&0}{0&J\inv \tp A\inv J}
\mid A\in\GL_2(\bQ_p)
\right\}.
\end{align*}
Then $H=N_HM_HU$.
We also define subgroups $N_G$ and $M_G$ of $G$ by
\begin{align*}
 N_G &= \left\{
n_G(b,c,d):=
\left(
\begin{array}{c|c} 
  {\begin{array}{cc} 1 & 0 \\ 0 & 1 \end{array}}
 & {\begin{array}{cc} b & c \\ c & d \end{array}} \\ 
    \hline
  O & {\begin{array}{cc} 1 & 0 \\ 0 & 1 \end{array}}
\end{array}
\right)
\mid b,c,d\in\bQ_p
\right\},\\
M_G&=
\left\{
d_G(A):=\Nmat{A&0}{0& \tp A\inv}
\mid A\in\GL_2(\bQ_p)
\right\}.
\end{align*}
Then $G=N_GM_GK$.

The Weil representation $r'$ given in Lemma \ref{lem:r'}
is reformulated as follows: For $\phi\in\mathscr{S}(\bQ_p^{(2,4)})$ and $x,y\in\bQ_p^{(2,2)}$,
\begin{align*}
 r'(1,g)\phi(x,y)&=\phi((x,y)g)
\qquad(g\in G),\\
 r'(d_H(A),1)\phi(x,y)&=|\det A|^2\,\phi(\tp Ax,\tp Ay)
\qquad(A\in\GL_2(\bQ_p)), \\
 r'(n_H(b),1)\phi(x,y)& =
\psi(-b\,\Tr(J\inv \tp xJy))\phi(x,y)
\qquad(b\in \bQ_p).
\end{align*}

 \subsection{}
\label{subsec:hecke-operators-G-unramified}

In this subsection, we prepare several facts on Hecke algebras.

For $i=1,2$, let $\cT_i^H$ and $\cT_i^G$ be the characteristic functions
of
\[
 T_i^H  =\{h\in H\mid ph\in \mathrm{M}_4(\bZ_p),\,
\mathrm{rank}_{\bF_p}(ph)=i\}
\]
and
\[
 T_i^G =\{g\in G\mid pg\in \mathrm{M}_4(\bZ_p),\,
\mathrm{rank}_{\bF_p}(pg)=i\},
\]
respectively.
Here, for $X\in\mathrm{M}_2(\bZ_p)$, we write
$\mathrm{rank}_{\bF_p}(X)$ 
for the rank of $(X \bmod{p\,\mathrm{M}_2(\bZ_p)})$ in $\mathrm{M}_2(\bF_p)$.
Then $\cT_i^H\in\cH(H,U)$ and $\cT_i^G\in\cH(G,K)$.

\begin{lemma}
 \label{lem:hecke-algebras-unramified}
We have 
\begin{align*}
 \cH(H,U)&=\bC[\cT_1^H,\cT_2^H],\\
 \cH(G,K)&=\bC[\cT_1^G,\cT_2^G].
\end{align*}
\end{lemma}

\begin{proof}
 This follows from \cite{Sa}, Section 5.1, Theorem 1.
\end{proof}

For $a\in\bQ_p$, put
\begin{align*}
\nu_H(a)&=d_H\Nmat{1&a}{0&1}=
\begin{pmatrix}
 1&a&0&0\\
 0&1 &0 &0 \\
 0&0 &1 &a \\
 0&0 &0 &1 
\end{pmatrix}\in H,\\
 \nu_G(a)&=d_G\Nmat{1&a}{0&1}=
\begin{pmatrix}
 1&a&0&0\\
 0&1 &0 &0 \\
 0&0 &1 &0 \\
 0&0 &-a &1 
\end{pmatrix}\in G.
\end{align*}

We put
\begin{align*}
 \Lambda_i&=\{(b,c,d)\in(\bZ_p/p\bZ_p)^3\mid \mathrm{rank}_{
\bF_p}\Nmat{b&c}{c&d}=i\}
\end{align*}
for $i=1,2$.
Note that $|\Lambda_1|=p^2-1$ and $|\Lambda_2|=p^3-p^2$.

The proofs of the following facts are straightforward and we omit it.

\begin{lemma}
 \label{lem:hecke-operators-H-unramified}
We have
\begin{align*}
 T_1^H &=
\bigsqcup_{a,\,b\in \bZ_p/p\bZ_p}\,n_H(b)\nu_H(a)d_H\Nmat{p&0}{0&1}U
\,\sqcup\, 
\bigsqcup_{b\in \bZ_p/p\bZ_p}\,n_H(b)d_H\Nmat{1&0}{0&p}U \\
&  \qquad \,\sqcup\, \bigsqcup_{a\in \bZ_p/p\bZ_p}\,\nu_H(a)d_H\Nmat{1&0}{0&p\inv}U 
\,\sqcup\, \bigsqcup d_H\Nmat{p\inv&0}{0&1}U
\end{align*}
and
\begin{align*}
 T_2^H &=
\bigsqcup_{a\in\bZ_p/p^2\bZ_p}\,\nu_H(a)d_H\Nmat{p&0}{0&p\inv}U
\,\sqcup\,   d_H\Nmat{p\inv&0}{0&p}U
\,\sqcup\, 
 \bigsqcup_{b\in \bZ_p/p^2\bZ_p}\,n_H(b)d_H\Nmat{p&0}{0&p}U \\
&\qquad \,\sqcup\,  d_H\Nmat{p\inv&0}{0&p\inv}U
\,\sqcup\,  \bigsqcup_{a\in \bZ_p\cross/p\bZ_p}\,\nu_H(p\inv a)\,U 
\,\sqcup\,  \bigsqcup_{b\in \bZ_p\cross/p\bZ_p} n_H(p\inv b)\,U,
\end{align*}
where $\bZ_p\cross/p\bZ_p$ means $(\bZ_p- p\bZ_p)/p\bZ_p$.
\end{lemma}

\begin{lemma}
  \label{lem:hecke-operators-G-unramified}
We have

\begin{align*}
 T_1^G 
&=\bigsqcup_{a,\,c\,\in\bZ_p/p\bZ_p,\,b\in\bZ_p/p^2\bZ_p}
   n_G(b,c,0)\nu_G(a)d_G\Nmat{p&0}{0&1}K\\
& \qquad \,\sqcup\, \bigsqcup_{c\,\in\bZ_p/p\bZ_p,\,d\in\bZ_p/p^2\bZ_p}
   n_G(0,c,d)d_G\Nmat{1&0}{0&p}K
\,\sqcup\, \bigsqcup_{(b,c,d)\in\Lambda_1}n_G(p\inv b,p\inv c,p\inv d)K\\
 &\qquad \,
\sqcup\, \bigsqcup_{a\in\bZ_p/p\bZ_p}\nu_G(a)d_G\Nmat{1&0}{0&p\inv}K
\,\sqcup\, d_G\Nmat{p\inv&0}{0&1}K.
\end{align*}
and
\begin{align*}
 T_2^G
&=
\bigsqcup_{b,\,c,\,d\in\bZ_p/p^2\bZ_p}\,n_G(b,c,d)d_G\Nmat{p&0}{0&p}K
\,\sqcup\,
\bigsqcup_{a,\,b\in\bZ_p/p^2\bZ_p}\,n_G(b,0,0)\nu_G(a)d_G\Nmat{p&0}{0&p\inv}K\\
 &\quad  \,\sqcup\, \bigsqcup_{d\in\bZ_p/p^2\bZ_p}\,n_G(0,0,d)d_G\Nmat{p\inv&0}{0&p}K
\,\sqcup\,d_G\Nmat{p\inv&0}{0&p\inv}K\\
&\quad\,\sqcup\,\bigsqcup_{a,\,m\in \bZ_p/p\bZ_p,\,d\in\bZ_p\cross/p\bZ_p,\,
l\in\bZ_p/p^2\bZ_p}\,n_G(p\inv a^2d+l,\,p\inv ad+m,\,p\inv d)
   \nu_G(a)d_G\Nmat{p&0}{0&1}K\\
&\quad\,\sqcup\,\bigsqcup_{b\in\bZ_p\cross/p\bZ_p,\,c\in\bZ_p/p\bZ_p,\,d\in\bZ_p/p^2\bZ_p}\,n_G(p\inv b,c,d)d_G\Nmat{1&0}{0&p}K\\
&\quad\,\sqcup\,\bigsqcup_{a\in\bZ_p/p\bZ_p,\,b\in\bZ_p\cross/p\bZ_p}
   \,n_G(p\inv b,0,0)\nu_G(a)d_G\Nmat{1&0}{0&p\inv}K
\,\sqcup\,\bigsqcup_{d\in\bZ_p\cross/p\bZ_p}\,n_G(0,0,p\inv d)
   d_G\Nmat{p\inv&0}{0&1}K\\
&\quad \,\sqcup\,\bigsqcup_{a\in\bZ_p\cross/p\bZ_p,\,c,\,m\in\bZ_p/p\bZ_p}
  n_G(p^{-2}(ac+pm),p\inv c,0)\nu_G(p\inv a)K\,\sqcup\,
\bigsqcup_{(b,c,d)\in\Lambda_2}\,n_G(p\inv b,p\inv c,p\inv d)K.
\end{align*}
\end{lemma}

Recall that the spherical functions
$\omega_{\chi}\colon H\to\bC$ and $\Omega_{\Xi}\colon G\to\bC$
for $\chi=(\chi_1,\chi_2),\Xi=(\Xi_1,\Xi_2)\in 
X_{\mathrm{unr}}(\bQ\cross_p)^2$
are defined in Section \ref{subsec:Satake-H-unramified} and
Section \ref{subsec:Satake-G-unramified}. 
Then we have
\begin{align*}
 (\omega_{\chi}*\varphi)(h)&:=\int_H\omega_{\chi}(hx\inv)\varphi(x)dx
=\lambda_{\chi}(\varphi)\omega_{\xi}(h)\qquad (h\in H,\varphi\in\cH(H,U)),\\
 (\Omega_{\Xi}*\Phi)(g)&:=\int_G\Omega_{\Xi}(gy\inv)\Phi(y)dy
=\Lambda_{\Xi}(\Phi)\Omega_{\Xi}(g)\qquad (g\in G,\Phi\in\cH(G,K))
\end{align*}
with $\lambda_{\chi}\in \mathrm{Hom}_{\bC\text{-alg}}(\cH(H,U),\bC)$
and $\Lambda_{\Xi}\in \mathrm{Hom}_{\bC\text{-alg}}(\cH(G,K),\bC)$
(cf. \cite{Sa}, Section 5.3).
Using Lemma \ref{lem:hecke-operators-H-unramified} and
\ref{lem:hecke-operators-G-unramified}, we get the following.

\begin{lemma}
 \label{lem:eigenvalues-spherical-f-unramified}
We have
\begin{align*}
 \lambda_{\chi}(\cT_1^H)&=p(\chi_1+\chi_2+\chi_1\inv+\chi_2\inv),\\  
 \lambda_{\chi}(\cT_2^H)&=p(\chi_1\chi_2+\chi_1\chi_2\inv+
       \chi_1\inv \chi_2+\chi_1\inv \chi_2\inv)+2p-2,\\
 \Lambda_{\Xi}(\cT_1^G)&=p^2(\Xi_1+\Xi_2+\Xi_1\inv+\Xi_2\inv)+p^2-1,\\  
 \Lambda_{\Xi}(\cT_2^G)&=p^3(\Xi_1\Xi_2+\Xi_1\Xi_2\inv+\Xi_1\inv \Xi_2
    +\Xi_1\inv \Xi_2\inv)+
     p^2(p-1)(\Xi_1+\Xi_2+\Xi_1\inv+\Xi_2\inv)+2p^3-2p^2,  
\end{align*}
where we write, for $i=1,2$, $\Xi_i$ and $\chi_i$ for $\Xi_i(p)$ and $\chi_i(p)$,
respectively. 
\end{lemma}

 \subsection{}
\label{subsec:construction-upsilon-unramified}

Denote by $\phi_0\in
\mathscr{S}(\bQ_p^{(2,4)})$ the characteristic function of $\bZ_p^{(2,4)}$
and set
\begin{align*}
  J_i^H(x,y) &=\int_{T_i^H}\,r'(h,1)\phi_0(x,y)dh,\\
 J_i^G(x,y) &=\int_{T_i^G}\,r'(1,g)\phi_0(x,y)dg
\end{align*}
for $x,y\in \bQ_p^{(2,2)}$ and $i=1,2$.

We now state the main results of this section,
whose proof is given in Section \ref{subsec:proof-ecr-unramified}.
\begin{proposition}
 \label{prop:ecr-unramified}
 For $x,y\in \bQ_p^{(2,2)}$, we have
\begin{equation}
\label{eq:ecr-unramified-1}
 J_1^G(x,y)=pJ_1^H(x,y)+(p^2-1)\phi_0(x,y)
\end{equation}
and
\begin{equation}
\label{eq:ecr-unramified-2}
 J_2^G(x,y)=(p^2-p)J_1^H(x,y)+p^2J_2^H(x,y).
\end{equation}
\end{proposition}

\bigskip

This proposition immediately implies the following result:

\begin{corollary}
  \label{cor:ecr-unramified}
The algebra homomorphism $\upsilon\colon \cH(G,K)\,\to\,\cH(H,U)$ given by
\begin{align*}
 \upsilon(\cT_1^G)&=p\,\cT_1^H+(p^2-1)\,\cT_0^H,\\
 \upsilon(\cT_2^G)&=(p^2-p)\,\cT_1^H+p^2\cT_2^H
\end{align*}
satisfies (\ref{eq:Eichler-commutation-relations})
(and hence (\ref{eq:Hecke-eigenvalues})),
where $\cT_0^H$ denotes the characteristic function of $U$.
\end{corollary}

 \subsection{}
\label{subsec:proof-L-unramified}

In this subsection, assuming Corollary \ref{cor:ecr-unramified},
we prove the following correspondence of the
local factors of the standard $L$-functions.

\begin{proposition}
 \label{prop:functoriality-unramified}
Let $f\in \cA_{\kappa}^H$ be a Hecke eigenform.
\begin{enumerate}
 \item The Arakawa lifting $\cL(f)$ of $f$ is a eigenfunction under the 
action of $\cH(G,K)$.
 \item We have\begin{equation*}
 \label{eq:functoriality-unramified-L}
L_p(s,\cL(f),\mathrm{std})=\zeta_p(s)L_p(s,f,\mathrm{std}).
\end{equation*}

\end{enumerate}
\end{proposition}

\begin{proof}

The first assertion is a direct consequence of Lemma
\ref{lem:hecke-algebras-unramified} and 
Corollary \ref{cor:ecr-unramified}.
To show the second, 
we take $\chi=(\chi_1,\chi_2)\in X_{\mathrm{unr}}(\bQ_p\cross)^2$
such that $\omega_p^f=\omega_{\chi}$ (see
Lemma \ref{lem:SatakeParameter-H-unramified})
and
$\Xi=(\Xi_1,\Xi_2)\in X_{\mathrm{unr}}(\bQ_p\cross)^2$
such that $\Omega_p^{\cL(f)}=\Omega_{\Xi}$ (see
Lemma \ref{lem:SatakeParameter-G-unramified}).
By abuse of notation, 
we write $\chi_i$ and $\Xi_i$ for $\chi_i(p)$ and $\Xi_i(p)$ 
($i=1,2$), respectively.
In view of Lemma \ref{lem:eigenvalues-spherical-f-unramified}, we have
\begin{align}
\label{eq:ev-T1H}  
\lambda_p^f(\cT_1^H)&=p\left(\chi_1+\chi_2
+\chi_1\inv+\chi_2\inv\right),\\ 
\label{eq:ev-T2H}
 \lambda_p^f(\cT_2^H)&=p\left(\chi_1\chi_2+\chi_1\chi_2\inv+\chi_1\inv
\chi_2+\chi_1\inv\chi_2\inv\right)+2p-2,\\
\label{eq:ev-T1G}  
 \Lambda_p^{\cL(f)}(\cT_1^G)&
  =p^2\left(\Xi_1+\Xi_2+\Xi_1\inv+\Xi_2\inv\right)+p^2-1
,\\
\label{eq:ev-T2G}  
 \Lambda_p^{\cL(f)}(\cT_2^G)&
=p^3\left(\Xi_1\Xi_2+\Xi_1\Xi_2\inv+\Xi_1\inv \Xi_2+\Xi_1\inv\Xi_2\inv\right)\\ \notag
&\quad
+p^2(p-1)\left(\Xi_1+\Xi_2+\Xi_1\inv+\Xi_2\inv\right)
+2p^3-2p^2.
\end{align}
By (\ref{eq:Hecke-eigenvalues}), Corollary \ref{cor:ecr-unramified},
(\ref{eq:ev-T1H}) and (\ref{eq:ev-T2H}), we have
\begin{align*}
 \Lambda_p^{\cL(f)}(\cT_1^G)&=\lambda_p^f(\upsilon(\cT_1^G))\\
 &=p\,\lambda_p^f(\cT_1^H) +(p^2-1)\,\lambda_p^f(\cT_0^H) \\
 & =p^2\,\left(\chi_1+\chi_2+\chi_1\inv+\chi_2\inv\right)+p^2-1
\end{align*}
and
\begin{align*}
 \Lambda_p^{\cL(f)}(\cT_2^G)&=\lambda_p^f(\upsilon(\cT_2^G))\\
 &=(p^2-p)\,\lambda_p^f(\cT_1^H) +p^2\,\lambda_p^f(\cT_2^H) \\
 & =(p^2-p)\,\left(\chi_1+\chi_2+\chi_1\inv+\chi_2\inv\right)\\
&\quad+p^3\,\left(\chi_1\chi_2+\chi_1\chi_2\inv+\chi_1\inv
\chi_2+\chi_1\inv\chi_2\inv\right)
+2p^3-2p^2.
\end{align*}
Comparing the above formulas with (\ref{eq:ev-T1G}) and (\ref{eq:ev-T2G}),
we obtain
\begin{align*}
 \Xi_1+\Xi_2+\Xi_1\inv+\Xi_2\inv&=\chi_1+\chi_2+\chi_1\inv+\chi_2\inv,\\
 \Xi_1\Xi_2+\Xi_1\Xi_2\inv+\Xi_1\inv \Xi_2+\Xi_1\inv\Xi_2\inv
&=\chi_1\chi_2+\chi_1\chi_2\inv+\chi_1\inv
\chi_2+\chi_1\inv\chi_2\inv.
\end{align*}
Hence we have
\[
 \prod_{i=1}^2\,(1-\Xi_i p^{-s})\inv\,(1-\Xi_i\inv p^{-s})\inv
=\prod_{i=1}^2\,(1-\chi_i p^{-s})\inv\,(1-\chi_i\inv p^{-s})\inv,
\]
which immediately implies the proposition.
\end{proof}
 \subsection{}
\label{subsec:proof-ecr-unramified}

In this subsection, we give a sketch of the proof of Proposition 
\ref{prop:ecr-unramified}. The detail of the proof is given in Section
\ref{sec:ecr-unramified-details}.
The following fact is easily verified.

\begin{lemma}
 \label{lem:reduction}
For $u_1,u_2\in\GL_2(\bZ_p)$ and $x,y\in\bQ_p^{(2,2)}$, we have
 \begin{align*}
  J_i^H(u_1xu_2,y)&=J_i^H(x,u_1\inv y\, \tp u_2),\\
  J_i^G(u_1xu_2,y)&=J_i^G(x,u_1\inv y\, \tp u_2).
 \end{align*}
\end{lemma}

Thus we may (and do) assume $x=\Nmat{p^{\alpha}&0}{0&p^{\beta}}\;\;
(\alpha\le\beta)$ and put
\begin{align*}
 J_i^H(y;\alpha,\beta)&=J_i^H\left(\Nmat{p^{\alpha}&0}{0&p^{\beta}},y\right),\\
 J_i^G(y; \alpha,\beta)&=J_i^G\left(\Nmat{p^{\alpha}&0}{0&p^{\beta}},y\right)
\end{align*}
to simplify the notation.
We have $J_i^H(y; \alpha,\beta)=J_i^G(y; \alpha,\beta)=0$ if $\beta\le -2$.
Henceforth we suppose that $\beta\ge -1$.

We denote by $\sigma$ the characteristic function of $\bZ_p$
and set
\[
 [i_1,i_2,i_3,i_4](y)=\prod_{k=1}^4\,\sigma(p^{-i_k}y_k)
\qquad \left(y=\Nmat{y_1&y_2}{y_3&y_4}\in\bQ_p^{(2,2)}\right).
\]

Proposition \ref{prop:ecr-unramified} is a straightforward 
consequence of the following four lemmas, whose
proofs are given in Section \ref{sec:ecr-unramified-details}.

\begin{lemma}
 \label{lem:J1H}
\begin{enumerate}
 \item If $\beta<-1$, we have
\[
 J_1^H(y; \alpha,\beta)=0.
\]
 \item If $\alpha=\beta=-1$, we have
\[
 J_1^H(y; \alpha,\beta)=0.
\]
 \item If $\alpha=0$ and $\beta=-1$, we have
\[
 J_1^H(y; \alpha,\beta)=p\inv[0,0,-1,-1](y)\,\sigma(p\inv y_2-y_3).
\]
 \item If $\alpha\ge 1$ and $\beta=-1$, we have
\[
 J_1^H(y; \alpha,\beta)=p\inv[0,1,-1,-1](y).
\]
 \item If $\alpha=\beta=0$, we have
\[
 J_1^H(y; \alpha,\beta)
=p\inv[-1,-1,-1,-1](y)\,\sigma(y_2-y_3)\sigma(p\,\det y)+
[0,0,0,0](y).
\]
 \item If $\alpha\ge 1$ and $\beta=0$, we have
\[
 J_1^H(y; \alpha,\beta)=[0,0,0,0](y)+p\inv [-1,0,-1,0](y)
-p\inv [0,0,-1,0](y)+p\inv [0,0,-1,-1](y)+p^2 [1,1,0,0](y).
\]
 \item If $\alpha\ge \beta\ge 1$, we have
\begin{align*}
  J_1^H(y; \alpha,\beta)&=p^2 [0,0,0,0](y)\,\sigma(p\inv \det y)
+p\inv [-1,-1,-1,-1](y)\,\sigma(p \det y)\\
&\quad +p^3 [1,1,1,1](y)+[0,0,0,0](y).
\end{align*}

\end{enumerate}
\end{lemma}


\begin{lemma}
 \label{lem:J2H}
\begin{enumerate}
\item If $\beta<-1$, we have
\[
 J_2^H(y; \alpha,\beta)=0.
\]
 \item If $\alpha=\beta=-1$, we have
\[
 J_2^H(y; \alpha,\beta)=p^{-2} [-1,-1,-1,-1](y)\,\sigma(p\inv(y_2-y_3)).
\]
 \item If $\alpha=0$ and $\beta=-1$, we have
\[
 J_2^H(y; \alpha,\beta)=p^{-2} [-1,0,-1,-1](y)\,\sigma(p\inv y_2-y_3).
\]
 \item If $\alpha\ge 1$ and $\beta=-1$, we have
\[
 J_2^H(y; \alpha,\beta)=[1,1,-1,-1](y)+p^{-2}[-1,1,-1,-1](y).
\]
 \item If $\alpha=\beta=0$, we have
\begin{align*}
  J_2^H(y; \alpha,\beta)&=
p^{-2}[-1,-1,-1,-1](y)\,\sigma(y_2-y_3)+p[0,0,0,0](y)
\,\sigma(p\inv(y_2-y_3))
-[0,0,0,0].
\end{align*}
 \item If $\alpha\ge 1$ and $\beta=0$, we have
\begin{align*}
  J_2^H(y; \alpha,\beta)&=
[0,0,-1,-1](y)\,\sigma(\det y)+p[1,1,0,0](y)+p[0,1,0,0](y)\\
 &\quad +p^{-2} [-1,0,-1,-1](y)-2[0,0,0,0](y).
\end{align*}
 \item If $\alpha\ge \beta\ge 1$, we have
\begin{align*}
  J_2^H(y; \alpha,\beta)&=p\,[0,0,0,0](y)\,\sigma(p\inv \det y)
+[-1,-1,-1,-1](y)\,\sigma(\det y)+(p^4+p^2)[1,1,1,1](y)
\\
&\quad +p^{-2} [-1,-1,-1,-1](y)+(p-2)[0,0,0,0](y).
\end{align*}

\end{enumerate}
\end{lemma}


\begin{lemma}
 \label{lem:J1G}
\begin{enumerate}
\item If $\beta<-1$, we have
\[
 J_1^G(y; \alpha,\beta)=0.
\]
 \item If $\alpha=\beta=-1$, we have
\[
 J_1^G(y; \alpha,\beta)=0.
\]
 \item If $\alpha=0$ and $\beta=-1$, we have
\[
 J_1^G(y; \alpha,\beta)= [0,0,-1,-1](y)\,\sigma(p\inv y_2-y_3).
\]
 \item If $\alpha\ge 1$ and $\beta=-1$, we have
\[
 J_1^G(y; \alpha,\beta)=[0,1,-1,-1](y).
\]
 \item If $\alpha=\beta=0$, we have
\begin{align*}
  J_1^G(y; \alpha,\beta)&=
[-1,-1,-1,-1](y)\,\sigma(y_2-y_3)\sigma(p\,\det y)+(p^2+p-1)[0,0,0,0](y).
\end{align*}
 \item If $\alpha\ge 1$ and $\beta=0$, we have
\begin{align*}
  J_1^G(y; \alpha,\beta)&=
p^3\,[1,1,0,0](y)+(p^2+p-1)[0,0,0,0](y)+[0,0,-1,-1](y)\\
 & \quad -[0,0,-1,0](y)+[-1,0,-1,0](y).
\end{align*}
 \item If $\alpha\ge \beta\ge 1$, we have
\begin{align*}
  J_1^G(y; \alpha,\beta)&=p^3\,[0,0,0,0](y)\,\sigma(p\inv \det y)
+[-1,-1,-1,-1](y)\,\sigma(p\det y)+p^4[1,1,1,1](y)
\\
&\quad +(p^2+p-1)[0,0,0,0](y).
\end{align*}

\end{enumerate}

\end{lemma}

%
\begin{lemma}
 \label{lem:J2G}
\begin{enumerate}
\item If $\beta<-1$, we have
\[
 J_2^G(y; \alpha,\beta)=0.
\]
 \item If $\alpha=\beta=-1$, we have
\[
 J_2^G(y; \alpha,\beta)=[-1,-1,-1,-1](y)\,\sigma(p\inv(y_2-y_3)).
\]
 \item If $\alpha=0$ and $\beta=-1$, we have
\[
 J_2^G(y; \alpha,\beta)= (p-1)\,[0,0,-1,-1](y)\,\sigma(p\inv y_2-y_3)+
[-1,0,-1,-1](y)\,\sigma(p\inv y_2-y_3).
\]
 \item If $\alpha\ge 1$ and $\beta=-1$, we have
\[
 J_2^G(y; \alpha,\beta)=(p-1)\,[0,1,-1,-1](y)+p^2[1,1,-1,-1](y)+[-1,1,-1,-1](y).
\]
 \item If $\alpha=\beta=0$, we have
\begin{align*}
  J_2^G(y; \alpha,\beta)&=p^3\,[0,0,0,0](y)\,\sigma(p\inv(y_2-y_3))
+(p-1)[-1,-1,-1,-1](y)\,\sigma(y_2-y_3)\sigma(p\,\det y)\\
 &\quad + [-1,-1,-1,-1](y)\,\sigma(y_2-y_3)-p[0,0,0,0](y).
\end{align*}
 \item If $\alpha\ge 1$ and $\beta=0$, we have
\begin{align*}
  J_2^G(y; \alpha,\beta)&=p^2\,[0,0,-1,-1](y)\,\sigma(\det y)+p^4\,[1,1,0,0](y)
            +p^3\,[0,1,0,0](y)\\
&\quad +(p-1)\,[-1,0,-1,0](y)+(-p+1)\,[0,0,-1,0](y)\\
&\quad +(p-1)\,[0,0,-1,-1](y)
+[-1,0,-1,-1](y)+(-p^2-p)\,[0,0,0,0](y).
\end{align*}
 \item If $\alpha\ge \beta\ge 1$, we have
\begin{align*}
  J_2^G(y; \alpha,\beta)&=
[-1,-1,-1,-1](y)\,(p^2\sigma(\det y)+1)
+
p^4\,[0,0,0,0](y)\,\sigma(p\inv \det y)\\
&\quad
+(p-1)[-1,-1,-1,-1](y)\,\sigma(p\det y)
+(p^6+p^5)\,[1,1,1,1](y)+[-1,-1,-1,-1](y)\\
&\quad
 +(p^3-p^2-p)[0,0,0,0](y).
\end{align*}

\end{enumerate}

\end{lemma}
 \section{Eichler commutation relations at ramified places}
\label{sec:ecr-ramified}

 \subsection{}
\label{subsec:ecr-ramified}

In this section, we consider the ramified case; namely we assume that
 $p$ divides $d_B$. Throughout this section, we often drop the subscript 
$p$. 
Recall that $\cO$ is the maximal order of $B$
and
$\grP$ is the maximal ideal of $\cO$.
We fix a prime element $\Pi$ of $\cO$ so that $\grP=\Pi\cO$.
Define $\ord_{\Pi}\colon B\cross\to\bZ$ such that
$\ord_{\Pi}(\alpha)=n$ for  $\alpha\in \Pi^n\cO\cross$.
We fix a prime element $\pi$ of $\bQ_p$.

Let
\begin{align*}
 H&=\left\{h\in\GL_2(B)\mid \tp\ol{h}\Nmat{0&1}{-1&0}h
=\Nmat{0&1}{-1&0}\right\},\\
 U&=G\,\cap \GL_2(\cO)\\
 G&=\left\{g\in\GL_2(B)\mid \tp\ol{g}\Nmat{0&1}{1&0}g
=\Nmat{0&1}{1&0}\right\},\\
 K&=G\,\cap \Nmat{\cO&\grP\inv}{\grP&\cO}.
\end{align*}

Define subgroups of $H$ by
\begin{align*}
N_H&=
\left\{
n_H(b):=\Nmat{1&b}{0&1}
\mid b\in\bQ_p
\right\},\\
M_H&=
\left\{
d_H(\alpha):=\Nmat{\alpha&0}{0& \ol{\alpha}\inv}
\mid \alpha\in B\cross
\right\}.
\end{align*}
Then $H=N_HM_HU$.

We also define subgroups of $G$ by
\begin{align*}
 N_G &= \left\{n_G(\beta):=
\Nmat{1&\beta}{0&1}
\mid \beta\in B^-\right\},\\
M_G&=\left\{
d_G(\alpha):=\Nmat{\alpha&0}{0&\ol{\alpha}\inv}
\mid \alpha\in B\cross
\right\}.
\end{align*}
Then $G=N_GM_GK$.

The Weil representation $r'$ given in Lemma \ref{lem:r'}
is reformulated as follows: For $\phi\in\mathscr{S}(B^{(1,2)})$
and $X\in B^{(1,2)}$,
\begin{align*}
 r'(1,g)\phi(X)&=\phi(Xg)
\qquad(g\in G),\\
 r'(d_H(\alpha),1)\phi(X)&=|N(\alpha)|^2\,\phi(\ol{\alpha}X)
\qquad(\alpha\in B\cross), \\
 r'(n_H(b),1)\phi(x,y)& =
\psi(-b\,\ol{X}Q\,\tp X)\phi(X)
\qquad(b\in \bQ_p).
\end{align*}

 \subsection{}
\label{subsec:hecke-operators-G-ramified} 

In this subsection, we prepare several facts on Hecke algebras.

For $n\in\bZ$, we put
\begin{align*}
 X_n&=\{\beta\in B^-\mid \ord_{\Pi}(\beta)\ge n\},\\
 X_n^0&=\{\beta\in B^-\mid \ord_{\Pi}(\beta)= n\}=X_n\setminus X_{n+1},\\
\dd_n^G&=d_G(\Pi^n),\\
\dd_n^H&=d_H(\Pi^n).
\end{align*}
Note that
\begin{align*}
 |X_n/\pi X_n|&=p^3,\\
 |X_n/X_{n+1}|& =
\begin{cases}
 p & (n\,\text{ is even}),\\
 p^2 & (n\,\text{ is odd}).
\end{cases}
\end{align*}

Let $\cT_1^G\in\cH(G,K)$ (respectively $\cT_1^H\in\cH(H,U)$) 
the characteristic function of $T_1^G=K\dd_1^GK$
(respectively $T_1^H=U \dd_1^HU$). 

\begin{lemma}
 \label{lem:hecke-algebras-ramified}
We have 
\begin{align*}
 \cH(H,U)&=\bC[\cT_1^H],\\
 \cH(G,K)&=\bC[\cT_1^G].
\end{align*}
\end{lemma}

\begin{proof}
 This follows from \cite{Sa}, Section 5.1, Theorem 1.
\end{proof}

The following facts are easily verified.

\begin{lemma}
 \label{lem:hecke-operators-H-ramified}
We have
\begin{align*}
 T_1^H&=\bigsqcup_{b\in \bZ_p/\pi\bZ_p}
\,n_H(b)\dd_1^HU \,\sqcup\,
\dd_{-1}^HU.
\end{align*}
\end{lemma}

\begin{lemma}
 \label{lem:hecke-operators-G-ramified}
We have
\begin{align*}
 T_1^G&=\bigsqcup_{\beta\in X_{-1}/\pi X_{-1}}
\,n_G(\beta)\dd_1^GK \, \sqcup\,
\bigsqcup_{\beta\in X_{-2}^0/X_{-1}}\,n_G(\beta)K\,\sqcup\,
\dd_{-1}^GK.
\end{align*}
\end{lemma}

Recall that the spherical functions
$\omega_{\chi}\colon H\to\bC$ and 
 $\Omega_{\Xi}\colon G\to\bC$
for $\chi,\Xi\in 
X_{\mathrm{unr}}(B\cross)$
are defined in Section \ref{subsec:Satake-H-ramified} and
Section \ref{subsec:Satake-G-ramified}, respectively.
Then we have
\begin{align*}
 (\omega_{\chi}*\varphi)(h)&:=\int_H\omega_{\chi}(hx\inv)\varphi(x)dx
=\lambda_{\chi}(\varphi)(h)\qquad (h\in H,\varphi\in\cH(H,U)),\\
 (\Omega_{\Xi}*\Phi)(g)&:=\int_G\Omega_{\Xi}(gy\inv)\Phi(y)dy
=\Lambda_{\Xi}(\Phi)(g)\qquad (g\in G,\Phi\in\cH(G,K))
\end{align*}
with $\lambda_{\chi}\in \mathrm{Hom}_{\bC\text{-alg}}(\cH(H,U),\bC)$
and
$\Lambda_{\Xi}\in \mathrm{Hom}_{\bC\text{-alg}}(\cH(G,K),\bC)$
(cf. \cite{Sa}, Section 5.4).
Using Lemma \ref{lem:hecke-operators-H-ramified} and
\ref{lem:hecke-operators-G-ramified}, we get the following.

\begin{lemma}
 \label{lem:eigenvalues-spherical-f-ramified}
We have
\begin{align*}
 \lambda_{\chi}(\cT_1^H)&=p^{1/2}(\chi(\Pi)+\chi\inv(\Pi)),\\
\Lambda_{\Xi}(\cT_1^G)&=p^{3/2}(\Xi(\Pi)+\Xi\inv(\Pi))+p-1.  
\end{align*}
\end{lemma}

 \subsection{}
\label{subsec:construction-upsilon-ramified}

Let $\phi_0\in W'$ be the characteristic function of $(\cO,\grP\inv)$ and
set
\begin{align*}
 J^H(X)&= \int_{T_1^H}\,r'(h,1)\phi_0(X)dh,\\
 J^G(X)&=\int_{T_1^G}\,r'(1,g)\phi_0(X)dg
\end{align*}
for $X\in B^{(1,2)}$.
We now state the main results of this section, whose proof
is given in Section \ref{subsec:proof-ecr-ramified}.

\begin{proposition}
 \label{prop:ecr-ramified}
We have
\begin{equation}
 \label{eq:ecr-ramified}
J^G(X)=pJ_H(X)+(p-1)\phi_0(X)\qquad(X\in B^{(1,2)}).
\end{equation}
\end{proposition}

\begin{corollary}
 \label{cor:ecr-ramified}
The $\bC$-algebra homomorphism $\upsilon\colon
\cH(G,K)\,\to\,\cH(H,U)$ given by
\[
 \upsilon(\cT_1^G)=p\,\cT_1^H+(p-1)\cT_0^H
\qquad(X\in B^{(1,2)})
\]
satisfies (\ref{eq:Eichler-commutation-relations})
(and hence (\ref{eq:Hecke-eigenvalues})),
where $\cT_0^H$ denotes the characteristic function of $U$.
\end{corollary}

We now prove the following result assuming Corollary 
\ref{cor:ecr-ramified}.

\begin{proposition}
 \label{prop:functoriality-ramified}
We have
\begin{equation}
 \label{eq:functoriality-ramified}
L_p(s,\cL(f),\mathrm{std})=\zeta_p(s)L_p(s,f,\mathrm{std}).
\end{equation}
\end{proposition}

\begin{proof}
 We take $\chi\in X_{\mathrm{unr}}(B\cross)$
such that $\omega_p^f=\omega_{\chi}$ (see
Lemma \ref{lem:SatakeParameter-H-ramified})
and $\Xi=\in X_{\mathrm{unr}}(B\cross)$
such that $\Omega_p^{\cL(f)}=\Omega_{\Xi}$ (see
Lemma \ref{lem:SatakeParameter-G-ramified}).
An argument similar to the proof of Proposition 
\ref{prop:functoriality-unramified} in Section
\ref{subsec:proof-L-unramified} shows 
\[
\Xi(\Pi)+\Xi\inv(\Pi)= \chi(\Pi)+\chi\inv(\Pi)
\]
and hence
\[
 (1-\Xi(\Pi)p^{-s})\inv\,(1-\Xi\inv(\Pi)p^{-s})\inv
= (1-\chi(\Pi)p^{-s})\inv\,(1-\chi\inv(\Pi)p^{-s})\inv.
\]
The proposition  immediately follows from this.
\end{proof}

 \subsection{}
\label{subsec:proof-ecr-ramified}

The proof of Proposition \ref{prop:ecr-ramified} is reduced to the 
following explicit formulas for $J^G(X)$ and $J^H(X)$,
whose proofs are given in Section
\ref{sec:ecr-ramified-details}.

We denote by $\sigma$ the characteristic function of $\cO$.

\begin{lemma}
 \label{lem:explicit-formula-JH-ramified}
For $x,y\in B$, we have
\begin{align*}
 & J^H(x,y)\\
&=
\begin{cases}
 0 & (\ord_{\Pi}(x)\le -2),\\
 p\inv \delta(\tr(\ol{x}y)\in\bZ_p)\,\sigma(\pi y) & (\ord_{\Pi}(x)=0,-1),\\
 p^2\,\sigma(y)+p\inv \sigma(\pi y) & (\ord_{\Pi}x\ge 1).
\end{cases}
\end{align*}
\end{lemma}

\begin{lemma}
 \label{lem:explicit-formula-JG-ramified}
For $x,y\in B$, we have
\begin{align*}
 & J^G(x,y)\\
&=
\begin{cases}
 0 & (\ord_{\Pi}(x)\le -2),\\
 \delta(\tr(\ol{x}y)\in\bZ_p)\,\sigma(\pi y) & (\ord_{\Pi}(x)=-1),\\
 \delta(\tr(\ol{x}y)\in\bZ_p)\,\sigma(\pi y)+(p-1)\sigma(\Pi y) & (\ord_{\Pi}(x)=0),\\
 p^3\,\sigma(y)+(p-1)\sigma(\Pi y)+\sigma(\pi y) & (\ord_{\Pi}x\ge 1).
\end{cases}
\end{align*}
\end{lemma}

 \section{The proof of Eichler commutation relations at unramified places}
\label{sec:ecr-unramified-details}

In this section, we give a detailed account of the 
proof of Eichler commutation relations at a finite place $p$ of $\bQ$
with $p\nmid d_B$. We freely use the notation of Section 
\ref{sec:ecr-unramified}.

 \subsection{}

First we recall that the Weil representation $r'$ 
of $H\times G$ on $\mathscr{S}(\bQ_p^{(2,4)})$
is given by
\begin{align*}
 r'(1,g)\phi(x,y)&=\phi((x,y)g)
\qquad(g\in G),\\
 r'(d_H(A),1)\phi(x,y)&=|\det A|^2\,\phi(\tp Ax,\tp Ay)
\qquad(A\in\GL_2(\bQ_p)), \\
 r'(n_H(b),1)\phi(x,y)& =
\psi(-b\,\Tr(\tp xJy))\phi(x,y)
\qquad(b\in \bQ_p).
\end{align*}

We also recall that
\begin{align*}
 \Lambda_i&=\{(b,c,d)\in(\bZ_p/p\bZ_p)^3\mid \mathrm{rank}_{
\bF_p}\Nmat{b&c}{c&d}=i\}\qquad(i=1,2)
\end{align*}
and $|\Lambda_1|=p^2-1,\,|\Lambda_2|=p^3-p^2$.

 \subsection{}

In this subsection, we prepare several notations and formulas for the
calculation in the next subsections.

We write $\sigma_{m,n}$ for the characteristic function of $\bZ_p^{(m,n)}$
and often write $\sigma$ if there is no fear of confusion.
Let $\tau$
denote the characteristic function of $\bZ_p\cross$.
Then $\tau(x)=\sigma(x)-\sigma(p\inv x)$.
We write $[i_1,i_2,i_3,i_4]$
for $[i_1,i_2,i_3,i_4](y)=\prod_{k=1}^4\,\sigma(p^{-i_k}y_k)
\;\;\left(y=\Nmat{y_1&y_2}{y_3&y_4}\in\bQ_p^{(2,2)}\right)$ 
to simplify the notation.
We abbreviate the notations as follows:
\begin{align*}
 \sum_a & \quad \text{ for }\sum_{a\in\bZ_p/p\bZ_p},\\
 \sum_a {}^{'} & \quad \text{ for }\sum_{a\in\bZ_p\cross/p\bZ_p},\\
 \sum_a {}^2 & \quad \text{ for }\sum_{a\in\bZ_p/p^2\bZ_p}.
\end{align*}

The following two lemmas are standard and we omit its proofs.

\begin{lemma}
 \label{lem:formula-1}
For $x\in\bQ_p$, we have
\begin{align*}
 \suma{a}\sigma(x+p\inv a)&=\sigma(px),\\
 \sumd{a}\sigma(x+p\inv a)&=\tau(x)=\sigma(px)-\sigma(x),\\
\sumt{a}\sigma(x+p^{-2} a)&=\sigma(p^2 x).
\end{align*}
\end{lemma}

\begin{lemma}
 \label{lem:formula-2}
For $x,y\in\bQ_p$, we have
\begin{align*}
 \suma{a}\sigma(x+p\inv a)\,\sigma(y+p\inv a)&=
   \sigma(px)\,\sigma(py)\,\sigma(x-y),\\
 \sumt{a}\sigma(x+p^{-2} a)\,\sigma(y+p^{-2} a)&=
   \sigma(p^2x)\,\sigma(p^2y)\,\sigma(x-y),\\ 
\sumt{a} \sigma(p\inv (x+a))\,\sigma(p\inv (y+p^{-1}a))&=
   \sigma(x)\,\sigma(py)\,\sigma(p\inv x-y).
\end{align*}
\end{lemma}

\begin{lemma}
 \label{lem:formula-3}
For $x,y\in\bQ_p$, we have
\begin{align*}
 \sigma(y)\,\suma{a}\sigma(p\inv (x+ay))&
=p\,\sigma(p\inv x)\,\sigma(p\inv y)+\sigma(x)\,\sigma(y)
-\sigma(x)\,\sigma(p\inv y).
\end{align*}
\end{lemma}

\begin{proof}
By Lemma \ref{lem:formula-1}, we have
 \begin{align*}
  &  \sigma(y)\,\suma{a}\sigma(p\inv (x+ay))\\
  &= \tau(y)\,\suma{a}\sigma(p\inv (x+ay))+
   \sigma(p\inv y)\,\suma{a}\sigma(p\inv (x+ay))\\
&=\tau(y)\,\suma{a}\sigma(p\inv (xy\inv +a))+
   p\,\sigma(p\inv x)\,\sigma(p\inv y)\\
  &= \tau(y)\,\sigma(xy\inv)+
   p\,\sigma(p\inv x)\,\sigma(p\inv y)\\
    &= \tau(y)\,\sigma(x)+
   p\,\sigma(p\inv x)\,\sigma(p\inv y)\\
  &= p\,\sigma(p\inv x)\,\sigma(p\inv y)+\sigma(x)\,\sigma(y)
-\sigma(x)\,\sigma(p\inv y).
 \end{align*}

\end{proof}

\begin{lemma}
 \label{lem:formula-4}
For $x,y\in\bQ_p$, we have
\begin{align*}
\sigma(x)\,\sigma(y)\,\sigma(p\inv xy)& =
\sigma(x)\,\sigma(p\inv y)+\sigma(p\inv x)\,\sigma(y)
-\sigma(p\inv x)\,\sigma(p\inv y),\\
\sigma(px)\,\sigma(py)\,\sigma(pxy)& =
\sigma(px)\,\sigma(y)+\sigma(x)\,\sigma(py)
-\sigma(x)\,\sigma(y).
\end{align*}
\end{lemma}

\begin{proof}
Using $\sigma(x)=\tau(x)+\sigma(p\inv x)$, we have
 \begin{align*}
  \sigma(x)\,\sigma(y)\,\sigma(p\inv xy)& =
\tau(x)\,\sigma(y)\,\sigma(p\inv xy)+\sigma(p\inv x)\,\sigma(y)\,\sigma(p\inv xy)\\
  &=\tau(x)\,\sigma(p\inv y) +\sigma(p\inv x)\,\sigma(y)\\
  & =\sigma(x)\,\sigma(p\inv y)-\sigma(p\inv x)\,\sigma(p\inv y)
+\sigma(p\inv x)\,\sigma(y).
 \end{align*}
\end{proof}

For $y=\left(
\begin{smallmatrix}
 y_1&y_2\\y_3&y_4
\end{smallmatrix}
\right)\in\bQ_p^{(2,2)}
$, define
\begin{align*}
 A(y)&=\sigma(y_2)\,\sigma(y_4)\,
\suma{a}\,\sigma(p\inv(y_1+ay_2))\,\sigma(p\inv(y_3+ay_4)),\\
 A'(y)&=A(py)\\
&=\sigma(py_2)\,\sigma(py_4)\,
\suma{a}\,\sigma(y_1+ay_2)\,\sigma(y_3+ay_4),\\
B(y)&=
\sigma(py_2)\,\sigma(py_4)\,
\sumt{a}\,\sigma(p\inv(y_1+ay_2))\,\sigma(p\inv(y_3+ay_4)).
\end{align*}

\begin{lemma}
 \label{lem:formula-5}
For $y\in\bQ_p^{(2,2)}$, we have
\begin{align*}
 A(y)&=[0,0,0,0]\,\sigma(p\inv\,\det y)+p\,[1,1,1,1]-[0,1,0,1],\\
 A'(y)&=[-1,-1,-1,-1]\,\sigma(p\,\det y)+p\,[0,0,0,0]-[-1,0,-1,0],\\
 B(y)&=[-1,-1,-1,-1]\,\sigma(\det y)+p[0,0,0,0]\,\sigma(p\inv\det y) 
       -[-1,0,-1,0]\,\sigma(\det y)\\
&\qquad +p^2\,[1,1,1,1]-p\,[0,1,0,1].
\end{align*}
\end{lemma}

\begin{proof}
 We have
\[
 A(y)=A_1(y)+A_2(y),
\]
where
\begin{align*}
 A_1(y)&=\tau(y_2)\,\sigma(y_4)\,
\suma{a}\,\sigma(p\inv(y_1+ay_2))\,\sigma(p\inv(y_3+ay_4)),\\
 A_2(y)& =
\sigma(p\inv y_2)\,\sigma(y_4)\,
\suma{a}\,\sigma(p\inv(y_1+ay_2))\,\sigma(p\inv(y_3+ay_4)).
\end{align*}
Then
\begin{align*}
 A_1(y)&=\sigma(y_1)\,\tau(y_2)\,\sigma(y_3)\,\sigma(y_4)\,
         \sigma(p\inv(y_3-y_1y_2\inv y_4))\\
 & =\sigma(y_1)\,\tau(y_2)\,\sigma(y_3)\,\sigma(y_4)\,
         \sigma(p\inv(y_1 y_4-y_2y_3))\\
 &= [0,0,0,0]\,\sigma(p\inv \det y)
-[0,1,0,0]\,\sigma(p\inv y_1y_4)\\
 & =[0,0,0,0]\,\sigma(p\inv \det y)
-[0,1,0,0]\,\{\sigma(y_1)\,\sigma(p\inv y_4)+
\sigma(p\inv y_1)\,\sigma(y_4)-\sigma(p\inv y_1)\,\sigma(p\inv y_4)
\}\\
 &= [0,0,0,0]\,\sigma(p\inv \det y)
-[0,1,0,1]-[1,1,0,0]+[1,1,0,1].
\end{align*}
By Lemma \ref{lem:formula-3}, we have
\begin{align*}
 A_2(y)&=\sigma(p\inv y_1)\,\sigma(p\inv y_2)\,\sigma(y_4)\,
          \suma{a}\,\sigma(p\inv(y_3+ay_4))\\
 &= \sigma(p\inv y_1)\,\sigma(p\inv y_2)\,\sigma(y_4)\,
\{
p\,\sigma(p\inv y_3)\,\sigma(p\inv y_4)+\sigma(y_3)\,\sigma(y_4)
-\sigma(y_3)\,\sigma(p\inv y_4)
\}\\
 &=p[1,1,1,1]+[1,1,0,0]-[1,1,0,1]. 
\end{align*}
Thus
\begin{align*}
 A(y)&
=[0,0,0,0]\,\sigma(p\inv\,\det y)+p\,[1,1,1,1]-[0,1,0,1]
\end{align*}
and
\begin{align*}
 A'(y)&=A(py)\\
 &= [-1,-1,-1,-1]\,\sigma(p\,\det y)+p\,[0,0,0,0]-[-1,0,-1,0].
\end{align*}
Using the above formulas for $A(y)$ and $A'(y)$, we have
\begin{align*}
 B(y)&=\sigma(py_2)\,\sigma(py_4)
\,\suma{a'}\,\suma{a''}\,\sigma(p\inv(y_1+a'y_2+a'' py_2))
                       \,\sigma(p\inv(y_3+a'y_4+a'' py_4))\\
 &=\suma{a'} \,A\Nmat{y_1+a'y_2& py_2}{y_3+a'y_4&py_4}\\
 &=\suma{a'}\,\sigma(y_1+a'y_2)\,\sigma(py_2)\,\sigma(y_3+a'y_4)\,
              \sigma(py_4)\,\sigma(\det y)\\
 &\quad +p\,\suma{a'}\,\sigma(p\inv(y_1+a'y_2))\,\sigma(y_2)
       \,\sigma(p\inv(y_3+a'y_4))\,
              \sigma(y_4)\\
 &\quad - \suma{a'}\,\sigma(y_1+a'y_2)\,\sigma(y_2)\,\sigma(y_3+a'y_4)\,
              \sigma(y_4)\\
 &=A'(y) \,\sigma(\det y)+p\,A(y)-p\,[0,0,0,0]\\
 &=\{[-1,-1,-1,-1] +p[0,0,0,0]-[-1,0,-1,0]\}\,\sigma(\det y)\\
&\quad   +p\,\{[0,0,0,0]\,\sigma(p\inv \det y)+p\,[1,1,1,1]-p\,[0,1,0,1]\}
   -p\,[0,0,0,0]\\
 &=[-1,-1,-1,-1]\,\sigma(\det y)+p[0,0,0,0]\,\sigma(p\inv\det y) 
       -[-1,0,-1,0]\,\sigma(\det y)\\
&\qquad +p^2\,[1,1,1,1]-p\,[0,1,0,1]. 
\end{align*}
\end{proof}

\begin{lemma}
 \label{lem:formula-6}
For $y\in\bQ_p^{(2,2)}$, we have
\begin{align*}
& \sum_{(b,c,d)\in\Lambda_1}
\sigma(y_1+p\inv b)\,\sigma(y_2+p\inv c)\,
\sigma(y_3+p\inv c)\,\sigma(y_4+p\inv d)\\
&\qquad =[-1,-1,-1,-1]\,\sigma(y_2-y_3)\,\sigma(p\,\det y)-[0,0,0,0],\\
& \sum_{(b,c,d)\in\Lambda_2}\sigma(y_1+p\inv b)\,\sigma(y_2+p\inv c)\,
\sigma(y_3+p\inv c)\,\sigma(y_4+p\inv d)\\
&\qquad =[-1,-1,-1,-1]\,\sigma(y_2-y_3)\,\left\{1-\sigma(p\,\det y)\right\}.
\end{align*}
\end{lemma}

\begin{proof}
 Observe that
\begin{align*}
 \Lambda_1&= \{(b,c,d)\in(\bZ_p/p\bZ_p)^3\mid (b,c,d)\ne (0,0,0),\,
bd=c^2\}\\
 &= \{(b,0,0)\in(\bZ_p/p\bZ_p)^3\mid b\ne 0\}
\,\cup\,\{(d\inv c^2,c,d)\in(\bZ_p/p\bZ_p)^3\mid d\ne 0\},\\
 \Lambda_2&= (\bZ_p/p\bZ_p)^3 \setminus (\Lambda_1\,\cup \{(0,0,0)\}).
   \end{align*}
We have
\begin{align*}
 \sum_{(b,c,d)\in\Lambda_1}
\sigma(y_1+p\inv b)\,\sigma(y_2+p\inv c)\,
\sigma(y_3+p\inv c)\,\sigma(y_4+p\inv d)
&=I_1(y)+I_2(y),
\end{align*}
where
\begin{align*}
 I_1(y)&=\sumd{b}\,\sigma(y_1+p\inv b)
\,\sigma(y_2)\,\sigma(y_3)\,\sigma(y_4),\\
 I_2(y)& =\suma{c}\,\sumd{d}\,\sigma(y_1+p\inv d\inv c^2)
          \,\sigma(y_2+p\inv c)\,\sigma(y_3+p\inv c)\,\sigma(y_4+p\inv d). 
\end{align*}
By Lemma \ref{lem:formula-1}, we have
\begin{align*}
 I_1(y)&=\{\sigma(py_1)-\sigma(y_1)\}\,\sigma(y_2)\,\sigma(y_3)\,\sigma(y_4)\\
 &=[-1,0,0,0] -[0,0,0,0].
\end{align*}
We have
\begin{align*}
 I_2(y)&=
\{\sigma(py_4)-\sigma(y_4)\}\,
        \suma{c}\,\sigma(y_1+p\inv(-py_4)\inv c^2)\,
        \sigma(y_2+p\inv c)\,\sigma(y_3+p\inv c)\\
 &= \{\sigma(py_4)-\sigma(y_4)\}\,
        \suma{c}\,\sigma(py_1y_4-p\inv c^2)\,
        \sigma(y_2+p\inv c)\,\sigma(y_3+p\inv c)\\
 &=\sigma(py_1)\, \sigma(py_2)\, \sigma(py_3)\, \sigma(y_2-y_3)\, 
   \{\sigma(py_4)-\sigma(y_4)\}\,\sigma(py_1y_4-py_2^2)\\
 &= \sigma(py_1)\, \sigma(py_2)\, \sigma(py_3)\, \sigma(y_2-y_3)\, 
   \{\sigma(py_4)-\sigma(y_4)\}\,\sigma(p\det y)\\
 &=[-1,-1,-1,-1]\,\sigma(y_2-y_3)\,\sigma(p \det y) 
-\sigma(py_1)\, \sigma(py_2)\, \sigma(py_3)\, \sigma(y_2-y_3)\, 
\sigma(py_2y_3)\,\sigma(y_4).
\end{align*}
Since
\begin{align*}
 & \sigma(py_2)\, \sigma(py_3)\, \sigma(y_2-y_3)\, \sigma(py_2y_3)\\
 &=\{\sigma(py_2)\,\sigma(y_3)+\sigma(y_2)\,\sigma(py_3)-
\sigma(y_2)\,\sigma(y_3)
\} \, \sigma(y_2-y_3)\\
 &= \sigma(y_2)\,\sigma(y_3),
\end{align*}
we have
\begin{align*}
  I_2(y)&=[-1,-1,-1,-1]\,\sigma(y_2-y_3)\,\sigma(p \det y) -
[-1,0,0,0].
\end{align*}
Thus
\begin{align*}
 & \sum_{(b,c,d)\in\Lambda_1}
\sigma(y_1+p\inv b)\,\sigma(y_2+p\inv c)\,
\sigma(y_3+p\inv c)\,\sigma(y_4+p\inv d)\\
 &= [-1,-1,-1,-1]\,\sigma(y_2-y_3)\,\sigma(p \det y) -[0,0,0,0].
\end{align*}
Next we have
\begin{align*}
 & \sum_{(b,c,d)\in\Lambda_2}
\sigma(y_1+p\inv b)\,\sigma(y_2+p\inv c)\,
\sigma(y_3+p\inv c)\,\sigma(y_4+p\inv d)\\
&=  \sum_{(b,c,d)\in (\bZ_p/p\bZ_p)^3}
\sigma(y_1+p\inv b)\,\sigma(y_2+p\inv c)\,
\sigma(y_3+p\inv c)\,\sigma(y_4+p\inv d)\\
 & \quad -\sum_{(b,c,d)\in\Lambda_1}
\sigma(y_1+p\inv b)\,\sigma(y_2+p\inv c)\,
\sigma(y_3+p\inv c)\,\sigma(y_4+p\inv d)
-[0,0,0,0]
\\
&=\sigma(py_1)\,\sigma(py_2)\,\sigma(py_3)\,\sigma(y_2-y_3)\,\sigma(py_4)
-\{
[-1,-1,-1,-1]\,\sigma(y_2-y_3)\,\sigma(p \det y) -[0,0,0,0]
\}
-[0,0,0,0]\\
 & =[-1,-1,-1,-1]\sigma(y_2-y_3)
-
[-1,-1,-1,-1]\,\sigma(y_2-y_3)\,\sigma(p \det y).
\end{align*}

\end{proof}

\begin{lemma}
 \label{lem:formula-7}
For $x,y\in\bQ_p$, we have
\begin{align*}
 \sum_{(b,c,d)\in\Lambda_1}\,\sigma(x+p\inv c)\,\sigma(y+p\inv d)
&=
(p-1)\,\sigma(x)\,\sigma(y)+\sigma(px)\,\sigma(py)-
\sigma(px)\,\sigma(y),\\
 \sum_{(b,c,d)\in\Lambda_2}\,\sigma(x+p\inv c)\,\sigma(y+p\inv d)
&=
(p-1)\,\sigma(px)\,\sigma(py)+\sigma(px)\,\sigma(y)-
p\,\sigma(x)\,\sigma(y).
\end{align*}
\end{lemma}

\begin{proof}
 We have
\begin{align*}
&  \sum_{(b,c,d)\in\Lambda_1}\,\sigma(x+p\inv c)\,\sigma(y+p\inv d)\\
 &=(p-1)\,\sigma(x)\,\sigma(y) +
\suma{c}\,\sumd{d}\,\sigma(x+p\inv c)\,\sigma(y+p\inv d) \\
 & =(p-1)\,\sigma(x)\,\sigma(y) +\sigma(px)\,\sigma(py)-
\sigma(px)\,\sigma(y)
\end{align*}
and
\begin{align*}
 &  \sum_{(b,c,d)\in\Lambda_2}\,\sigma(x+p\inv c)\,\sigma(y+p\inv d)\\
&=\suma{b}\,\suma{c}\,\suma{d}\,\sigma(x+p\inv c)\,\sigma(y+p\inv d)
-\sum_{(b,c,d)\in\Lambda_1}\,\sigma(x+p\inv c)\,\sigma(y+p\inv d)
-\sigma(x)\,\sigma(y)\\
 &=p\,\sigma(px)\,\sigma(py)-
\{
(p-1)\,\sigma(x)\,\sigma(y)+\sigma(px)\,\sigma(py)-
\sigma(px)\,\sigma(y)
\} -\sigma(x)\,\sigma(y)\\
&=(p-1)\,\sigma(px)\,\sigma(py)+\sigma(px)\,\sigma(y)-
p\,\sigma(x)\,\sigma(y).
\end{align*}
\end{proof}

 \subsection{}

In this subsection, we let 
\[
 y=\Nmat{y_1&y_2}{y_3&y_4}\in\bQ^{(2,2)}_p
\]
and put
\[
 \xi=p^{\alpha}y_3-p^{\beta}y_2.
\]
The following four lemmas are proved by a direct calculation.

\begin{lemma}
 We have
\[
 J_1^H(y:\alpha,\beta)=\sum_{i=1}^4\,J_1^H[i](y:\alpha,\beta),
\]
where
\begin{align*}
 J_1^H[1](y:\alpha,\beta)&=p^{-2}\sum_{b\in\bZ_p/p\bZ_p}
  \,\psi(b\xi)\,\sum_{a\in\bZ_p/p\bZ_p}
\, \sigma\Nmat{p^{\alpha+1}&0}{p^{\alpha}a&p^{\beta}}
   \sigma\left(\Nmat{p&0}{a&1}y\right),\\
 J_1^H[2](y:\alpha,\beta)&=p^{-2}\,
\sum_{b\in\bZ_p/p\bZ_p}\,\psi(b\xi)
\, \sigma\Nmat{p^{\alpha}&0}{0&p^{\beta+1}}
   \sigma\left(\Nmat{1&0}{0&p}y\right),\\
 J_1^H[3](y:\alpha,\beta)&=p^2\,
\sum_{a\in\bZ_p/p\bZ_p}
\, \sigma\Nmat{p^{\alpha}&0}{p^{\alpha-1}a&p^{\beta-1}}
   \sigma\left(\Nmat{1&0}{p\inv a&p\inv}y\right),\\
 J_1^H[4](y:\alpha,\beta)&=p^2\,
\sigma\Nmat{p^{\alpha-1}&0}{0&p^{\beta}}
   \sigma\left(\Nmat{p\inv &0}{0&1}y\right).
\end{align*}
\end{lemma}

\begin{lemma}
 We have
\[
 J_2^H(y:\alpha,\beta)=\sum_{i=1}^6\,J_2^H[i](y:\alpha,\beta),
\]
where
\begin{align*}
 J_2^H[1](y:\alpha,\beta)&=
\sum_{a\in \bZ_p/p^2\bZ_p}\,
\sigma\Nmat{p^{\alpha+1}&0}{p^{\alpha-1}a&p^{\beta-1}}
   \sigma\left(\Nmat{p &0}{p\inv a&p\inv}y\right)
,\\
 J_2^H[2](y:\alpha,\beta)&=\sigma\Nmat{p^{\alpha-1}&0}{0&p^{\beta+1}}
   \sigma\left(\Nmat{p\inv &0}{0&p}y\right),\\
 J_2^H[3](y:\alpha,\beta)&=p^{-4}\,\sum_{b\in\bZ_p/p^2\bZ_p}\,
    \psi(-b\xi)\,
   \sigma\Nmat{p^{\alpha+1}&0}{0&p^{\beta+1}}
   \sigma\left(\Nmat{p &0}{0&p}y\right)
,\\
J_2^H[4](y:\alpha,\beta)&=p^4\,
\sigma\Nmat{p^{\alpha-1}&0}{0&p^{\beta-1}}
   \sigma\left(\Nmat{p\inv &0}{0&p\inv}y\right)
,\\
J_2^H[5](y:\alpha,\beta)&=
\sum_{a\in\bZ_p\cross/p\bZ_p}\,
\sigma\Nmat{p^{\alpha}&0}{p^{\alpha-1}a&p^{\beta}}
   \sigma\left(\Nmat{1 &0}{p\inv a&1}y\right)
,\\
J_2^H[6](y:\alpha,\beta)&=\sum_{b\in\bZ_p\cross/p\bZ_p}\,
\psi(-p\inv b \xi)\,\sigma\Nmat{p^{\alpha}&0}{0&p^{\beta}}\,\sigma(y).
\end{align*}
\end{lemma}

\begin{lemma}
 We have
\[
 J_1^G(y:\alpha,\beta)
=\sum_{i=1}^5\,J_1^G[i](y:\alpha,\beta),
\]
where
\begin{align*}
  J_1^G[1](y:\alpha,\beta)&=
\sum_{\substack{a,c\in\bZ_p/p\bZ_p  \\ b\in\bZ_p/p^2\bZ_p}}\,
\sigma\Nmat{p^{\alpha+1}&p^{\alpha}a}{0&p^{\beta}}\,
\sigma\Nmat{p\inv\{y_1+p^{\alpha}b-a(y_2+p^{\alpha}c)\} & y_2+p^{\alpha}c}
           {p\inv\{y_3+p^{\beta}c-ay_4\} & y_4}
,\\
J_1^G[2](y:\alpha,\beta)&=
\sum_{\substack{c\in\bZ_p/p\bZ_p  \\ d\in\bZ_p/p^2\bZ_p}}\,
\sigma\Nmat{p^{\alpha}& 0}{0&p^{\beta+1}}\,
\sigma\Nmat{y_1 & p\inv(y_2+p^{\alpha}c)}
           {y_3+p^{\beta}c & p\inv(y_4+p^{\beta}d)}
,\\
J_1^G[3](y:\alpha,\beta)&=
\sum_{(b,c,d)\in\Lambda_1}\,
\sigma\Nmat{p^{\alpha}& 0}{0&p^{\beta}}\,
\sigma\Nmat{y_1+p^{\alpha-1}b & y_2+p^{\alpha-1}c}
           {y_3+p^{\beta-1}c & y_4+p^{\beta-1}d}
,\\
J_1^G[4](y:\alpha,\beta)&=
\sum_{a\in\bZ_p/p\bZ_p}\,
\sigma\Nmat{p^{\alpha}& p^{\alpha-1}a}{0&p^{\beta-1}}\,
\sigma\Nmat{y_1-ay_2 & py_2}
           {y_3-ay_4 & py_4}
,\\
J_1^G[5](y:\alpha,\beta)&=
\sigma\Nmat{p^{\alpha-1}& p^{\alpha}}{0&p^{\beta}}\,
\sigma\Nmat{py_1 & y_2}
           {py_3 & y_4}.
\end{align*}

\end{lemma}

\begin{lemma}
We put
\[
 x=\Nmat{p^{\alpha}&0}{0&p^{\beta}}.
\]
 We have
\[
 J_2^G(y:\alpha,\beta)
=\sum_{i=1}^{10}\,J_2^G[i](y:\alpha,\beta),
\]
where
\begin{align*}
 J_2^G[1](y:\alpha,\beta)&=
\sum_{b,c,d\in\bZ_p/p^2\bZ_p}
\sigma\!\Nmat{p^{\alpha+1}&0}{0&p^{\beta+1}}\,
\sigma\!\Nmat{p\inv(y_1+p^{\alpha}b) & p\inv(y_2+p^{\alpha}c)}
 {p\inv(y_3+p^{\beta}c) & p\inv(y_4+p^{\beta}d)}
,\\
 J_2^G[2](y:\alpha,\beta)&=
\sum_{a,b\in\bZ_p/p^2\bZ_p}\,
\sigma\!\Nmat{p^{\alpha+1} & p^{\alpha-1}a}{0 & p^{\beta-1}}\,
\sigma\!\Nmat{p\inv(y_1-ay_2+p^{\alpha}b) & py_3}
{p\inv(y_3-ay_4) & py_4}
,\\
 J_2^G[3](y:\alpha,\beta)&=
\sum_{d\in\bZ_p/p^2\bZ_p}\,
\sigma\!\Nmat{p^{\alpha-1}&0}{0&p^{\beta+1}}\,
\sigma\!\Nmat{py_1 & p\inv y_2}{py_3 & p\inv(y_4+p^{\beta}d)}
,\\
 J_2^G[4](y:\alpha,\beta)&=
\sigma\!\Nmat{p^{\alpha-1}&0}{0&p^{\beta-1}}\,
\sigma(py)
,\\
 J_2^G[5](y:\alpha,\beta)&=
\sum_{a,m\in\bZ_p/p\bZ_p}\,\sum_{d\in\bZ_p\cross/p\bZ_p}\,
\sum_{l\in\bZ_p/p^2\bZ_p}
\,\sigma\!\Nmat{p^{\alpha+1}&p^{\alpha}a}{0&p^{\beta}}\\
&\qquad
\sigma\!
\Nmat{p\inv(y_1-ay_2)+p^{\alpha-1}l & y_2+p^{\alpha}(m+p\inv ad)}
{p\inv(y_3-ay_4)+p^{\beta-1}m & y_4+p^{\beta-1}d}
,\\
 J_2^G[6](y:\alpha,\beta)&=
\sum_{b\in\bZ_p\cross/p\bZ_p}\,\sum_{c\in\bZ_p/p\bZ_p}\,
\sum_{d\in\bZ_p/p^2\bZ_p}
\,\sigma\!\Nmat{p^{\alpha}&0}{0&p^{\beta+1}}\,
\sigma\!\Nmat{y_1+p^{\alpha-1}b & p\inv(y_2+p^{\alpha}c)}
{y_3+p^{\beta}c & p\inv(y_4+p^{\beta}d)}
,\\
 J_2^G[7](y:\alpha,\beta)&=
\sum_{a\in\bZ_p/p\bZ_p}\,\sum_{b\in\bZ_p\cross/p\bZ_p}\,
\,\sigma\!\Nmat{p^{\alpha} & p^{\alpha-1}a}{0& p^{\beta-1}}
\sigma\!\Nmat{y_1-ay_2+p^{\alpha-1}b & py_2}{y_3-ay_4 & py_4}
,\\
 J_2^G[8](y:\alpha,\beta)&=
\sum_{d\in\bZ_p\cross/p\bZ_p}\,
\,\sigma\!\Nmat{p^{\alpha-1}&0}{0&p^{\beta}}\,
\sigma\!\Nmat{py_1 & y_2}{py_3 & y_4+p^{\beta-1}d}
,\\
 J_2^G[9](y:\alpha,\beta)&=
\sum_{a\in\bZ_p\cross/p\bZ_p}\,\sum_{c,m\in\bZ_p/p\bZ_p}\,
\,\sigma\!\Nmat{p^{\alpha}&p^{\alpha-1}a}{0&p^{\beta}}
\sigma\!\Nmat{y_1-p\inv ay_2+p^{\alpha-1}m & y_2+p^{\alpha-1}c}
{y_3-p\inv ay_4+p^{\beta-1}c & y_4}
,\\
 J_2^G[10](y:\alpha,\beta)&=
\sum_{(b,c,d)\in\Lambda_2}\,\sigma\!\Nmat{p^{\alpha}&0}{0&p^{\beta}}\,
\sigma\!\Nmat{y_1+p^{\alpha-1}b & y_2+p^{\alpha-1}c}
{y_3+p^{\beta-1}c & y_4+p^{\beta-1}d}.
\end{align*}
\end{lemma}

 \subsection{The proof of Lemma \ref{lem:J1H}}

In this subsection, we let $\alpha\ge\beta$ and write $\cJ$ and $\cJ[i]\;\;(1\le i\le 4)$
for $J_1^H(y;\alpha,\beta)$ and $J_1^H[i](y;\alpha,\beta)$,
respectively.

\begin{enumerate}
 \item Suppose that $\beta<-1$. Then $\cJ=0$.
 \item Suppose that  $\alpha=\beta=-1$. Then $\cJ=0$.
 \item Suppose that  $\alpha=0$ and $\beta=-1$.
We have $\cJ[1]=\cJ[3]=\cJ[4]=0$ and
\begin{align*}
 \cJ[2]&=p^{-2}\,\suma{b}\psi(b(y_3-p\inv y_2))\,[0,0,-1,-1]\\
 &=p\inv\,\sigma(p\inv y_2-y_3) \,[0,0,-1,-1].
\end{align*}
Thus
\[
 \cJ=p\inv\,\sigma(p\inv y_2-y_3) \,[0,0,-1,-1].
\]
 \item Suppose that  $\alpha\ge 1$ and $\beta=-1$.
We have
$\cJ[1]=\cJ[3]=\cJ[4]=0$ and
\begin{align*}
 \cJ[2]&=p^{-2}\,\suma{b}\psi(b(p^{\alpha}y_3-p\inv y_2))\,[0,0,-1,-1]\\
 &=p^{-2}\,\suma{b}\psi(b\cdot p\inv y_2) \,[0,0,-1,-1]\\
 &=p\inv\,[0,1,-1,-1]. 
\end{align*}
Thus
\[
 \cJ=p\inv\,[0,1,-1,-1]. 
\]
 \item Suppose that  $\alpha=\beta=0$. We have $\cJ[3]=\cJ[4]=0$.
By Lemma \ref{lem:formula-5}, we have
\begin{align*}
 \cJ[1]&=p^{-2}\,\suma{b}\psi(b(y_3-y_2))\,
\suma{a}\,\sigma\!\Nmat{py_1&py_2}{ay_1+y_3&ay_2+y_4}\\
 &=p\inv\,\sigma(y_2-y_3) 
     A'\Nmat{y_4&y_2}{y_3&y_1}\\
&=p\inv\,\sigma(y_2-y_3) \,
\left\{
   [-1,-1,-1,-1]\sigma(p\det y)+p[0,0,0,0]-[0,0,-1,-1]
\right\}\\
&=p\inv\,[-1,-1,-1,-1]\sigma(y_2-y_3) \sigma(p\det y)+[0,0,0,0]
-p\inv [0,0,0,-1].
\end{align*}
Next we have
\begin{align*}
 \cJ[2]&=p^{-2}\,\suma{b}\psi(b(y_3-y_2))\,\sigma\!
\Nmat{y_1&y_2}{py_3&py_4}\\
 & =p\inv\,\sigma(y_2-y_3)\,[0,0,-1,-1]\\
&=p\inv\,[0,0,0,-1].
\end{align*}
Thus
\[
 \cJ=p\inv\,[-1,-1,-1,-1]\sigma(y_2-y_3) \sigma(p\det y)+[0,0,0,0].
\]
 \item If $\alpha\ge 1$ and $\beta=0$, we have $\cJ[3]=0$.
By Lemma \ref{lem:formula-5}, we have
\begin{align*}
 \cJ[1]&=p^{-2}\,\suma{b}\psi(b(p^{\alpha}y_3-y_2))\,
\suma{a}\,\sigma\!\Nmat{py_1&py_2}{ay_1+y_3&ay_2+y_4}\\
&=p\inv\,\sigma(y_2-p^{\alpha}y_3) 
     A'\Nmat{y_4&y_2}{y_3&y_1}\\
&=p\inv\,\sigma(y_2-p^{\alpha}y_3) \,
\left\{
   [-1,-1,-1,-1]\sigma(p\det y)+p[0,0,0,0]-[0,0,-1,-1]
\right\}\\
&=p\inv\,[-1,0,-1,-1]\, \sigma(p\det y)
+[0,0,0,0]-p\inv [0,0,-1,-1].
\end{align*}
Observe that, by Lemma \ref{lem:formula-4},
\begin{align*}
 [-1,0,-1,-1]\, \sigma(p\det y)&=[-1,0,-1,-1]\, \sigma(py_1y_4)\\
 &=[-1,0,-1,-1]\{\sigma(py_1)\,\sigma(y_4)+\sigma(y_1)\,\sigma(py_4)
-\sigma(y_1)\,\sigma(y_4)
\} \\
 &=[-1,0,-1,0]+[0,0,-1,-1]-[0,0,-1,0]. 
\end{align*}
It follows that
\[
 \cJ[1]=p\inv\,[-1,0,-1,0]-p\inv\,[0,0,-1,0]+[0,0,0,0].
\]
Next we have
\begin{align*}
 \cJ[2]&=p^{-2}\,\suma{b}\,\psi(b(p^{\alpha}y_3-y_2))\,[0,0,-1,-1]\\
 &=p\inv\,\sigma(p^{\alpha}y_3-y_2) \,[0,0,-1,-1]\\
&=p\inv\,[0,0,-1,-1]
\end{align*}
and 
\[
 \cJ[4]=p^2\,[1,1,0,0].
\]
Thus
\[
 \cJ=[0,0,0,0]+p\inv\,[-1,0,-1,0]-p\inv\,[0,0,-1,0]+
p\inv\,[0,0,-1,-1]+p^2[1,1,0,0].
\]
 \item Suppose that $\alpha\ge\beta\ge 1$.
By Lemma \ref{lem:formula-5}, we have
\begin{align*}
 \cJ[1]&=p\inv \,\sigma(p^{\alpha}y_3-p^{\beta}y_2)\,
A'\!\Nmat{y_4&y_2}{y_3&y_1}\\
 & = p\inv \,\sigma(p^{\alpha}y_3-p^{\beta}y_2)\,
\{
[-1,-1,-1,-1]\sigma(p\det y)+p[0,0,0,0]-[0,0,-1,-1]
\}\\
&=p\inv \,[-1,-1,-1,-1]\sigma(p\det y)+[0,0,0,0]-p\inv \,[0,0,-1,-1].
\end{align*}
We have
\begin{align*}
 \cJ[2]&=
p\inv\,\sigma(p^{\alpha}y_3-p^{\beta}y_2)[0,0,-1,-1]\\
 &= p\inv\,[0,0,-1,-1].
\end{align*}
By Lemma \ref{lem:formula-5}, we have
\begin{align*}
 \cJ[3]&=p^2\,\suma{a}\,
\sigma\!\Nmat{y_1&y_2}{p\inv (ay_1+y_3)&p\inv (ay_2+y_4)}\\
&=p^2\,A\!\Nmat{y_4&y_2}{y_3&y_1}\\
&=p^2\,[0,0,0,0]\,\sigma(p\inv\,\det y)
+p^3\,[1,1,1,1]-
p^2\,[1,1,0,0].
\end{align*}
We have
\[
 \cJ[4]=p^2\,[1,1,0,0].
\]
Thus we obtain
\begin{align*}
 \cJ&=p^2\,[0,0,0,0]\,\sigma(p\inv\,\det y)+
p\inv\,[-1,-1,-1,-1]\,\sigma(p\,\det y)+
p^3[1,1,1,1]+[0,0,0,0].
\end{align*}
\end{enumerate}

 \subsection{The proof of Lemma \ref{lem:J2H}}

In this subsection, we let $\alpha\ge\beta$ and write $\cJ$ and $\cJ[i]\;\;(1\le i\le 4)$
for $J_2^H(y;\alpha,\beta)$ and $J_2^H[i](y;\alpha,\beta)$,
respectively.

\begin{enumerate}
 \item Suppose that $\beta<-1$. Then $\cJ=0$.
 \item Suppose that  $\alpha=\beta=-1$. 
Then $\cJ[1]=\cJ[2]=\cJ[4]=\cJ[5]=0$. We have
\begin{align*}
 \cJ[3]&=p^{-4}\,\sumt{b}\,\psi(-p\inv b(y_3-y_2))\,[-1,-1,-1,-1]\\
 &=p^{-2}\,[-1,-1,-1,-1]\,\sigma(p\inv(y_2-y_3)) 
\end{align*}
and hence
\[
 \cJ=p^{-2}\,[-1,-1,-1,-1]\,\sigma(p\inv(y_2-y_3)). 
\]
  \item Suppose that  $\alpha=0$ and $\beta=-1$.
Then $\cJ[1]=\cJ[2]=\cJ[4]=\cJ[5]=\cJ[6]=0$.
We have
\begin{align*}
 \cJ[3]&=p^{-4}\,\sumt{b}\,\psi(-p\inv(y_3-p\inv\,y_2))\,[-1,-1,-1,-1]\\
&=p^{-2}\,\sigma(p\inv y_2-y_3)\,[-1,-1,-1,-1]\\
 & =p^{-2}\,\sigma(p\inv y_2-y_3)\,[-1,0,-1,-1]
\end{align*}
and hence
\[
 \cJ=p^{-2}\,\sigma(p\inv y_2-y_3)\,[-1,0,-1,-1].
\]

  \item Suppose that  $\alpha\ge 1$ and $\beta=-1$.
Then $\cJ[1]=\cJ[4]=\cJ[5]=\cJ[6]=0$.
We have
\begin{align*}
 \cJ[2]&=[1,1,-1,-1]
\end{align*}
and 
\begin{align*}
 \cJ[3]&=p^{-4}\,\sumt{b}\,\psi(-b(p^{\alpha}\,y_3-p\inv\,y_2))\,[-1,-1,-1,-1]\\
&=p^{-2}\,\sigma(p\inv\,y_2-p^{\alpha}y_3)\,[-1,-1,-1,-1]\\
 &= p^{-2}\,[-1,1,-1,-1].
\end{align*}
Thus we have
\[
 \cJ=[1,1,-1,-1]+p^{-2}\,[-1,1,-1,-1].
\]
  \item Suppose that  $\alpha=\beta=0$.
Then $\cJ[1]=\cJ[2]=\cJ[4]=\cJ[5]=0$.
We have
\begin{align*}
 \cJ[3]&=p^{-2}\,\sigma(y_2-y_3)\,[-1,-1,-1,-1],\\
 \cJ[6]&=\sumd{b}\,\psi(-p\inv(y_3-y_2)b)\,\sigma(y)\\
 & =\{p\,\sigma(p\inv(y_2-y_3))-1\}\,[0,0,0,0].
\end{align*}
Thus
\[
 \cJ=p^{-2}\,[-1,-1,-1,-1]\,\sigma(y_2-y_3)+
p\,[0,0,0,0]\,\sigma(p\inv(y_2-y_3))-[0,0,0,0].
\]

  \item Suppose that  $\alpha\ge 1$ and $\beta=0$.
Then $\cJ[1]=\cJ[4]=0$.
By Lemma \ref{lem:formula-5}, we have
\begin{align*}
 \cJ[2]&=[1,1,-1,-1],\\
 \cJ[3]&=p^{-2}\,\sigma(p^{\alpha}y_3-y_2)\,[-1,-1,-1,-1]\\
 &=p^{-2} \,[-1,0,-1,-1],\\
 \cJ[5]&=\sumd{a}\sigma\!\Nmat{y_1&y_2}{y_3+p\inv ay_1&y_4+p\inv a y_2}\\
 & =A\Nmat{py_4&y_2}{py_3&y_1}-\sigma(y)\\
&=[0,0,-1,-1]\,\sigma(\det y)
+p\,[1,1,0,0]-[1,1,-1,-1]-[0,0,0,0],\\
 \cJ[6]&=\{p\,\sigma(p\inv(p^{\alpha}y_3-y_2))-1\}\,\sigma(y)\\
 &=p\,[0,1,0,0] -[0,0,0,0].
\end{align*}
Thus
\begin{align*}
 \cJ&=[0,0,-1,-1]\,\sigma(\det y)+p[1,1,0,0]+p[0,1,0,0] +p^{-2} [-1,0,-1,-1]-2\,[0,0,0,0].
\end{align*}
  \item Suppose that  $\alpha\ge\beta\ge 1$.
By Lemma \ref{lem:formula-5}, we have
\begin{align*}
 \cJ[1]&=\sumt{a}\,
\sigma\!\Nmat{py_1&py_2}{p\inv(y_3+ay_1)&p\inv(y_4+ay_2)}\\
&=B\!\Nmat{y_4&y_2}{y_3&y_1}\\
&=[-1,-1,-1,-1]\,\sigma(\det y)+p\,[0,0,0,0]\,\sigma(p\inv\,\det y)
-[0,0,-1,-1]\,\sigma(\det y)\\
&\quad +p^2\,[1,1,1,1]-p\,[1,1,0,0],\\
 \cJ[2]&=[1,1,-1,-1],\\
 \cJ[3]&=p^{-2}\,\sigma(p^{\beta}y_2-p^{\alpha}y_3)\,\sigma(py)\\
 & =p^{-2}\,[-1,-1,-1,-1],\\
 \cJ[4]&=p^4\,[1,1,1,1],\\
 \cJ[5]&=\sumd{a}\,\sigma\!\Nmat{y_1&y_2}{y_3+p\inv ay_1&y_4+p\inv ay_2}\\
&=A\!\Nmat{py_4&y_2}{py_3&y_1}-\sigma(y)\\
&=[0,0,-1,-1]\,\sigma(\det y)+p\,[1,1,0,0]
-[1,1,-1,-1]-[0,0,0,0],\\
 \cJ[6]&=\sumd{b}\,\psi(-p\inv(p^{\alpha}y_3-p^{\beta}y_2))\,\sigma(y)\\
 &=(p-1)\,[0,0,0,0]. 
\end{align*}
Thus
\begin{align*}
 \cJ&=p\,[0,0,0,0]\,\sigma(p\inv\,\det y)+
[-1,-1,-1,-1]\,\sigma(\det y)+(p^4+p^2)\,[1,1,1,1]+p^{-2}\,[-1,-1,-1,-1]\\
 &\quad +(p-2)\,[0,0,0,0]. 
\end{align*}
\end{enumerate}


 \subsection{The proof of Lemma \ref{lem:J1G}}

In this subsection, we let $\alpha\ge\beta$ and write $\cJ$ and $\cJ[i]\;\;(1\le i\le 4)$
for $J_1^G(y;\alpha,\beta)$ and $J_1^G[i](y;\alpha,\beta)$,
respectively.

\begin{enumerate}
 \item Suppose that $\beta<-1$. Then $\cJ=0$.
 \item Suppose that  $\alpha=\beta=-1$. Then $\cJ=0$.
 \item Suppose that  $\alpha=0$ and $\beta=-1$. 
Then $\cJ[1]=\cJ[3]=\cJ[4]=\cJ[5]=0$.
We have
\begin{align*}
\cJ[2]&=
\suma{c}\,\sumt{d}\,\sigma\!\Nmat{y_1&p\inv(y_2+c)}
{y_3+p\inv c & p\inv(y_4+p\inv d)}\\
&=\sigma(y_1)\,\suma{c}\,\sigma(p\inv y_2+p\inv c)
   \,\sigma(y_3+p\inv c)\,
  \sumt{d}\,\sigma(p^{-2} y_4+p^{-2} d)\\
 &=\sigma(y_1)\,\sigma(y_2)\,\sigma(py_3)\,\sigma(p\inv y_2-y_3)
\,\sigma(py_4)
\end{align*}
by Lemma \ref{lem:formula-1}
 and Lemma \ref{lem:formula-2}.
Thus
\[
 \cJ=[0,0,-1,-1]\,\sigma(p\inv y_2-y_3).
\]
 \item Suppose that  $\alpha\ge 1$ and $\beta=-1$.
Then $\cJ[1]=\cJ[3]=\cJ[4]=\cJ[5]=0$.
We have
\begin{align*}
 \cJ[2]&=
\suma{c}\,\sumt{d}\,\sigma\!\Nmat{y_1&p\inv y_2}
{y_3+p\inv c & p\inv(y_4+p\inv d)}\\
&=\sigma(y_1)\,\sigma(p\inv y_2)\,\suma{c}\,
   \,\sigma(y_3+p\inv c)\,
  \sumt{d}\,\sigma(p^{-2} y_4+p^{-2} d)\\
 &=\sigma(y_1)\,\sigma(y_2)\,\sigma(py_3)\,
\,\sigma(py_4)
\end{align*}
by Lemma \ref{lem:formula-1}.
Thus
\[
 \cJ=[0,1,-1,-1].
\]
 \item Suppose that  $\alpha=\beta=0$.
Then $\cJ[4]=\cJ[5]=0$.
By Lemma \ref{lem:formula-1} and
Lemma \ref{lem:formula-6} and Lemma \ref{lem:formula-4},
we have
\begin{align*}
 \cJ[1]&=\suma{a,\,c}\,\sumt{b}\,
\sigma\!\Nmat{p\inv(y_1+b-ay_2-ac) & y_2+c}{p\inv(y_3+c-ay_4) & y_4}\\
&=\sigma(y_2)\,\sigma(y_4)\,\suma{a,\,c}\,\sigma(p\inv(y_3-ay_4+c))\,
\sumt{b}\,\sigma(p\inv(y_1-ay_2-ac)+p\inv b)\\
&=p\,\sigma(y_1)\,\sigma(y_2)\,\sigma(y_4)\,
  \suma{a,\,c}\,\sigma(p\inv(y_3-ay_4)+p\inv c)\\
&=p\,\sigma(y_1)\,\sigma(y_2)\,\sigma(y_4)\,
\suma{a}\,\sigma(y_3-ay_4)\\
&=p^2\,[0,0,0,0],\\
 \cJ[2]&=
\sigma(y_1)\,\sigma(y_3)\,
\suma{c}\,\sigma(p\inv(y_2+c))
\,\sumt{d}\,\sigma(p\inv(y_4+d))\\
 &= \sigma(y_1)\,\sigma(y_3)\,\sigma(y_2)\,p\,\sigma(y_4)\\
&=p\,[0,0,0,0],\\
 \cJ[3]&=
\sum_{(b,c,d)\in\Lambda_1}\,
\sigma(y_1+p\inv b)\,\sigma(y_2+p\inv c)\,\sigma(y_3+p\inv c)
\,\sigma(y_4+p\inv d)\\
 &=[-1,-1,-1,-1]\,\sigma(y_2-y_3)\,\sigma(p\,\det y)-[0,0,0,0].
\end{align*}

Thus
\begin{align*}
 \cJ&=[-1,-1,-1,-1]\,\sigma(y_2-y_3)\,\sigma(p\,\det y)+
(p^2+p-1)\,[0,0,0,0] .
\end{align*}
 \item Suppose that  $\alpha\ge 1$ and $\beta=0$.
Then $\cJ[4]=0$.
By Lemma \ref{lem:formula-1}, Lemma \ref{lem:formula-3} 
and Lemma \ref{lem:formula-7},
We have
\begin{align*}
 \cJ[1]&=
p^2\,\sigma(y_2)\,\sigma(y_4)\,
\suma{a}\,\sigma(p\inv(y_1-ay_2))\,
\suma{c}\,\sigma(p\inv(y_3-ay_4)+p\inv c)\\
 &=p^2\, \sigma(y_2)\,\sigma(y_3)\,\sigma(y_4)\,
\suma{a}\,\sigma(p\inv(y_1-ay_2))\\
 &=p^2\,\sigma(y_3)\,\sigma(y_4)\,\left\{
  p\,\sigma(p\inv y_1)\,\sigma(p\inv y_2)+\sigma(y_1)\,\sigma(y_2)
-\sigma(y_1)\,\sigma(p\inv y_2)
\right\}\\
 &=p^3\,[1,1,0,0]+p^2\,[0,0,0,0]-p^2\,[0,1,0,0],\\
 \cJ[2]&=
\suma{c}\,\sumt{d}\,\sigma(y_1)\,
\sigma(p\inv y_2)\,\sigma(y_3)\,\sigma(p\inv(y_4+d))\\
 &=p^2\,[0,1,0,0],\\
 \cJ[3]&=\sum_{(b,c,d)\in\Lambda_1}\,
\sigma(y_1)\,\sigma(y_2)\,\sigma(y_3+p\inv c)\,\sigma(y_4+p\inv d)\\
 &= \sigma(y_1)\,\sigma(y_2)\,
\left\{
(p-1)\,\sigma(y_3)\,\sigma(y_4)+\sigma(py_3)\,\sigma(py_4)-
\sigma(py_3)\,\sigma(y_4)
\right\}\\
 &=(p-1)\,[0,0,0,0]+[0,0,-1,-1]-[0,0,-1,0], \\
 \cJ[5]&=[-1,0,-1,0].
\end{align*}
Thus
\begin{align*}
 \cJ&=
p^3\,[1,1,0,0]+(p^2+p-1)\,[0,0,0,0]+[0,0,-1,-1]-[0,0,-1,0]
+[-1,0,-1,0].
\end{align*}
 \item Suppose that  $\alpha\ge \beta\ge 1$.
By Lemma \ref{lem:formula-5},
we have
\begin{align*}
 \cJ[1]&=p^3\,\sigma(y_2)\,\sigma(y_4)\,
  \suma{a}\,\sigma(p\inv(y_1-ay_2))\,\sigma(p\inv(y_3-ay_4))\\
 &=p^3\,A(y) \\
 &=p^3\,[0,0,0,0]\,\sigma(p\inv\,\det y)
+p^4\,[1,1,1,1]-p^3\,[0,1,0,1],\\
 \cJ[2]&=p^3\,[0,1,0,1],\\
 \cJ[3]&=|\Lambda_1|\,[0,0,0,0]\\
 &=(p^2-1) \,[0,0,0,0],\\
 \cJ[4]&=\sigma(py_2)\,\sigma(py_4)\,\suma{a}\,\sigma(y_1-ay_2)
\,\sigma(y_3-ay_4)\\
 &=A'(y) \\
 &=[-1,-1,-1,-1]\,\sigma(p\,\det y)+p\,[0,0,0,0]-[-1,0,-1,0],\\
 \cJ[5]&=[-1,0,-1,0].
\end{align*}
Thus
\begin{align*}
 \cJ&=p^3\,[0,0,0,0]\,\sigma(p\inv\,\det y)+
[-1,-1,-1,-1]\,\sigma(p\,\det y)+p^4\,[1,1,1,1]+(p^2+p-1)\,[0,0,0,0].
\end{align*}
\end{enumerate}

 \subsection{The proof of Lemma \ref{lem:J2G}}%

In this subsection, we let $\alpha\ge\beta$ and write $\cJ$ and $\cJ[i]\;\;(1\le i\le 4)$
for $J_2^G(y;\alpha,\beta)$ and $J_2^G[i](y;\alpha,\beta)$,
respectively.

\begin{enumerate}
 \item Suppose that $\beta<-1$. Then $\cJ=0$.
 \item Suppose that  $\alpha=\beta=-1$. 
Then $\cJ[i]=0$ for $2\le i\le 10$. By Lemma \ref{lem:formula-1}
and Lemma \ref{lem:formula-2},
we have
\begin{align*}
 \cJ[1]&=\sumt{b}\,\sigma(p\inv y_1+p^{-2}b)\,
   \sumt{d}\,\sigma(p\inv y_4+p^{-2}d)\,
  \sumt{c}\,\sigma(p\inv y_2+p^{-2}c)\,\sigma(p\inv y_3+p^{-2}c)\\
 &=\sigma(py_1) \,\sigma(py_4) \,\sigma(py_2) \,\sigma(py_3) \,
   \sigma(p\inv(y_2-y_3)).
\end{align*}
Thus
\[
 \cJ=[-1,-1,-1,-1]\,\sigma(p\inv(y_2-y_3)).
\]
 \item Suppose that  $\alpha=0$ and $\beta=-1$. 
Then $\cJ[i]=0$ unless $i=1,6$.
By Lemma \ref{lem:formula-1} and Lemma \ref{lem:formula-2}, we have
\begin{align*}
 \cJ[1]&=\sumt{b}\,\sigma(p\inv y_1+p\inv b)\,
\sumt{d}\,\sigma(p\inv y_4+p^{-2} d)\,
\sumt{c}\,\sigma(p\inv(y_2+c))\,\sigma(p\inv(y_3+p\inv c))\\
 &=p\,\sigma(y_1)\, \sigma(py_4)\, \sigma(y_2)\, \sigma(py_3)\,
  \sigma(p\inv y_2-y_3) \\
 &=[0,0,-1,-1]\,\sigma(p\inv y_2-y_3),\\
 \cJ[6]&=\sumd{b}\,\sigma(y_1+p\inv b)\,
         \sumt{d}\,\sigma(p\inv y_4+p^{-2}d)\,
         \suma{c}\,\sigma(p\inv y_2+p^{-1}c)\,\sigma(y_3+p^{-1}c)\\
 &=\{\sigma(py_1)-\sigma(y_1)\}\,\sigma(py_4)\,\sigma(y_1)\,\sigma(py_3)\,
   \sigma(p\inv y_2-y_3) \\
 &=\{[-1,0,-1,-1]-[0,0,-1,-1]\}\,\sigma(p\inv y_2-y_3). 
\end{align*}
Thus
\[
 \cJ=(p-1)\,[0,0,-1,-1]\,\sigma(p\inv y_2-y_3) +
   [-1,0,-1,-1]\,\sigma(p\inv y_2-y_3).
\]
 \item Suppose that  $\alpha\ge 1$ and $\beta=-1$. 
Then $\cJ[i]=0$ unless $i=1,3,6$.
By Lemma \ref{lem:formula-1} and Lemma \ref{lem:formula-2}, we have
\begin{align*}
 \cJ[1]&=p^2\,\sigma(p\inv y_1)\,\sigma(p\inv y_2)\,
\sumt{c}\,\sigma(p\inv y_3+p^{-2}c)\,
\sumt{d}\,\sigma(p\inv y_4+p^{-2}d)\\
 &=p^2 \,\sigma(p\inv y_1)\,\sigma(p\inv y_2)\,\sigma(py_3)\,\sigma(py_4)\\
 &=p^2\,[1,1,-1,-1], \\
 \cJ[3]&=\sigma(py_1)\,\sigma(p\inv y_2)\,\sigma(py_3)\,
\sumt{d}\,\sigma(p\inv y_4+p^{-2 d})\\
 &= \sigma(py_1)\,\sigma(p\inv y_2)\,\sigma(py_3)\,\sigma(py_4)\\
 &=[-1,1,-1,-1], \\
 \cJ[6]&=(p-1)\,\sigma(y_1)\,\sigma(p\inv y_2)\,
\suma{c}\,\sigma(y_3+p\inv c)\,
\sumt{d}\,\sigma(p\inv y_4+p^{-2} d)\\
 &= (p-1)\,\sigma(y_1)\,\sigma(p\inv y_2)\,\sigma(py_3)\,\sigma(py_4)\\
 &=(p-1)\,[0,1,-1,-1]. 
\end{align*}
Thus
\[
 \cJ=(p-1)\,[0,1,-1,-1]+p^2\,[1,1,-1,-1]+[-1,1,-1,-1].
\]
 \item Suppose that  $\alpha=\beta=0$. 
Then $\cJ[i]=0$ unless $i=1,5,6,10$.
By Lemma \ref{lem:formula-1} and Lemma \ref{lem:formula-2}, 
we have
\begin{align*}
 \cJ[1]&=
\sumt{b}\,\sigma(p\inv (y_1+b))\,
\sumt{c}\,\sigma(p\inv (y_2+c))\,\sigma(p\inv (y_3+c))\,
\sumt{d}\,\sigma(p\inv (y_4+d))\\
 &=p^3\,\sigma(y_1)\,\sigma(y_2)\,\sigma(y_3)\,
   \sigma(p\inv(y_2-y_3))\,\sigma(y_4) \\
 &=p^3\,[0,0,0,0]\,\sigma(p\inv(y_2-y_3))
\end{align*}
and
\begin{align*}
 \cJ[5]&=\suma{a}\,\sumd{d}\,\sigma(y_2+p\inv ad)\,\sigma(y_4+p\inv d)\,
  \sumt{l}\,\sigma(p\inv(y_1-ay_2)+p\inv l)\\
&\qquad  \suma{m}\,\sigma(p\inv(y_3-ay_4)+p\inv m)\\
 & =p\,\suma{a}\,\sigma(y_1-ay_2)\,\sigma(y_3-ay_4)\,
   \sumd{d}\,\sigma(y_2+p\inv ad)\,\sigma(y_4+p\inv d). 
\end{align*}
By Lemma \ref{lem:formula-5}, we have
\begin{align*}
 &\sumd{d}\,\sigma(y_2+p\inv ad)\,\sigma(y_4+p\inv d)\\
 &=-\sigma(y_2)\,\sigma(y_4)+A'\!\Nmat{y_2&p\inv a}{y_4&p\inv}\\
 &=-\sigma(y_2)\,\sigma(y_4)+\sigma(py_2)\,\sigma(py_4)\,\sigma(y_2-ay_4)
\end{align*}
and hence
\begin{align*}
 \cJ[5]&=-p\,\suma{a}\,\sigma(y_2)\,\sigma(y_4)\,
          \sigma(y_1-ay_2)\,\sigma(y_3-ay_4)\\
&\qquad +
         p\,\sigma(py_2)\,\sigma(py_4)\,
         \suma{a}\,\sigma(y_1-ay_2)\,\sigma(y_3-ay_4)\,\sigma(y_2-ay_4)\\
 &=-p^2\,[0,0,0,0]+p\,\sigma(py_2)\,\sigma(py_4) \sigma(y_2-y_3)\,
     \suma{a}\,\sigma(y_1+ay_2)\,\sigma(y_3+ay_4)\\
&=-p^2\,[0,0,0,0]+p\,\sigma(py_2)\,\sigma(py_4) \sigma(y_2-y_3)\,
     A'(y)\\
 &= -p^2\,[0,0,0,0]+p\,\sigma(py_2)\,\sigma(py_4) \sigma(y_2-y_3)\\
&\qquad \,\{
[-1,-1,-1,-1]\sigma(p\det y)+p[0,0,0,0]-[-1,0,-1,0]
\}\\
 &=p\,[-1,-1,-1,-1]\,\sigma(y_2-y_3)\,\sigma(p\det y)-p\,[-1,0,0,0]. 
\end{align*}
By Lemma \ref{lem:formula-1} and 
Lemma \ref{lem:formula-6}, we have
\begin{align*}
 \cJ[6]&=\sigma(y_3)\,\sumd{b}\,\sigma(y_1+p\inv b)\,
         \suma{c}\,\sigma(p\inv(y_2+c))\,\sumt{d}\,
         \sigma(p\inv(y_4+d))\\
 &=p\,\{\sigma(py_1)-\sigma(y_1)\}\,\sigma(y_2)\,\sigma(y_3)\, \sigma(y_4)\\
 &=p\,[-1,0,0,0]-p\,[0,0,0,0],\\
 \cJ[10]&=\sum_{(b,c,d)\in \Lambda_2}\,
     \sigma(y_1+p\inv b)\,\sigma(y_2+p\inv c)\,
     \sigma(y_3+p\inv c)\,\sigma(y_4+p\inv d)\\
 &=[-1,-1,-1,-1]\,\sigma(y_2-y_3) \{1-\sigma(p\det y)\}.
\end{align*}
Thus
\begin{align*}
 \cJ&=p^3\,[0,0,0,0]\,\sigma(p\inv(y_2-y_3))
 +(p-1)\,[-1,-1,-1,-1]\,\sigma(y_2-y_3) \{1-\sigma(p\det y)\}\\
&\quad +[-1,-1,-1,-1]\,\sigma(y_2-y_3) -p[0,0,0,0].
\end{align*}
 \item Suppose that  $\alpha\ge 1$ and $\beta=0$. 
Then $\cJ[2]=\cJ[4]=\cJ[7]=0$.
By Lemma \ref{lem:formula-1}, 
Lemma \ref{lem:formula-2} and Lemma \ref{lem:formula-6}, we have
\begin{align*}
 \cJ[1]&=p^2\,\sigma(p\inv y_1)\,\sigma(p\inv y_2)\,
    \sumt{c}\,\sigma(p\inv (y_3+c))\,
    \sumt{d}\,\sigma(p\inv (y_4+d))\\
 &=p^4\,\sigma(p\inv y_1)\,\sigma(p\inv y_2)\,\sigma(y_3)\,\sigma(y_4) \\
 &= p^4\,[1,1,0,0],\\
 \cJ[3]&=\sigma(p y_1)\,\sigma(p\inv y_2)\,\sigma(py_3)\,
  \sumd{d}\,\sigma(p\inv(y_4+d))\\
 &=p\, \sigma(p y_1)\,\sigma(p\inv y_2)\,\sigma(py_3)\,\sigma(y_4)\\
 &=p\,[-1,1,-1,0],
\end{align*}
\begin{align*}
 \cJ[5]&=p^2\,\sigma(y_2)\,
\suma{a}\,\sigma(p\inv(y_1-ay_2))\,
\suma{m}\,\sigma(p\inv(y_3-ay_4)+p\inv m)\,
\sumd{d}\,\sigma(y_4+p\inv d)\\
 &= p^2\,\sigma(y_2)\,
\{\sigma(py_4)-\sigma(y_4)\}\,
\suma{a}\,\sigma(p\inv(y_1-ay_2))\,\sigma(y_3-ay_4)\\
 &= p^2\,\sigma(y_2)\,
\{\sigma(py_4)-\sigma(y_4)\}\,A'\!\Nmat{p\inv y_1 & p\inv y_2}{y_3&y_4}\\
 & =p^2\,\sigma(y_2)\,
\{\sigma(py_4)-\sigma(y_4)\}\,
\{
[0,0,-1,-1]\,\sigma(\det y)+p\,[1,1,0,0]-[0,1,-1,0]
\}\\
 &=p^2\{[0,0,-1,-1]\,\sigma(\det y)-[0,0,-1,0]\,\sigma(\det y)\} \\
 &=p^2\{[0,0,-1,-1]\,\sigma(\det y)-[0,0,-1,0]\,\sigma(y_2y_3)\} \\
 &= p^2\,[0,0,-1,-1]\,\sigma(\det y)-[0,0,-1,0]\,
\{\sigma(y_2)\,\sigma(y_3)+\sigma(p\inv y_2)\,\sigma(py_3)-
\sigma(p\inv y_2)\,\sigma(y_3)
\}\\
 &=p^2\,\{
[0,0,-1,-1]\,\sigma(\det y)-[0,0,0,0]-[0,1,-1,0]+[0,1,0,0]
\},\\
 \cJ[6]&=p(p-1)\,\sigma(y_1)\,\sigma(p\inv y_2)\,\sigma(y_3)\,
\sumt{d}\,\sigma(p\inv(y_4+d))\\
 &=(p^3-p^2) \,\sigma(y_1)\,\sigma(p\inv y_2)\,\sigma(y_3)\,\sigma(y_4)\\
 &=(p^3-p^2) \,[0,1,0,0],\\
 \cJ[8]&=\sigma(py_1)\,\sigma(y_2)\,\sigma(py_3)\,
\sumd{d}\,\sigma(y_4+p\inv d)\\
 &= \sigma(py_1)\,\sigma(y_2)\,\sigma(py_3)\,\{\sigma(py_4)-\sigma(y_4)\}\\
 & =[-1,0,-1,-1]-[-1,0,-1,0],\\
 \cJ[9]&=p\,\sigma(y_2)\,\sigma(y_4)\,\sumd{a}\,\sigma(y_1-p\inv ay_2)\,
\suma{c}\,\sigma(y_3-p\inv ay_4+p\inv c)\\
 &=p\,\sigma(y_2)\,\sigma(y_4)\,
\sumd{a}\,\sigma(y_1-p\inv ay_2)
   \,\sigma(py_3-ay_4)  \\
 & =p\,\sigma(y_2)\,\sigma(py_3)\,\sigma(y_4)\,
\sumd{a}\,\sigma(y_1-p\inv ay_2)\\
 &= p\,\sigma(y_2)\,\sigma(py_3)\,\sigma(y_4)\,
\{p\,\sigma(y_1)\,\sigma(p\inv y_2)+\sigma(py_1)\,\sigma(y_2)
-\sigma(py_1)\,\sigma(p\inv y_2)-\sigma(y_1)
\}\\
 & =p^2\,[0,1,-1,0]+p\,[-1,0,-1,0]-p\,[-1,1,-1,0]-p\,[0,0,-1,0],\\
 \cJ[10]&=\sigma(y_1)\,\sigma(y_2)
\,\sum_{(b,c,d)\in\Lambda_2}\,\sigma(y_3+p\inv c)\,\sigma(y_4+p\inv d)\\
 & =\sigma(y_1)\,\sigma(y_2)\,
  \{(p-1)\,\sigma(py_3)\,\sigma(py_4)+\sigma(py_3)\,\sigma(y_4)-
p\,\sigma(y_3)\,\sigma(y_4)
\}\\
 &=(p-1)\,[0,0,-1,-1]+[0,0,-1,0]-p\,[0,0,0,0]. 
\end{align*}
Thus
\begin{align*}
 \cJ&=p^2\,[0,0,-1,-1]\,\sigma(\det y)+p^4\,[1,1,0,0]
            +p^3\,[0,1,0,0]\\
&\quad +(p-1)\,[-1,0,-1,0]+(-p+1)\,[0,0,-1,0]\\
&\quad +(p-1)\,[0,0,-1,-1]
+[-1,0,-1,-1]+(-p^2-p)\,[0,0,0,0].
\end{align*}

 \item Suppose that  $\alpha\ge \beta\ge 1$. 
By Lemma \ref{lem:formula-1},
Lemma \ref{lem:formula-2}, Lemma \ref{lem:formula-5} 
and Lemma \ref{lem:formula-6}, we have
\begin{align*}
 \cJ[1]&=p^6\,[1,1,1,1],\\
 \cJ[2]&=p^2\,\sigma(py_3)\,\sigma(py_4)\,
\sumt{a}\,\sigma(p\inv(y_1-ay_2))\,\sigma(p\inv(y_3-ay_4))\\
 &=p^2\,B(y) \\
 &=p^2\,[-1,-1,-1,-1] \,\sigma(\det y)+
   p^3\,[0,0,0,0] \,\sigma(p\inv \det y)-
   p^2\,[-1,0,-1,0] \,\sigma(\det y)\\
&\quad +   p^4\,[1,1,1,1]-p^3\,[0,1,0,1],\\
 \cJ[3]&=p^2\,[-1,1,-1,1],\\
 \cJ[4]&=[-1,-1,-1,-1],\\
 \cJ[5]&=(p^4-p^3)\,\sigma(y_2)\,\sigma(y_4)\suma{a}
\,\sigma(p\inv(y_1-ay_2))\,\sigma(p\inv(y_3-ay_4))\\
 &=(p^4-p^3) \,A(y)\\
&=(p^4-p^3)\,
[0,0,0,0]\,\sigma(p\inv\det y)
+(p^5-p^4)\,[1,1,1,1]-(p^4-p^3)\,[0,1,0,1],\\
 \cJ[6]&=(p^4-p^3)\,[0,1,0,1],\\
 \cJ[7]&=(p-1)\,\sigma(py_2)\,\sigma(py_4)\,
\sumt{a}\,\sigma(y_1-ay_2)\,\sigma(y_3-ay_4)\\
 & =(p-1)\,A'(y)\\
 & =(p-1)\,[-1,-1,-1,-1]\,\sigma(p\det y)+(p^2-p)\,[0,0,0,0]
-(p-1)\,[-1,0,-1,0],
\end{align*}
\begin{align*}
 \cJ[8]&=(p-1)\,[-1,0,-1,0],\\
 \cJ[9]&=p^2\,\sigma(y_2)\,\sigma(y_4)\,
    \sumd{a}\,\sigma(y_1-p\inv ay_2)\,\sigma(y_3-p\inv ay_4)\\
 &=p^2\,A\!\Nmat{py_1&y_2} {py_3&y_4}-p^2\,[0,0,0,0]\\
 &=p^2\,[-1,0,-1,0] \,\sigma(\det y)+p^3\,[0,1,0,1]-p^2\,[-1,1,-1,1]
-p^2\,[0,0,0,0],\\
 \cJ[10]&=|\Lambda_2|\cdot [0,0,0,0]\\
 & =(p^3-p^2)\,[0,0,0,0].
\end{align*}
Thus
\begin{align*}
 \cJ&=[-1,-1,-1,-1]\,(p^2\sigma(\det y)+1)
+
p^4\,[0,0,0,0]\,\sigma(p\inv \det y)\\
&\quad
+(p-1)[-1,-1,-1,-1]\,\sigma(p\det y)
+(p^6+p^5)\,[1,1,1,1]+[-1,-1,-1,-1]\\
&\quad
 +(p^3-p^2-p)[0,0,0,0].
\end{align*}
\end{enumerate}

 \section{The proof of Eichler commutation relations at ramified places}
\label{sec:ecr-ramified-details}

In this section, we give a detailed account of the 
proof of Eichler commutation relations at a finite place $p$ of $\bQ$
with $p | d_B$. We  freely use the notation of Section 
\ref{sec:ecr-ramified}.

\subsection{}

In this subsection, we prepare several notations and formulas for the
calculation in the next subsection.
Denote by $\sigma$ the characteristic function of $\cO$.
Recall that $\phi_0(x,y)=\sigma(x)\sigma(\Pi y)$. 

\begin{lemma}[\cite{Na4}, Lemma A.2 and Lemma A.3],
 We have
\begin{align*}
 [\cO : \pi\cO]&=p^4,\\
 [\cO : \Pi\cO]&=p^2,\\
 [\cO^- : \pi\cO^-]&=p^3,\\
  [\cO^- : (\Pi\cO)^-]&=p,\\ 
  [(\Pi\cO)^- : \pi\cO^-]&=p^2.
\end{align*}
\end{lemma}

\begin{lemma}[\cite{MN1}, Lemma 8.2]
We have
\begin{align*}
 \sum_{\beta\in X_{-1}/X_0}\,\sigma(x+\beta) &=\sigma(\Pi x),\\
 \sum_{\beta\in X_{-1}/\pi X_{-1}}\,\sigma(\Pi\inv(x+\beta)) &
=\delta(\tr(x)\in p\bZ_p)\,\sigma(\Pi x),\\
 \sum_{\beta\in X_{-2}/X_{-1}}\,\sigma(\Pi(x+\beta)) &
=\delta(\tr(x)\in \bZ_p)\,\sigma(\pi x).
\end{align*}
 
\end{lemma}

 \subsection{The proof of Lemma \ref{lem:explicit-formula-JH-ramified}}

Let $x,y\in B$. 
By Lemma \ref{lem:hecke-operators-H-ramified}, we have
\begin{align*}
  J^H(x,y)&=
    \sum_{b\in \bZ_p/p\bZ_p}\,
r'\left(n_H(b)\dd^H_1,1\right)\,\phi_0(x,y)
+r'\left(\dd_{-1}^H,1\right)\,\phi_0(x,y)\\
 &=p^{-2} \,\sum_{b\in \bZ_p/p\bZ_p}\,\psi(-b\,\tr(\ol{x}y))\,
    \phi_0(\ol{\Pi}x,\ol{\Pi}y)
+p^2\,\phi_0(\ol{\Pi}\inv x,\ol{\Pi}\inv y)\\
 &= p\inv\,\delta(\tr(\ol{x}y)\in\bZ_p)\,\sigma(\Pi x)\,\sigma(\pi y)
+p^2\,\sigma(\Pi\inv x)\,\sigma(y).
\end{align*}
We have $J^H(x,y)=0$ if $\ord_{\Pi}(x)\le -2$.
If  $\ord_{\Pi}(x)=-1 \text{ or }0$, we have
\[
 J^H(x,y)=p\inv\,\delta(\tr(\ol{x}y)\in\bZ_p)\,\sigma(\pi y).
\]
If $\ord_{\Pi}(x)\ge 1$, we have
\[
 J^H(x,y)=p\inv\,\sigma(\pi y)
+p^2\,\sigma(y),
\]
which completes the proof 
of Lemma \ref{lem:explicit-formula-JH-ramified}.

 \subsection{The proof of Lemma \ref{lem:explicit-formula-JG-ramified}}

\begin{lemma}
 For $x,y\in B$, we have
\begin{align*}
 J^G(x,y)&=J_1(x,y)+J_2(x,y)+J_3(x,y),
\end{align*}
where
\begin{align*}
 J_1(x,y)&=\sum_{\beta\in X_{-1}/\pi X_{-1}}\,\sigma(x\Pi)\,
\sigma(y+x\beta),\\
 J_2(x,y)&=\sum_{\beta\in X^0_{-2}/X_{-1}}\,\sigma(x)\,\sigma(\Pi(y+x\beta)),\\
J_3(x,y)&=\sigma(\Pi\inv x)\,\sigma(\pi y).
\end{align*}
\end{lemma}

\begin{proof}
By Lemma \ref{lem:hecke-operators-G-ramified}, we have
\begin{align*}
& J^G(x,y)\\
&=
\sum_{\beta\in X_{-1}/\pi X_{-1}}\,r'(1,n_G(\beta)\dd_1^G)\phi_0(x,y)\\
&\quad +\sum_{\beta\in X^0_{-2}/X_{-1}}\,r'(1,n_G(\beta))\phi_0(x,y)
+r'(1,d^G_{-1})\phi_0(x,y)\\
 &=
\sum_{\beta\in X_{-1}/\pi X_{-1}}\,\phi_0((x,y)n_G(\beta)\dd_1^G) 
+\sum_{\beta\in X^0_{-2}/X_{-1}}\,\phi_0((x,y)n_G(\beta))
+\phi_0((x,y) \,\dd^G_{-1})\\
 & =
\sum_{\beta\in X_{-1}/\pi X_{-1}}\,\phi_0(x\Pi,(y+x\beta)\ol{\Pi}\inv) 
+\sum_{\beta\in X^0_{-2}/X_{-1}}\,\phi_0(x,y+x\beta)
+\phi_0(x\Pi\inv,y\ol{\Pi})\\
&=\sum_{\beta\in X_{-1}/\pi X_{-1}}\,\sigma(x\Pi)\,
\sigma(\Pi (y+x\beta)\ol{\Pi}\inv) )
+\sum_{\beta\in X^0_{-2}/X_{-1}}\,\sigma(x)\,\sigma(\Pi(y+x\beta))
\\ &\quad
+\sigma(x\Pi\inv)\,\sigma(\Pi y\ol{\Pi})\\
&=\sum_{\beta\in X_{-1}/\pi X_{-1}}\,\sigma(x\Pi)\,\sigma(y+x\beta)+
\sum_{\beta\in X^0_{-2}/X_{-1}}\,\sigma(x)\,\sigma(\Pi(y+x\beta))
+\sigma(\Pi\inv x)\,\sigma(\pi y).
\end{align*}
\end{proof}

We now prove Lemma \ref{lem:explicit-formula-JG-ramified}.

It is easy to see that $J^G(x,y)=0$ if $\ord_{\Pi}(x)\le -2$.

Suppose that $\ord_{\Pi}(x)=-1$.
Then $J_2(x,y)=J_3(x,y)=0$ and
\begin{align*}
 J_1(x,y)&=\sum_{\beta\in X_{-1}/\pi X_{-1}}\,\sigma(y+x\beta)\\
 &=\sum_{\beta\in X_{-1}/\pi X_{-1}}\,\sigma(\Pi\inv(x\inv y+\beta))\\
 &=\delta(\tr(x\inv y)\in p\bZ_p)\,\sigma(\Pi x\inv y) \\
 & =\delta(\tr(\ol{x} y)\in \bZ_p)\,\sigma(\pi y).
\end{align*}

Suppose that $\ord_{\Pi}(x)=0$.
Then $J_3(x,y)=0$. We have
\begin{align*}
 J_2(x,y)&=\sum_{\beta\in X_{-1}/\pi X_{-1}}\,\sigma(y+x\beta)\\
 &=\sum_{\beta\in X_{-1}/\pi X_{-1}}\,\sigma(x\inv y+\beta)\\ 
&=\sum_{\beta'\in X_{-1}/\pi X_{0}}\,\sum_{\beta''\in X_{0}/\pi X_{1}}\,
\sigma(x\inv y+\beta'+\beta'')\\ 
&=|X_0/X_1|\,\sum_{\beta'\in X_{-1}/\pi X_{0}}\,
\sigma(x\inv y+\beta')\\
 &=p\,\sigma(\Pi x\inv y) \\
 &=p\,\sigma(\Pi y) 
\end{align*}
and
\begin{align*}
 J_2(x,y)&=
\sum_{\beta\in X^0_{-2}/X_{-1}}\,\sigma(\Pi(x\inv y+\beta))\\
 &=\delta(\tr(x\inv y)\in\bZ_p) \,\sigma(\pi x\inv y)-\sigma(\Pi x\inv y)\\
 & =\delta(\tr(\ol{x} y)\in\bZ_p) \,\sigma(\pi  y)-\sigma(\Pi  y).
\end{align*}

Suppose that $\ord_{\Pi}(x)\ge 1$.
We have
\begin{align*}
 J_1(x,y)&=|X_{-1}/\pi X_{-1}|\,\sigma(y)=p^3\,\sigma(y), \\
 J_2(x,y)&=|X^0_{-}/X_{-1}|\,\sigma(\Pi y)=(p-1)\,\sigma(\Pi y), \\
 J_3(x,y)&=\sigma(\pi y),
\end{align*}
which completes the proof of Lemma \ref{lem:explicit-formula-JG-ramified}.


\bigskip

\begin{flushleft}
Atsushi Murase

Department of Mathematical Science

Faculty of Science

Kyoto Sangyo University

Motoyama, Kamigamo, Kita-ku

Kyoto 603-8555, Japan

E-mail: murase@cc.kyoto-su.ac.jp
 
\end{flushleft}

\medskip

\begin{flushleft}
 Hiro-aki Narita

Department of Mathematics

Faculty of Science and Engineering

Waseda University

3-4-1 Okubo, Shinjuku-ku

Tokyo 169-8555, Japan

E-mail: hnarita@waseda.jp
\end{flushleft}

\end{document}